\documentclass{amsart}

\usepackage[latin1]{inputenc}

\usepackage{amssymb}
\usepackage{amsmath}
\usepackage{mathrsfs}
\usepackage{latexsym}
\usepackage{bbm}

\newtheorem{theorem}{\sc Theorem}[section] 
\newtheorem{lemma}{\sc Lemma}[section]     
\newtheorem{corollary}{\sc Corollary}[section]
\newtheorem{proposition}{\sc Proposition}[section]
\newtheorem{dfn}{\sc Definition}
\newtheorem{remark}{\sc Remark}
\makeatletter

\numberwithin{equation}{section}

\renewcommand\section{\@startsection {section}{1}{\z@}%
                                   {-3.5ex \@plus -1ex \@minus -.2ex}%
                                   {2.3ex \@plus.2ex}%
                                   {\centering\normalfont\bf}}

\def\F{\mathcal F}
\def\N{\mathbb N}

\def\q{\bold q}
\def\C{\mathcal C}

\def\R{\mathcal R}
\def\K{\mathcal K}
\def\a{\bold a}
\def\t{\bold t}
\def\Z{\mathbb Z}

\newcommand\hdim{\dim_{\mathcal H}}

\def\Z{\mathbb Z}

\def\D{\mathcal D}
\usepackage[usenames]{color}

\title[Mass transference principle]
 {Mass transference principle from rectangles to rectangles in Diophantine approximation} 

\author{Baowei Wang and Jun Wu}

\address{School of Mathematics and Statistics, Huazhong University of Science and Technology, Wuhan 430074, P. R. China}

\email{bwei\_wang@hust.edu.cn, jun.wu@mail.hust.edu.cn}


\begin{document}
\maketitle

\begin{abstract}
The limsup sets defined by balls or defined by rectangles appear at the most fundamental level in Diophantine approximation: one follows from Dirichlet's theorem, the other follows from Minkowski's theorem. The metric theory for the former has been well studied, while the theory for latter is rather incomplete, even in classic Diophantine approximation.

This paper aims at setting up a general principle for the Hausdorff theory of limsup sets defined by rectangles. By introducing a notation called {\em ubiquity for rectangles},
a mass transference principle from rectangles to rectangles is presented, i.e., if a limsup set defined by a sequence of big rectangles has a full measure property, then this full measure property can be transferred to the full Hausdorff measure/dimension property for the limsup set defined by shrinking these big
rectangles to smaller ones. Here the full measure property for the bigger limsup set goes with or without the assumption of ubiquity for rectangles.
Together with the landmark work of Beresnevich \& Velani in 2006 where a transference principle from balls to balls is established, a coherent Hausdorff theory for metric Diophantine approximation is built.

Besides completing the dimensional theory in simultaneous Diophantine approximation, linear forms, the dimensional theory for limsup sets defined by rectangles also underpins the dimensional theory in multiplicative Diophantine approximation where unexpected phenomenon occurs and the usually used methods fail to work.\end{abstract}



\tableofcontents

\section{Dirichlet's theorem vs Minkowski's theorem}

Diophantine
approximation concerns how well a real number can
be approximated by rationals. A qualitative answer is provided by the density of rational numbers. Seeking a quantitative
answer leads to the theory of metric Diophantine approximation.

\subsection{Dirichlet's theorem}

Dirichlet's
theorem (1842) is the first result in this aspect, which opens up the extensive study on the metric theory of limsup sets defined by balls.
\begin{theorem}[Dirichlet's theorem] Let $x=(x_1,\cdots, x_d)\in \mathbb{R}^d$ with $d\in \N$. For any $Q>1$, there exists an integer $1\le q\le Q$ such that $$
\|qx_i\|<Q^{-1/d}, \ \ 1\le i\le d,$$ where $\|\cdot\|$ denotes the distance to the integers.
\end{theorem}

As a consequence, one has that \begin{corollary}
For any $x=(x_1,\cdots, x_d)\in \mathbb{R}^d$, there exist infinitely many integers $p_1,\cdots, p_d,q$ such that \begin{equation}\label{f16}
|x_i-p_i/q|<q^{-(1+1/d)}, \ 1\le i\le d.
\end{equation}
\end{corollary}

Dirichlet's theorem is the starting point of the metric Diophantine approximation which leads to a solid study on the distribution of rational numbers or rational vectors. More precisely, one considers the size of the so called $\psi$-well approximable sets, i.e. \begin{equation}\label{e3}
\Big\{x\in \mathbb{R}^d: \max_{1\le i\le d}|x_i-p_i/q|<\psi(q), \ {\text{i.m.}}\ (p_1,\cdots, p_d,q)\in \Z^d\times \N\Big\}.
\end{equation}
where $\psi: \N\to \mathbb{R}^+$ is a positive function and {\em i.m.} denotes {\em infinitely many} for brevity.
Extended further, one can consider the Diophantine approximation for linear forms: \begin{align}\label{e5}
\Big\{x\in \mathbb{R}^{m\times n}: \|q_1x_{i1}+q_2x_{i2}+\cdots+q_nx_{in}\|&<\psi(q), 1\le i\le m\nonumber\\
 & \text{i.m.}\ q=(q_1,\cdots,q_n)\in \Z^n\Big\}.
\end{align}

Essentially the above sets concern the metric theory for limsup set defined by balls. Even for linear forms (\ref{e5}), it considers the limsup set defined by an {\em isotropic} thicken of a sequence of sets.
By a comparison, the sets in (\ref{e3}), (\ref{e5}) can be viewed as the limsup set obtained by
shrinking the sequence of balls in (\ref{f16}) to a sequence of smaller balls.


Since the proof of Dirichlet's theorem uses only the simple pigeonhole principle, one wants to know whether Dirichlet's theorem can be strengthened. This is true which is known as Minkowski's theorem (for convex body) in 1896.

\subsection{Minkowski's theorem}

Minkowski's theorem is a strengthen of Dirichlet's theorem and leads to the study of the metric theory for limsup sets defined by rectangles. \begin{theorem}[Minkowski \cite{Mi,Sch}]
Let $x=(x_1,\cdots,x_d)\in \mathbb{R}^d$. For any non-negative vector $(\hat{a}_1,\cdots,\hat{a}_d)$ with $\hat{a}_1+\cdots+\hat{a}_d=1$, and any $Q\in \N$, there exists an integer $1\le q\le Q$ such that $$
\|qx_i\|<Q^{-\hat{a}_i}, \ 1\le i\le d.
$$
\end{theorem}
Consequently, one has the following (by letting $a_i=1+\hat{a}_i$, $1\le i\le d$). \begin{corollary}
Let $a_i\ge 1$ for $1\le i\le d$ and $a_1+\cdots+a_d=d+1$. For any $x\in \mathbb{R}^d$, there exist infinitely many integers $p_1,\cdots, p_d,q$ such that \begin{equation}\label{f17}
|x_i-p_i/q|<q^{-a_i}, \ 1\le i\le d.
\end{equation}
\end{corollary}

So, as the sets defined after Dirichlet's theorem, one concerns the size of the set \begin{equation}\label{e4}
{W}(\psi_1,\cdots, \psi_d):=\Big\{(x_1,\cdots,x_d)\in \mathbb{R}^d: \|qx_i\|<\psi_i(q), \ 1\le i\le d, \ {\text{i.m.}}\ q\in \N\Big\},
\end{equation}
and more generally, the set \begin{align*}
W_{mn}(\psi_1,\cdots, \psi_m)=\Big\{x\in \mathbb{R}^{m\times n}: \|q_1x_{i1}+q_2&x_{i2}+\cdots+q_nx_{in}\|<\psi_i(q), \\
&\ \ 1\le i\le m,\ \  {\text{i.m.}}\ q=(q_1,\cdots,q_n)\in \Z^n \Big\}.
\end{align*}
These sets are essentially limsup sets defined by a sequence of rectangles or the anisotropic thicken of a sequence of sets.
Similarly, one can also say that from the limsup set in (\ref{f17}) to the limsup set in (\ref{e4}),
one just shrinks a sequence of rectangles to a sequence of smaller rectangles.

Minkowski's theorem provides more profound understandings about the distribution of rational vectors, which works sufficiently well in high dimensional Diophantine approximation compared with Dirichlet's theorem (For example, Minkowski's theorem intervenes as an essential tool in most works about Diophantine approximation on manifolds, see \cite{B1}, \cite{B2} for examples).
So the study on limsup sets defined by rectangles after Minkowski's theorem should (also) play a central role in metric Diophantine approximation as the study on limsup sets defined by balls after Dirichlet's theorem.

\subsection{Multiplicative Diophantine approximation}

The metric theory for limsup sets defined by rectangles is also tightly inherent in the study of the multiplicative Diophantine approximation especially the dimensional theory there.

In the most classic case, multiplicative Diophantine approximation concerns the size of the set  $$
M(t):=\Big\{(x_1,\cdots, x_d)\in [0,1)^d: \|qx_1-\theta_1\| \cdots \|qx_d-\theta_d\|<\psi(q)^{t}, \ {\text{i.m.}}\ q\in \N\Big\}.
$$

If we define, for any positive numbers $(t_1,\cdots, t_d)$, $$
M(t_1,\cdots, t_d)=\Big\{(x_1,\cdots, x_d)\in [0,1)^d: \|qx_i-\theta_i\|<\psi(q)^{t_i}, \ 1\le i\le d, \ {\text{i.m.}}\ q\in \N\Big\}
$$ one can see that for any integer $N$ large,
\begin{align}\label{fff4}\bigcup_{(j_1,\cdots, j_d)\in \N^d_{\ge 0}: j_1+\cdots+j_d=N}M\Big(\frac{j_1t}{N}&,\cdots, \frac{j_dt}{N}\Big)\subset M(t)\nonumber\\
&\subset
\bigcup_{(i_1,\cdots, i_d)\in \N^d_{\ge 0}: j_1+\cdots+j_d=N-d}M\Big(\frac{j_1t}{N},\cdots, \frac{j_dt}{N}\Big).\end{align}
Thus for the Hausdorff dimension, denoted by $\hdim$, of $M(t)$, one will have $$
\hdim M(t)=\sup_{t_1+\cdots+t_d=t, t_i\ge 0, 1\le i\le d}\ \hdim M(t_1,\cdots, t_d)
$$ if there is some continuity about the dimension of the latter set which should be expected.
Thus once the dimension of the limsup set $M(t_1,\cdots, t_d)$ defined by rectangles is given, the dimension of $M(t)$ follows directly.

So one can say that the dimensional theory for limsup set defined by rectangles underpins that of multiplicative Diophantine approximation.

\subsection{Current situation}

In a summary, the limsup set defined by rectangles also appears at the most fundamental level in metric Diophantine approximation. The study on it not only improves Minkowski's theorem, develops the metric theory for limsup sets, it also underpins the metric theory for multiplicative Diophantine approximation. Moreover, in a dynamical system with more than one transformations involved, the Diophantine property of the orbits is merely the analysis of the limsup set defined by rectangles (see Section \ref{s5.1.3} for an example). But the metric theory on limsup sets defined by rectangles is rather incomplete, even in classic Diophantine approximation.

\subsubsection{Known results for limsup sets defined by balls}\

The metric theory on the limsup sets defined by balls has been very extensively studied and many substantial general principles are disclosed. We give a rough history of the progress.\begin{itemize}\item Khintchine \cite{Khint} (1924): \begin{center}Lebesgue measure for $\psi$-approximable set;\end{center}\smallskip
 \item Jarn\'{i}k \cite{Jarn1,Jarn} (1929, 1931), Besicovitch \cite{Besi} (1934):
 \begin{center}Hausdorff dimension/measure theory for $\psi$-approximable set;\end{center}\smallskip
  \item Groshev \cite{Gro} (1938): \begin{center}Lebesgue measure theory for linear forms;\end{center}\smallskip
   \item Baker \& Schmidt \cite{BakS} (1970): \begin{center}Regular system in $\mathbb{R}$;\end{center}\smallskip
    \item Dodson, Rynne \& Vickers \cite{DoRV} (1990): \begin{center}Ubiquitous system in $\mathbb{R}^d$;\end{center}\smallskip
     \item Bovey \& Dodson \cite{BovD1} (1986), Dodson \cite{Dod} (1992): \begin{center}Hausdorff dimension for linear forms;\end{center}\smallskip
       \item Dickinson \& Velani \cite{DickV} (1997): \begin{center}Hausdorff measure theory for linear forms;\end{center}\smallskip
        \item Levesley \cite{Lev} (1998), Bugeaud \cite{Bug} (2004): \begin{center}Inhomogeneous Diophantine approximation;\end{center}\smallskip
        \item Beresnevich, Dickinson \& Velani \cite{BDV06} (2006): \begin{center}Metric theory for limsup sets under ubiquity for balls;\end{center}\smallskip
         \item Beresnevich \& Velani \cite{BV06} (2006): \begin{center}Mass transference principle from balls to balls; \end{center}\smallskip
         \item Allen \& Beresnevich \cite{AllB} (2018): \begin{center}Mass transference principle for linear forms;\end{center}\smallskip
         \item Allen \& Baker \cite{AllBa} (2019): \begin{center}Mass transference principle for general forms;\end{center}\smallskip
         \item Hussain \& Simmons \cite{HusDa} (2019): \begin{center}Another Hausdorff theory for limsup sets of balls;\end{center}\smallskip
\item ......\end{itemize}

\subsubsection{Known results for limsup sets defined by rectangles}\

However, the case is rather different for limsup sets defined by rectangles. There are only two known results.
\begin{itemize}
\item Schmidt \cite{Sch1} (1960), Gallagher \cite{Gall} (1962): \begin{center}Lebesgue measure theory for $W(\psi_1,\cdots, \psi_d)$;\end{center}\smallskip
         \item Rynne \cite{Ryn} (1998), Dickinson \& Rynne \cite{DRyn} (2000): \begin{center}Hausdorff dimension for $W(t_1,\cdots, t_d)$, i.e. $\psi_i(q)=q^{t_i}$. \end{center}
\end{itemize}\medskip
Even the dimension of $W(\psi_1,\cdots, \psi_d)$ for general $\psi_i$ is unknown.

Philosophically, the further improvement of Minkowski's theorem should go along with that of Dirichlet's theorem, however the fact is the Hausdorff theory for the sets considered after Minkowski's theorem lays much behind the theory after Dirichlet's theorem. This results in the incompleteness of some theory even in classical Diophantine approximation, let along general principles for the Hausdorff theory of limsup sets defined by general rectangles. In this paper, we hope to change this unbalance.

\subsubsection{Known results for multiplicative Diophantine approximation}\

In the known cases, the dimensional theory for multiplicative Diophantine approximation is usually obtained by a combination of a covering lemma by Bovey \& Dodson \cite{BD} and a slicing lemma \cite{Fal}: \begin{lemma}[Bovey \& Dodson, Covering lemma, \cite{BD}]\label{l5}
Let $\rho$ be a sufficiently small positive number. For any $s\in (d-1,d)$, the set
$$H=\Big\{(y_1,\ldots,y_d)\in [0,1)^d: y_1\cdots y_d< \rho\Big\}$$
has a covering of $d$-dimensional hypercubes $\mathscr{C}=\{C_k\}_{k\geq 1}$ whose $s$-volume satisfies
$$\sum\limits_{C_i\in \mathscr{C_i}}|C_i|^s\ll\rho^{s-d+1},$$
where the constant implied in $\ll$ is independent of $\rho$.
\end{lemma}
\begin{lemma}[K. Falconer, Slicing lemma, Proposition 7.9 \cite{Fal}]\label{l7} Let $X, Y$ be two metric space and let $E\subset X\times Y$.
If there is a subset $X_o\subset X$ of Hausdorff dimension $s$ such that for any $x\in X_o$, $\hdim \{y\in Y: (x,y)\in E\}\ge t,$ then $
\hdim E\ge s+t.
$
\end{lemma}

For example, for the dimension of $M(t)$ by Bovey \& Dodson \cite{BD} and the dimension of $M(t)$ intersecting a planar curve by Beresnevich \& Velani \cite{BV15}, the above two lemmas work well.
However, as far as some other cases are concerned (see Section \ref{s11}), the above two lemmas, even their generalizations, may not give the exact dimension of the set in question.
So, we have to explore their essential nature about the relation to limsup sets defined by rectangles as given in (\ref{fff4}).
%
%
%

%
\bigskip

\subsection{Notation}
\begin{itemize}\item Dimension function $f$: continuous and increasing with $f(r)\to 0$ as $r\to 0$.
\item a doubling function $g$: $|g(x)|\le |g(2x)|\le \lambda |g(x)|$ for some $\lambda\ge 1$.
\item $\mathcal{H}^f$: $f$-Hausdorff measure;
\item $\mathcal{H}^s$: $s$-dimensional Hausdorff measure, i.e. when $f(x)=x^s$ for $\mathcal{H}^f$;
\item$\hdim$: Hausdorff dimension.
\item $\R$: a general resonant set for which the points in a metric space $X$ will be approximated;
\item $R$ and $\widetilde{R}$: rectangles in $X$;
    \item $\Delta(\R, \rho)$: $\rho$-neighborhood of a resonant set $\R$.
    \item $cB$: for a ball $B=B(x,r)$, $cB$ denotes the ball $B(x, cr)$.
    \item $cR$: for a rectangle $R=\prod_{i=1}^dB(x_i,r_i)$, $cR$ denotes the rectangle $\prod_{i=1}^dB(x_i, cr_i)$;
    \item $\ll$: $a\ll b$ if there is an unspecified constant $c$ such that $|a|\le c|b|$; \item
    $a\asymp b$: $a\ll b$ and $b\ll a$.
      \item $\mathcal{L}(E)$ or $|E|$: Lebesgue measure of a set $E$ if no confusion;
      \item $r_B$: the radius of a ball $B$;
      \item $\sharp A$: the cardinality of a finite set $A$;
      \item For the Cartesian product $\prod_{i=1}^d A_i:=A$, call $A_i$ the part of $A$ in the $i$th direction.

\end{itemize}
\section{Mass transference principle: known results}


At the very beginning, the Hausdorff dimension/measure theory in classic Diophantine approximation are treated each time only for one case. The situation changes since the notion called
``{\em regular system}"  was introduced by A. Baker \& W. Schmidt \cite{BakS} in 1970 and ``{\em ubiquitous system}"
introduced by Dodson, Rynne \& Vickers \cite{DoRV} in 1990.
Both of them perform sufficiently well in establishing a general principle for the Hausdorff theory of limsup set defined by balls (see \cite{Ber99}, \cite{BBD02}, \cite{BerD}, \cite{Bug1}, \cite{Lev}, \cite{Ryn1}, etc for applications). In 2006, another landmark work called {\em mass transference principle} was established by Beresnevich \& Velani. Besides of unifying most of the known results, these two powerful principles also intervene many of the recent works on the Hausdorff theory for limsup defined by balls. Here we state the results in their most modern version after Beresnevich, Dickinson \& Velani \cite{BDV06} and Beresnevich \& Velani \cite{BV06}.

\subsection{Mass transference principle from balls to balls}

\subsubsection{{\bf{Transference principle under ubiquity}}}\

Let $\Omega$ be a compact metric space equipped with a non-atomic probability measure $m$. Let $\{\R_{\alpha}: \alpha\in J\}$
 be a family of subsets in $\Omega$ indexed by an infinite, countable
set $J$. The sets $\{\R_{\alpha}: \alpha\in J\}$ are referred to as resonant sets. Let $\beta: J \to \mathbb{R}^+ : \alpha \to \beta_{\alpha}$ which
 attaches a weight $\beta_{\alpha}$¯® to the resonant set $\R_{\alpha}$ such that for any $M>1$, $\{\alpha\in J: \beta_{\alpha}<M\}$ is finite.

Let $\rho: \mathbb{R}^+\to \mathbb{R}^+$ be a non-increasing function with $\rho(u)\to 0$ as $u\to \infty$. Let $\{\ell_n, u_n\}$ be two increasing sequences  with $$
\ell_n\le u_n, \ \ {\text{and}}\ \lim_{n\to \infty}\ell_n=\infty.$$Define $$
J_n=\{\alpha\in J: \ell_n\le \beta_{\alpha}\le u_n\}.
$$

Additionally, assume the following holds.
\begin{itemize}\item The measure $m$ is $\delta$-Ahlfors regular: there exist constants $c_1, c_2, r_o>0$ such that for any $x\in \Omega$ and $r\le r_o$, $$
c_1r^{\delta}\le m(B(x,r))\le c_2r^{\delta}.$$

\item
 Intersection property: for some $0\le \kappa< 1$, with sufficiently large $n$,
for any $\alpha\in J$ with $\beta_{\alpha}\le u_n$, $x\in \R_{\alpha}$ and $0<\lambda\le \rho(u_n)$, one has \begin{align*}
m\Big(B(x,1/2\rho(u_n))\cap \Delta(\R_{\alpha}, \lambda)\Big)\gg \lambda^{\delta(1-\kappa)}\rho(u_n)^{\delta\kappa},\\
 \ m\Big(B'\cap B(x,3\rho(u_n))\cap \Delta(\R_{\alpha}, 3\lambda)\Big)\ll \lambda^{\delta(1-\kappa)}\Big(r_{B'}\Big)^{\delta\kappa},
\end{align*} where $B'$ is any ball centered on the resonant set $\R_{\alpha}$ with radius $r_{B'}\le 3\rho(u_n)$.
\end{itemize}

\begin{dfn}[Local ubiquity system] Call $(\{\R_{\alpha}\}_{\alpha\in J}, \beta)$ a local $m$-ubiquity system with respect to $\rho$ if there exists a constant $c>0$ such that for any ball $B$ in $X$, \begin{equation}\label{fff7}
m\left(B\cap \bigcup_{\alpha\in J_n}\Delta(\R_{\alpha}, \rho(u_n))\right)\ge c\cdot m(B), \ \ {\text{for}}\ n\ge n_o(B).
\end{equation}\end{dfn}

Consider the set $$
\Lambda(\psi)=\Big\{x\in \Omega: x\in \Delta(\R_{\alpha}, \psi(\beta_{\alpha})), \ {\text{i.m.}}\ \alpha\in J\Big\},
$$ where $\psi$ is a non-increasing positive function defined on $\mathbb{R}^+$. 
\begin{theorem}[Beresnevich, Dickinson \& Velani \cite{BDV06}]\label{t22}
Let $\Omega$ be a compact metric space with a $\delta$-Ahlfors regular probability measure $m$. Suppose that $(\{\R_{\alpha}\}_{\alpha\in J}, \beta)$ is a local $m$-ubiquity system with respect to $\rho$ and satisfies the intersection property. Let $f$ be a dimension function such that $f(r)/r^{\delta}$ decreases to infinity as $r\to 0$. Furthermore, suppose that $f(r)/r^{\delta\kappa}$ is increasing. Let $h$ be a real positive function given by $$
h(u):=f(\psi(u))\psi(u)^{-\delta\kappa}\rho(u)^{-\delta(1-\kappa)}\ {\text{and let}}\ H:=\limsup_{n\to\infty}h(u_n).
$$ \begin{itemize}\item (i) Suppose that $H=0$ and $\rho$ satisfies, for some $c<1$, $\rho(u_{n+1})\le c\rho(u_n)$ for all $n\gg 1$. Then $$
\mathcal{H}^f(\Lambda(\psi))=\infty, \ \ {\text{if}}\ \ \sum_{n=1}^{\infty}h(u_n)=\infty.
$$

\item (ii) Suppose that $0<H\le \infty$. Then $\mathcal{H}^f(\Lambda(\psi))=\infty.$
\end{itemize}\end{theorem}

\subsubsection{{\bf{Transference principle under full measure}}}\

Another landmark work about the Hausdorff measure/dimension theory for limsup set is presented by Beresnevich \& Velani \cite{BV06}
 where the ubiquity condition is weakened to the full measure condition of the limsup set $\limsup \Delta(\R_{\alpha}, \rho(\beta_{\alpha}))$
 when the resonant sets $\{\R_{\alpha}: \alpha\in J\}$ are points.
%

Let $\Omega$ be a locally compact metric space, and $g$ a doubling dimension function.
Suppose that there exist $0<c_1\le c_2<\infty$ and $r_o>0$ such that for all $x\in \Omega$ and $0<r<r_o$,
\begin{equation*}
c_1g(r)\le \mathcal{H}^{g}(B(x,r))\le c_2 g(r).
\end{equation*}
Let $f$ be a dimension function and write $B^f(x,r)$ for the ball $B(x, g^{-1} (f(r)))$.

\begin{theorem}[Beresnevich \& Velani \cite{BV06}]\label{t33}
Let $\Omega$ be a locally compact metric space, $f$ a dimension function and $g$ a doubling dimension function. Assume that $\{B_i\}_{i\in \N}$ is a sequence of balls in $\Omega$ with radii tending to $0$, and that $\frac{f(r)}{g(r)}$ increases as $r\to 0_+$.
If, for any ball $B$ in $\Omega$,
$$\mathcal{H}^{g}\Big(B \cap \limsup_{i\to \infty}B_i^f\Big)= \mathcal{H}^{g}(B);$$
then, for any ball $B$ in $\Omega$,
$$\mathcal{H}^f\Big(B \cap \limsup_{i\to \infty}B_i^g\Big)= \mathcal{H}^f(B).$$
\end{theorem}


Theorem \ref{t33} is also generalized to the case when the resonant sets $\{\R_{\alpha}: \alpha\in J\}$ are planes in $\mathbb{R}^d$ by Allen \& Beresnevich \cite{AllB} and general space with the intersection property similar to that for affine space by Allen \& Baker \cite{AllBa}.


Let $\{\R_{j}: j\ge 1\}$ be a sequence of resonant sets in $\Omega$. Call them satisfy a local scaling property with respect to $\kappa\in [0,1)$,
if there exists $r_1>0$ such that for any $\lambda<r<r_1$ and all $x\in \R_{j}$  with $j\ge 1$,
\begin{equation*}
\mathcal{H}^g\Big(B(x,r)\cap \Delta(\R_{j}, \lambda)\Big)\asymp g(\lambda)^{1-\kappa}\cdot g(r)^{\kappa}.
\end{equation*}Let $\Upsilon:\N\to \mathbb{R}: j\to \Upsilon_j$ be a non-negative real valued function on $\N$ such that $\Upsilon_j\to 0$ as $j\to \infty$. Consider the set $$
\Lambda(\Upsilon):=\Big\{x\in \Omega: x\in \Delta(\R_j, \Upsilon_j), \ {\text{i.m.}}\ j\in \N\Big\}.
$$
\begin{theorem}[Allen \& Beresnevich \cite{AllB}, Allen \& Baker \cite{AllBa}]
Let $\Omega$ be a locally compact metric space and let $g$ be a doubling dimension function. Assume the $\kappa$-scaling property for $0\le \kappa<1$. Let $f$ be a dimension function such that  $f/g$ is monotonic and $f/g^{\kappa}$ be also a dimension function.  Suppose that, for any ball $B$ in $\Omega$, $$
\mathcal{H}^g\left(B\cap \Lambda\left(g^{-1}\left(\left(\frac{f(\Upsilon)}{g(\Upsilon)^{\kappa}}\right)^{\frac{1}{1-\kappa}}\right)\right)\right)=\mathcal{H}^g(B).
$$  Then for any ball $B$ in $\Omega$,
$$\mathcal{H}^f(B\cap \Lambda(\Upsilon))=\mathcal{H}^f(B).$$
\end{theorem}

We call all the above notable results as mass transference principles from balls to balls, since even for general resonant sets, the above results concern the transference principle from limsup sets defined by an isotropic thicken of the resonant sets to limsup sets defined by a smaller isotropic thicken of the resonant sets.
\subsection{Mass transference principle from balls to rectangles}

As stated before, the limsup sets defined by rectangles also take important role in metric Diophantine approximation. So, a first attempt is to see whether there are some transference principles from balls to rectangles. This is given by Wang, Wu \& Xu \cite{WWX}.

Let $\{x_n\}_{n\ge 1}$ be a sequence of points in the unit cube $[0,1]^d$ with $d\ge 1$ and $\{r_n\}_{n\ge 1}$ be a sequence of positive numbers tending to zero.
For any ${\bold t}=(t_1,\cdots,t_d)$ with $1\le t_1\le t_2\le \cdots \le t_d$, define
$$W_{\bold t}:=\Big\{x\in [0,1]^d\colon x\in B(x_n,r_n^{\bold{t}}), \ {\text{i.m.}}\ n\in \N\Big\},$$
where we use $B(x,r^{\bold t})$ to denote a rectangle with center $x$ and side-length $(r^{t_1}, r^{t_2},\cdots,r^{t_d})$.

\begin{theorem}[Wang, Wu \& Xu \cite{WWX}]
Let ${\bold t}=(t_1,\cdots,t_d)$ with $1\le t_1\le \cdots \le t_d$. Let $\{B_i: i\ge 1\}$ be a sequence of balls such that for any ball $B\subset [0,1]^d$ $$\mathcal{L}\Big(B\cap \limsup_{i\to\infty}B_i\Big)=\mathcal{L}(B).$$ Then we have
$$\hdim W_{\bold t}\ge \min\Big\{\frac{d+jt_j-\sum_{i=1}^jt_i}{t_j}, 1\le j\le d\Big\}:=s,$$
 and for any ball $B\subset [0,1]^d$, $$\mathcal{H}^s(W_{\bold t}\cap B)=\mathcal{H}^s(B).$$
\end{theorem}

The above transference  principle from balls to rectangles is extended to limsup sets defined by open sets by Koivusalo \& Rams \cite{Ram}.
Let $E$ be a bounded open set. A generalized singular function in \cite{Ram} is defined as $$
\varphi^t(E)=\sup_{\mu}\inf_{x\in E}\inf_{x>0}\frac{r^s}{\mu(E\cap B(x,r))},
$$ where the supremum is taken over all Borel probability measures supported on $E$.\begin{theorem}[Koivusalo \& Rams \cite{Ram}]
Let $\{B_i: i\ge 1\}$ be a sequence of balls such that for any ball $B\subset [0,1]^d$ $$\mathcal{L}\Big(B\cap\limsup_{i\to\infty}B_i\Big)=\mathcal{L}(B).$$
Let $\{E_i: i\ge 1\}$ be a sequence of open sets such that $E_i\subset B_i$ for each $i\ge 1$. Then one has
$$\hdim \Big(\limsup_{i\to\infty}E_i\Big)\ge \sup\bigg\{t: \mathcal{L}\left(\limsup\{B_i: \varphi^t(E_i)\ge \mathcal{L}(B_i)\}\right)=1\bigg\}.$$

\end{theorem}

For a survey on mass transference principles, one is also referred to Allen \& Troscheit \cite{AllT}. For dimensional result of random limsup sets, one is referred to \cite{Durand,EJJ,FanWu,FJJ,Per} and references therein.
\smallskip

Clearly the mass transference principle from balls to balls cannot deal with the rectangle case. A shortage of the mass transference principle from balls to rectangles is that they require $t_i\ge 1$ for all $i\ge 1$. For example, when applied to the set $W(\tau_1,\cdots, \tau_d)$ (defined in (\ref{e4}) when $\psi_i(q)=q^{-\tau_i}$), they needs a full measure limsup set defined by balls which compels that $\tau_i\ge 1/d$ for all $1\le i\le d$ by Dirichlet/Khintchine's theorem\cite{WWX}. But if the full measure limsup set is not required to be defined by balls but rectangles, one will have much freedom on the choice of the full measure set by Minkowski's theorem. So as far as a limsup set defined by rectangles is concerned, it would be much reasonable to develop the transference principle from rectangles to rectangles. This is the main task of this paper.

%

\section{Mass transference principle from rectangles to rectangles}

Generally speaking, the mass transference principle says that if there is a full (Lebesgue) measure statement for a limsup set, then there will be a full Hausdorff measure statement for the shrunk limsup set.

In this paper, by introducing a notion called {\em ubiquity for rectangles}, one can also transfer a full measure property of a limsup set defined by a sequence of big rectangles to the full Hausdorff measure/dimension theory for the limsup set defined by shrinking the big rectangles to smaller ones. Here the full measure property goes with or without the assumption of ubiquity for rectangles. So, we call them {\em transference principle under ubiquity} and {\em transference principle under full measure} respectively. As will be seen later, the argument under the full measure case follows almost directly from that under the ubiquity case.

We begin with the framework we work with.


\subsection{Framework}Fix an integer $d\ge 1$.  For each $1\le i\le d$, let $(X_i, |\cdot |_i, m_i)$ be a bounded locally compact metric
space with $m_i$ a $\delta_i$-Ahlfors regular probability measure.

Then we consider the product space $(X, |\cdot|, m)$, where $$
{\text{space}}\  X=\prod_{i=1}^dX_i; \   \ \ {\text{measure}}\  m=\prod_{i=1}^dm_i; \ \ \ {\text{metric}}\ \ |\cdot|=\max_{1\le i\le d}|\cdot |_i.
$$ So a ball $B(x,r)$ in  $X$ is in fact the product of balls in $\{X_i\}_{1\le i\le d}$, i.e. $$
B(x,r)=\prod_{i=1}^d B(x_i,r), \ {\text{for}}\ x=(x_1,\cdots, x_d).
$$ Also, we do not distinguish a ball with a hypercube.

As before, let $J$ be an infinite countable index set; $\beta: J\to \mathbb{R}^+$ a positive function such that for any $M>1$, $\{\alpha\in J: \beta_{\alpha}<M\}$ is finite; $\{\ell_n, u_n: n\ge 1\}$ two sequences of integers such that $u_n\ge \ell_n\to \infty$ as $n\to \infty$ and define $$
J_n=\{\alpha\in J: \ell_n\le\beta_{\alpha}\le u_n\}.
$$Let $\rho:\mathbb{R}^+\to \mathbb{R}^+$ be non-increasing and $\rho(u)\to 0$ as $u\to \infty$.

For each $1\le i\le d$, let $\{\R_{\alpha, i}: \alpha\in J\}$ be a sequence of subsets of $X_i$. The resonant sets in $X$ we are considering are $$
\Big\{\R_{\alpha}=\prod_{i=1}^d\R_{\alpha, i}, \ \ \alpha\in J\Big\}.
$$ 
For any $\a=(a_1,\cdots, a_d)\in (\mathbb{R}^{+})^d$, {denote}$$\Delta(\R_{\alpha}, \rho^{\a})=\prod_{i=1}^d\Delta(\R_{\alpha,i}, \rho^{a_i}),
$$ where $\Delta(\R_{\alpha,i}, \rho^{a_i})$ is the neighborhood of $\R_{\alpha,i}$ in $X_i$ and call it the part of $\Delta(\R_{\alpha}, \rho^{\a})$ in the $i$th direction.

Fix $\a=(a_1,\cdots, a_d)\in (\mathbb{R}^{+})^d$. The set we are considering is: for $\t=(t_1,\cdots,t_d)\in (\mathbb{R}^+)^d,$ the set $$
W(\bold{t})=\Big\{x\in X: x\in \Delta(\R_{\alpha}, \rho(\beta_{\alpha})^{\bold{a}+\bold{t}}), \ {\text{i.m.}}\ \alpha\in J\Big\},
$$ and more generally, replacing $\rho^{\bold{t}}$ by general functions $
\Psi=(\psi_1,\cdots,\psi_d): \mathbb{R}^+\to (\mathbb{R}^+)^d,
$ the set
 $$
W(\Psi)=\Big\{x\in X: x\in \prod_{i=1}^d\Delta\Big(\R_{\alpha,i}, \ \ \rho(\beta_{\alpha})^{{a_i}}\cdot \psi_i(\beta_{\alpha})\Big), \ {\text{i.m.}}\ \alpha\in J\Big\}.
$$
The smaller {\em `rectangle'} $\Delta(\R_{\alpha}, \rho(\beta_{\alpha})^{\bold{a}+\bold{t}})$ can be viewed as the shrunk one from $\Delta(\R_{\alpha}, \rho(\beta_{\alpha})^{\bold{a}})$.
\medskip

We require the resonant sets having special forms which is a generalization when the resonant sets are points or affine subspaces.
\begin{dfn}[$\kappa$-scaling property] Let $0\le \kappa<1$. For each $1\le i\le d$, call $\{\R_{\alpha, i}\}_{\alpha\in J}$ having
$\kappa$-scaling property if for any $\alpha\in J$ and any ball $B(x_i,r)$ in $X_i$ with center $x_i\in \R_{\alpha, i}$ and $0<\epsilon<r$,
one has \begin{equation*}
c_2 r^{\delta_i\kappa}\epsilon^{\delta_i (1-\kappa)}\le m_i\Big(B(x_i,r)\cap \Delta(\R_{\alpha, i}, \epsilon)\Big)\le c_3 r^{\delta_i\kappa}\epsilon^{\delta_i (1-\kappa)},
\end{equation*} for some absolute constants $c_2, c_3>0$.\end{dfn}

We list some examples for which the $\kappa$-scaling property is valid.\begin{itemize}
\item (1) For each $\alpha\in J$ and $1\le i\le d$, the $i$th coordinate $\R_{\alpha, i}$ is a point in $X_i$, so $\kappa=0$.\smallskip

\item (2) Let $X_i=\mathbb{R}^n$ for each $1\le i\le d$. For each $\alpha\in J$ and $1\le i\le d$, the $i$th coordinate $\R_{\alpha, i}$ is an $l$-dimensional affine subspace in $X_i$, so $\delta_i=n$ and $\kappa=l/n$.\smallskip

\item (3) Let $X_i=\mathbb{R}^n$ and $\R_{\alpha, i}$ be an $l$-dimensional smooth compact manifold embedded in $X_i$ for all $1\le i\le d$ and $\alpha\in J$. Then $\delta_i=n$ and $\kappa=l/n$.\smallskip

\item (4) Let $X_i=\mathbb{R}^n$ and $\R_{\alpha, i}$ be some self-similar set with open set condition with dimension $l$ for all $1\le i\le d$ and $\alpha\in J$. Then $\delta_i=n$ and $\kappa=l/n$.
\end{itemize}
For a proof of the last two examples, one is referred to Allen \& Baker \cite{AllBa}.

\subsection{Transference principle under ubiquity}\

The ubiquity condition for balls (\ref{fff7}) 
is mainly rooted in Dirichlet's theorem \cite{BDV06}. So, for the metric theory
of limsup sets defined by rectangles, Minkowski's theorem should intervene in some form. Thus, the following extended ubiquity condition should be the right one suitable for the rectangle case.

The notion {\em ``local ubiquity for rectangles"} we will introduce is an extended one of the notion {\em ``local ubiquity for balls"}
introduced by Beresnevich, Dickinson \& Velani \cite{BDV06} and try to catch the nature of the rectangles inspired by Minkowski's theorem. 

\begin{dfn}[Local ubiquity system for rectangles] Call $(\{\R_{\alpha}\}_{\alpha\in J}, \beta)$ a local ubiquity system for rectangles with respect to $(\rho, \a)$ if there exists a constant $c>0$ such that for any ball $B$ in $X$,
\begin{equation*}
\limsup_{n\to\infty}m\left(B\cap \bigcup_{\alpha\in J_n}\Delta(\R_{\alpha}, \rho(u_n)^{\a})\right)\ge c\cdot m(B).
\end{equation*}\end{dfn}

\begin{dfn}[Uniform local ubiquity system for rectangles]\label{dfn1} Call $(\{\R_{\alpha}\}_{\alpha\in J}, \beta)$ a uniform local ubiquity system for rectangles with respect to $(\rho, \a)$ if there exists a constant $c>0$ such that for any ball $B$ in $X$, \begin{equation}\label{ff5}
m\left(B\cap \bigcup_{\alpha\in J_n}\Delta(\R_{\alpha}, \rho(u_n)^{\a})\right)\ge c\cdot m(B), \ \ {\text{for all}}\ \ n\ge n_o(B).
\end{equation}\end{dfn}

The local ubiquity property for rectangles plays the role of the full $m$-measure statement for a limsup set defined by rectangles.

%
%
%
%

\begin{theorem}[Hausdorff dimension theory]\label{t1}
Under the setting given above. Assume the $\delta_i$-Ahlfors regularity for $m_i$ with $1\le i\le d$, the local ubiquity for rectangles and the $\kappa$-scaling property. One has $$
\hdim W(\bold{t})\ge \min_{A\in \mathcal{A}}\Big\{\sum_{k\in \K_1}\delta_k+\sum_{k\in \K_2}\delta_k+\kappa \sum_{k\in \K_3}\delta_k+(1-\kappa)\frac{\sum_{k\in \K_3}a_k\delta_k-\sum_{k\in \K_2}t_k\delta_k}{A}\Big\}
$$ where $$
\mathcal{A}=\{a_i, a_i+t_i, 1\le i\le d\}
$$ and for each $A\in \mathcal{A}$, the sets $\K_1, \K_2,\K_3$ give a partition of $\{1,\cdots,d\}$ defined as \begin{align*}&\K_1=\Big\{k: a_k\ge A\Big\},\ \ \K_2=\Big\{k: a_k+{t_k}\le A\Big\}\setminus \K_1,\ \
\K_3=\{1,\cdots, d\}\setminus (\K_1\cup \K_2).\end{align*}\end{theorem}


If denote by $s=s(\bold{t})$ the dimensional number given in Theorem \ref{t1}, we have\begin{theorem}[Hausdorff measure theory]\label{t2}
For any ball $B\subset X$, $$
\mathcal{H}^s(B\cap W(\bold{t}))=\mathcal{H}^s(B).$$
\end{theorem}


\begin{remark}
The dimensional numbers given in Theorem \ref{t1} are just those corresponding to cover the collection of rectangles
$$\Big\{\Delta(\R_{\alpha}, \rho(u_n)^{\a+\bold{t}}): \alpha\in J_n\Big\}$$
by balls of radius $\rho(u_n)^{A}$ with $A\in \mathcal{A}$.
We
will give a detailed explanation of how these dimensional numbers arise in a rather natural way in the next section.
\end{remark}

For general functions $\Psi=(\psi_1,\cdots,\psi_d)$, we denote by $\widehat{\mathcal{U}}$ and $\mathcal{U}$ respectively the set of the accumulation points of the sequences $$
\left\{\Big(\frac{\log \psi_1(u_n)}{\log \rho(u_n)},\cdots, \frac{\log \psi_d(u_n)}{\log \rho(u_n)}\Big): n\in \N\right\}, \ \ {\text{and}}\ \ \left\{\Big(\frac{\log \psi_1(n)}{\log \rho(n)},\cdots, \frac{\log \psi_d(n)}{\log \rho(n)}\Big): n\in \N\right\}.
$$ Then we have \begin{theorem}[General case]\label{t3}
Assume that $\psi_i$ is decreasing for each $1\le i\le d$. Assume the $\delta_i$-Ahlfors regularity for $m_i$ with $1\le i\le d$, the uniform
local ubiquity for rectangles and the $\kappa$-scaling property. One has
$$\hdim W(\Psi)\ge \max\Big\{s(\bold{t}): \bold{t}\in \widehat{\mathcal{U}}\Big\}.
$$ If one further has that \begin{equation}\label{d3}
\lim_{n\to\infty}\frac{\log \rho(u_{n+1})}{\log \rho(u_n)}=1,
\end{equation} then $$\hdim W(\Psi)\ge \max\Big\{s(\bold{t}): \bold{t}\in \mathcal{U}\Big\}.
$$
\end{theorem}

In latter applications, the systems always have a uniform ubiquity property for rectangles and
the extra condition (\ref{d3}) is also satisfied so that Theorem \ref{t3} can be applied.

For the simultaneous Diophantine approximation defined by balls (\ref{e3}), it is known that the dimension of the set (\ref{e3})
depends on $\tau=\liminf_{q\to \infty}-\frac{\log \psi(q)}{\log q}$. While for the rectangle case (\ref{e4}), it will be seen that the dimension of the set (\ref{e4}) depends on the accumulation points of the sequences $$ \left\{\Big(\frac{-\log \psi_1(q)}{\log q},\cdots, \frac{-\log \psi_d(q)}{\log q}\Big): q\in \N\right\}.
$$ Even such an observation, though simple, does not appear in the literature before.

\subsection{Transference principle under full measure}\

In the ubiquity case, what we really consider is the subset $$
\Big\{x\in X: x\in \bigcup_{\alpha\in J_n}\Delta(\R_{\alpha}, \rho(u_n)^{\bold{a+t}}), \ {\text{i.m.}}\ n\in \N\Big\},
$$ of $W(\bold{t})$. In other words, all the rectangles with $\alpha\in J_n$ are all of the same size. So, one would like to know how about the case without the ubiquity assumption. This is done by replacing the ubiquity assumption to that the set
$$
\Big\{x\in X: x\in \Delta(\R_{\alpha}, \rho(\beta_{\alpha})^{\bold{a}}), \ {\text{i.m.}}\ \alpha\in J\Big\}.
$$ has full $m$-measure directly.
\begin{theorem}\label{ttt1}
Assume that for each $1\le i\le d$, the measure $m_i$ is $\delta_i$-Ahlfors regular and the resonant set $\R_{\alpha,i}$ has the $\kappa$-scaling property for $\alpha\in J$. Suppose \begin{equation}\label{j1}
m\left(\limsup_{\alpha\in J, \beta_{\alpha}\to\infty}\Delta(\R_{\alpha}, \rho(\beta_{\alpha})^{\a})\right)=m(X).
\end{equation} Then we have $$
\hdim W(\bold{t})\ge s(\bold{t}).
$$ 
\end{theorem}

\begin{remark}
Ubiquity is stronger than full measure, where the former requires that all the rectangles are of the same size, while the latter is not. However it will be seen that the proof of Theorem \ref{ttt1} follows
almost directly from that of the ubiquity case (Theorem \ref{t2}) by one more step: group the rectangles of almost the same size.
\end{remark}

\subsection{A property about the dimensional number}

If we look into the dimensional number $s(\bold{t})$ a little further, we will have the following property on it. Let $$
\widehat{\mathcal{A}}=\{a_i+t_i: 1\le i\le d\},
$$ which corresponds to exponent of the sidelengths of the shrunk rectangles. 
Define $$
\hat{s}(\bold{t}):=\min_{A\in \widehat{\mathcal{A}}}\Big\{\sum_{k\in \K_1}\delta_k+\sum_{k\in \K_2}\delta_k+\kappa \sum_{k\in \K_3}\delta_k+(1-\kappa)\frac{\sum_{k\in \K_3}a_k\delta_k-\sum_{k\in \K_2}t_k\delta_k}{a_i+t_i}\Big\},
$$ where $\K_1, \K_2,\K_3$ are defined as in Theorem \ref{t1} for $A=a_i+t_i$ with $1\le i\le d$.  Then
\begin{proposition}\label{p6}$
s(\bold{t})=\hat{s}(\bold{t}).
$
\end{proposition}

\begin{remark}
The appearance of the dimensional number $\hat{s}(\bold{t})$ is not as natural as $s(\bold t)$ as will be seen in Section \ref{s4}. This is why we state Theorem \ref{t1} by using $s(\bold{t})$. But the significance of the observation of Proposition \ref{p6} is that when covering a collection of shrunk rectangles with the same sidelengths, to reach the optimal cover, we need only consider the covers by balls with the radius being the sidelengths of the shrunk rectangles, without the necessity of considering the cover by balls with radius being the sidelengths of the big rectangles.
\end{remark}

One of the main difficulty of the rectangle setting lies in finding the optimal cover of a sequence
of shrunk rectangles. As Proposition \ref{p6} indicates, the optimal cover is given by covering them by balls with radius as the sidelengths
of the shrunk rectangles, but this is far from being obvious.

We end this section by a remark on the sidelengths of the big rectangles and small rectangles. \begin{remark}
  In the previous statement, the sidelengths of the rectangles are of the form $$
  {\text{big one}}: \rho(u_n)^{\bold{a}}, \ \rho(\beta_{\alpha})^{\bold{a}}; \ \ {\text{small one}}: \rho(u_n)^{\bold{a}+\bold{t}}, \ \rho(\beta_{\alpha})^{\bold{a}+\bold{t}}.
  $$ We remark that the results and the proof are unchanged if extra constants intervene, namely, $$
  {\text{big one}}: c\cdot\rho(u_n)^{\bold{a}}, \ c\cdot\rho(\beta_{\alpha})^{\bold{a}}; \ \ {\text{small one}}: c'\cdot\rho(u_n)^{\bold{a}+\bold{t}}, \ c'\cdot\rho(\beta_{\alpha})^{\bold{a}+\bold{t}}
  $$ for any $c>0, c'>0$.
\end{remark}
\section{Geometric interpretation of the dimensional number}\label{s4}

In this section, we explain the nature about the dimensional numbers and the sets $\K_i$ appearing in Theorem \ref{t1}. To make the ideas more apparent, we consider one simple case that the resonant sets $\{\R_{\alpha}: \alpha\in J\}$ are points $\{x_{\alpha}: \alpha\in J\}$ in $X=\prod_{i=1}^dX_i$. Then $\kappa=0$ in the scaling property. Let's consider the dimension of the following set
$$
\widehat{W}(\bold{t})=\bigcap_{N=1}^{\infty}\bigcup_{n=N}^{\infty} \bigcup_{\alpha\in \widetilde{J}_n}\Delta(x_{\alpha}, \rho(u_n)^{\bold{a}+\bold{t}}),
$$ where $\widetilde{J}_n$ is a subset of $J_n$ such that the big rectangles \begin{equation}\label{ff1}
\Big\{\Delta(x_{\alpha}, \rho(u_n)^{\bold{a}}): \alpha\in \widetilde{J}_n\Big\}
\end{equation} are disjoint and we also assume that these big rectangles are very {\em nice} placed.

The set $\widehat{W}(\bold{t})$ is a subset of $W(\bold{t})$, so the dimension of $\widehat{W}(\bold{t})$ gives a lower bound for that of $W(\bold{t})$.
 Though determining the lower bound of the dimension of some set is always more difficult than for the upper bound, the argument for the upper bound
always provides some information about the potential exact dimension of a set. So we pay attention to the upper bound of the dimension of $\widehat{W}(\bold{t})$.

The task is to find an optimal cover for {\em a collection of rectangles} instead of one. When covering a single rectangle, we know that Falconer's singular value function \cite{Fal2} gives the optimal one. More precisely, for a rectangle $R$ in $\mathbb{R}^d$ of sidelengths $l_1\ge l_2\ge \cdots \ge l_d$, the singular function $\varphi^s$  defined by $$
\varphi^s(R)=\min\Big\{l_1 \cdots l_{i-1}l_i^{s-i+1}: 1\le i\le d\Big\},
$$ is the minimal value for the $s$-volume of all covers of $R$ by balls (up to a constant multiple).

However when one needs to cover a collection of rectangles, one should be careful about the relative positions of these rectangles.
This is because if these rectangles are close enough to each other, then when covering them by balls, one ball may cover part of many rectangles.
\medskip

Write $r_n=\rho(u_n)$. We call respectively the rectangles $$
\Delta(x_{\alpha}, r_n^{\bold{a}})=\prod_{k=1}^dB(x_{\alpha, k}, r_n^{a_k}), \ \ \Delta(x_{\alpha}, r_n^{\bold{a+t}})=\prod_{k=1}^dB(x_{\alpha,k}, r_n^{a_k+t_k})
$$ as big rectangle and shrunk rectangle. For each $1\le k\le d$, we also call the balls $$
B(x_{\alpha, k}, r_n^{a_k}), \ \ B(x_{\alpha,k}, r_n^{a_k+t_k})
$$ as big ball and shrunk ball in the $k$th direction of $
\Delta(x_{\alpha}, r_n^{\bold{a}})$ and $\Delta(x_{\alpha}, r_n^{\bold{a+t}})
$.


The separation condition and the nice placed assumption of the big rectangles  (\ref{ff1}) implies the following two observations: by a volume argument, one has \begin{itemize}
\item in the $k$th direction, the number $T_k$ of big balls $$
\Big\{B(x_{\alpha, k}, r_n^{a_k}): \alpha\in \widetilde{J}_n\Big\}
$$ can be estimated as $$
T_k\ll r_n^{-a_k\delta_k}.
$$

\item Let $r\ge \rho(u_n)^{a_k}=r_n^{a_k}$. Then in the $k$th direction, one needs about $$
\Big(\frac{1}{r}\Big)^{\delta_k}
$$ balls of radius $r$ to cover $X_k$, so all the big balls in the $k$th direction. More precisely, for each $x\in X_k$, sprout it into a ball of radius $r/5$; then by $5r$-covering lemma to get a collection of balls $\{B(x_i, r): {i\in I}\}$ which covers $X_k$ but $\{\frac{1}{5}B(x_i, r): i\in I\}$ are disjoint. Then together with the Ahlfors-regularity of $m_k$, $$
1=m_k(X_k)\ge \sum_{i\in I}m_k\Big(B(x_i,\frac{r}{5})\Big)\asymp \sum_{i\in I}r^{\delta_k}= r^{\delta_k}\cdot
\sharp I.
$$
\end{itemize}

Now we give the detail about how to find an optimal cover of the rectangles \begin{equation}\label{ff2}
\bigcup_{\alpha\in \widetilde{J}_n}\Delta(x_{\alpha}, r_n^{\bold{a}+\bold{t}}).
\end{equation}


\begin{remark}\label{r1} At first, we give a remark on how the sets $\K_1,\K_2$ and $\K_3$ appear. Think of that there is a collection of disjoint rectangles of the same sidelengths, saying $(r_n^{a_1},\cdots,r_n^{a_d})$. Then we shrink each rectangle to smaller one with sidelengths $(r_n^{a_1+t_1},\cdots, r_n^{a_d+t_d})$. Now the task is to find an optimal cover of these shrunk rectangles. The potential optimal cover would possibly arise among the cases that we cover them by balls of radius $$
r_n^{a_i}, \ \ {\text{and}}\ \ r_n^{a_i+t_i}, \ \ \ {\text{for}}\ 1\le i\le d,
$$in other words, by balls of radius $$
r_n^{A}, \ \ \ A\in \mathcal{A}.
$$

Given a ball $B$ of radius $r_n^{A}$ with $A\in \mathcal{A}$, we need compare $r_n^{A}$ with the sidelengths of the big rectangles and the shrunk rectangles.
\begin{itemize}\item (a). In some directions $k$, the radius $r_n^{A}$ is larger than the sidelengths $r_n^{a_k}$ of the big balls, so in these directions the ball $B$ may cover many big balls so for the shrunk balls, which leads to the definition of $\K_1=\{k: a_k\ge A\}$.

\item (b). In some other directions $k$, the radius $r_n^{A}$ is even smaller than the sidelength $r_n^{a_k+t_k}$ of the shrunk ball, so in these directions the ball $B$ can only cover part of one shrunk ball, which leads to the definition of $\K_2=\{k: a_k+t_k\le A\}$.

    \item (c). For the left directions $k$, the radius $r_n^{A}$ is between the sidelength $r_n^{a_k}$ of the big ball and that $r_n^{a_k+t_k}$ of the shrunk one, so in these directions, the ball $B$ can cover and can only cover one shrunk ball. This gives the definition of $\K_3$.\end{itemize}
\end{remark}

Recall the alphabet $$
\mathcal{A}=\{a_k, \ a_k+t_k: 1\le k\le d\}.
$$
 By a change of the coordinates if necessary, we may assume that $a_d+t_d$ is the largest element in $\mathcal{A}$ and $a_1$ is the smallest one.

Now we cover the collection of the shrunk rectangles in (\ref{ff2}) by balls of radius $r$.

(i). When $r< r_n^{a_d+t_d}$.

For each direction $1\le k\le d$, a ball of radius $r$ can only cover a small part of a shrunk ball $B(x_{\alpha, k}, r_n^{a_k+t_k})$. So,
by a volume argument, the number of balls of radius $r$ needed to cover one shrunk ball in the $k$th direction is about $$
\left(\frac{r_n^{a_k+t_k}}{r}\right)^{\delta_k}.
$$
So, the total number of balls to cover the collection of shrunk rectangles in (\ref{ff2}) can be estimated as $$
\ll\prod_{k=1}^dT_k\cdot\left(\frac{r_n^{a_k+t_k}}{r}\right)^{\delta_k}\ll \prod_{k=1}^d r_n^{-a_k\delta_k}\cdot \left(\frac{r_n^{a_k+t_k}}{r}\right)^{\delta_k}=\prod_{k=1}^d \left(\frac{r_n^{t_k}}{r}\right)^{\delta_k}.
$$
Then the total $s$-volume of these balls is \begin{equation}\label{ff3}
\ll r^s\cdot \prod_{k=1}^d \left(\frac{r_n^{t_k}}{r}\right)^{\delta_k}.
\end{equation}

When $s\le \delta_1+\cdots+\delta_d$ the dimension of the space $X$, the quantity in (\ref{ff3}) is decreasing as $r$ increases. Bear in mind that we are looking for the optimal covers. So, in such a range of $r$, the optimal one occurs when $r$ is taken to be $r_n^{a_d+t_d}$. Hence, in such a range of $r$, the $s$-volume of the optimal cover of (\ref{ff2}) is estimated as $$
\ll r_n^{(a_d+t_d)s}\cdot \prod_{k=1}^d \left(\frac{r_n^{t_k}}{r_n^{a_d+t_d}}\right)^{\delta_k}.
$$ This quantity equals 1 if \begin{equation}\label{ff4}
s=s_d:=\delta_1+\cdots+\delta_d-\frac{\sum_{1\le k\le d}\delta_kt_k}{a_d+t_d}.
\end{equation}
The number in the right side is nothing but the dimensional number corresponding to $A=a_d+t_d$ in Theorem \ref{t1}. More precisely, when $A=a_d+t_d$ the largest element in $\mathcal{A}$, one has $$\K_1=\emptyset, \ \K_2=\{1, 2, \cdots, d\}, \ \K_3=\emptyset.$$

(ii). $r\ge r_n^{a_1}$. This gives the most loose cover. For each direction $1\le k\le d$, a ball of radius $r$ can not only cover the shrunk ball $B(x_{\alpha, k}, r_n^{a_k+t_k})$, but can also cover several big balls $B(x_{\alpha, k}, r_n^{a_k})$. Then by a volume argument, one needs $
\ll r^{-\delta_k}
$ many balls of radius $r$ to cover all the big balls in the $k$th direction. Thus, the $s$-volume of the balls of radius $r$ needed to cover the rectangles (\ref{ff2}) can be estimated as$$
\ll r^s\cdot \prod_{k=1}^dr^{-\delta_k},
$$ which will be 1 if $$
s=\delta_1+\cdots+\delta_d,
$$ which is clearly larger than the number in (\ref{ff4}), so cannot give the optimal cover.

(iii). $r_n^{a_d+t_d}\le r< r_n^{a_1}$. Write $r=r_n^A$ (Here $A$ is just a real number may not in $\mathcal{A}$).

In this case, we have to compare $A$ with $a_k$ and $a_k+t_k$ to see in the $k$th direction whether a ball of radius $r$ can cover many big balls ($a_k\ge A$), or can cover only a small part of a shrunk ball ($A\ge a_k+t_k$), or can cover the whole shrunk ball but intersect only one big ball ($a_k\le A\le a_k+t_k$). Under such a consideration, $\K_1, \K_2,\K_3$ intervene naturally.

Arrange the elements in $\mathcal{A}$ in non-decreasing order. Let $A_{i+1}, A_i$ be two consecutive and distinct terms in $\mathcal{A}$ such that $$
r_{n}^{A_{i+1}}\le r<r_{n}^{A_i}, \ \  {\text{i.e.}} \ A_i<A\le A_{i+1}.
$$

We give a partition of the directions $k\in \{1, 2,\cdots,d\}$.

\begin{itemize}
\item (a). $\K_1'$:
$$\K_1':=\Big\{k: r_{n}^{a_k}< r_{n}^{A_{i}}\Big\}=\Big\{k: r_{n}^{a_k}\le r_{n}^{A_{i+1}}\Big\}
=\{k: a_k> A_{i}\}
,$$
since there are no elements of $\mathcal{A}$ in $(A_i, A_{i+1})$.

In these directions, a ball of radius $r$ can cover many big balls. So the number of balls with radius $r$ needed to cover all the big balls in these directions is $$
\ll \prod_{k\in \K_1'}\frac{1}{r^{\delta_k}}.
$$

\item (b). $\K_2'$:
$$
\K_2'=\Big\{k: r_{n}^{A_i}\le r_{n}^{a_k+t_k}\Big\}=\Big\{k: r_{n}^{A_{i+1}}<r_{n}^{a_k+t_k}\Big\}=\{k: a_k+t_k\le A_i\}.
$$
In each of these directions, a ball of radius $r$ can only cover part of a shrunk ball. So the number of balls with radius $r$ needed to cover all the shrunk balls in these directions is $$\ll \prod_{k\in \K_2'}T_k\cdot \left(\frac{r_{n}^{a_k+t_k}}{r}\right)^{\delta_k}\ll \prod_{k\in \K_2'}\frac{1}{r_{n}^{a_k\delta_k}}\cdot \left(\frac{r_{n}^{a_k+t_k}}{r}\right)^{\delta_k}.$$

\item (c). $\K_3'$: the left terms. For each $k\in \K_3'$,
    we have $$
    r_{n}^{a_k+t_k}\le r< r_{n}^{a_k}.
    $$ So, a ball with radius $r$ can cover one shrunk ball but can intersect at most $3$ big balls in each of these directions. So the number of balls with radius $r$ needed to cover all the shrunk balls in these directions is
    $$\ll \prod_{k\in \K_3'}T_k\ll \prod_{k\in \K_3'}\left(\frac{1}{r_{n}^{a_k}}\right)^{\delta_k}.$$
\end{itemize}

Thus the total number of balls of radius $r$ needed to cover all the rectangles in (\ref{ff2}) is $$
\ll \left(\prod_{k\in \K_1'}\frac{1}{r^{\delta_k}}\right)\times \left(\prod_{k\in \K_2}\frac{1}{r_{n}^{a_k\delta_k}}\cdot \left(\frac{r_{n}^{a_k+t_k}}{r}\right)^{\delta_k}\right)\times \left(\prod_{k\in \K_3'}\left(\frac{1}{r_{n}^{a_k}}\right)^{\delta_k}\right):=T.
$$
Consider the $s$-volume of the balls in this cover, we get $Tr^s$ which is monotonic with respect to $r$. So the economic cover occurs when $r$ takes the boundary values.

When $r=r_{n}^{A_{i+1}}$. Let $Tr^s=1$. It follows that $$
s=s_{i+1}':=\sum_{k\in\K_1'}\delta_k+\sum_{k\in\K_2'}\delta_k+\frac{\sum_{k\in \K_3'}a_k\delta_k-\sum_{k\in \K_2'}t_k\delta_k}{A_{i+1}}.
$$  When $r=r_{n}^{A_{i}}$. Let $Tr^s=1$. It follows that $$
s=s_i':=\sum_{k\in\K_1'}\delta_k+\sum_{k\in\K_2'}\delta_k+\frac{\sum_{k\in \K_3'}a_k\delta_k-\sum_{k\in \K_2'}t_k\delta_k}{A_{i}}.
$$

Finally, we check that $s_i'$ and $s_{i+1}'$ are also the terms in Theorem \ref{t1}. We only give the detail for $s_i'$ since the case for $s_{i+1}'$ can be treated similarly.
Recall that $$
\K_1^i=\{k: a_k\ge A_i\}, \ \ \K_2^i=\{k: a_k+t_k\le A_i\}, \ \ \K_3^i=\{1,\cdots, d\}\setminus (\K_1^i\cup \K_2^i).
$$
So,
$$
\K_1'=\K_1^i\setminus\{k: a_k= A_i\}, \ \ \K_2'=\K_2^i, \ \ \K_3'=\K_3^i\cup \{k: a_k= A_i\}.
$$
Thus \begin{align}\label{d4}
s_i'&=\sum_{k\in\K_1^{i}}\delta_k-\sum_{k: a_k=A_i}\delta_k+\sum_{k\in\K_2^{i}}\delta_k+\frac{\sum_{k\in \K_3^{i}}a_k\delta_k+\sum_{k: a_k=A_i}a_k\delta_k-\sum_{k\in \K_2^{i}}t_k\delta_k}{A_{i}}\nonumber\\
&=\sum_{k\in\K_1^{i}}\delta_k+\sum_{k\in\K_2^{i}}\delta_k+\frac{\sum_{k\in \K_3^{i}}a_k\delta_k-\sum_{k\in \K_2^{i}}t_k\delta_k}{A_{i}}.
\end{align}


These give the reason how the dimensional numbers in Theorem \ref{t1} arise. From the argument above, one can see that the optimal cover occurs only when one covers the collection of rectangles by balls of radius $$
r=r_n^{A}, \ A\in \mathcal{A},
$$ and then the dimensional number follows correspondingly. 

Meanwhile the above argument in fact shows that \begin{corollary}\label{c1}
Assume the separation condition (\ref{ff1}) and assume that for any $\epsilon>0$, $$
\sum_{n\ge 1}\rho(u_n)^{\epsilon}<\infty.
$$ Then one has $$
\hdim \widehat{W}(\bold{t})\le \min_{A_i\in \mathcal{A}}\Big\{\sum_{k\in \K_1}\delta_k+\sum_{k\in \K_2}\delta_k+\frac{\sum_{k\in \K_3}a_k\delta_k-\sum_{k\in \K_2}t_k\delta_k}{A_i}\Big\}
$$ where $
\mathcal{A}
$ and $\K_1, \K_2,\K_3$ are the same as in Theorem \ref{t1}.
\end{corollary}

\begin{remark}\label{r2} From the geometrical point of view or as shown for $s_i'$ in the argument above, there are some freedom for the choices of $\K$ which will not affect the dimensional number. We will not go into the details.

(i). We can change $\ge$ to $>$ in $\K_1$ or $\le$ to $<$ in $\K_2$ (can be simultaneously).

(ii). If there is an index $k$ such that $$
a_k=A=a_k+t_k,
$$ it can be put into $\K_2$ from $\K_1$.

This releases the dilemma of to which set the indexes $i, k\in \{1,\cdots,d\}$ should belong when $a_k=a_i+t_i$.
\end{remark}
\section{Local ubiquity for rectangles}\label{s5}

In this section, we give some examples satisfying the local ubiquity property for rectangles and some consequences of the local ubiquity property to be used in proving Theorem \ref{t2}.

\subsection{Systems with ubiquity property}

The following examples should be expected to satisfy the ubiquity property for rectangles when equipped with Minkowski's theorem.

\subsubsection{Simultaneous Diophantine approximation}\ \smallskip

Simultaneous Diophantine approximation concerns the set \begin{align*}
&\Big\{(x_1,\cdots, x_m)\in [0,1]^m: \|qx_1\|<\psi_1(q),\cdots, \|qx_m\|<\psi_{m}(q), \ {\text{i.m.}}\ q\in \N\Big\}\\
=&\Big\{(x_1,\cdots, x_d)\in [0,1]^m: \big|x_i-\frac{p_i}{q}|<\frac{\psi_i(q)}{q},\ 1\le i\le m, \ {\text{i.m.}}\ (q, p_1,\cdots, p_m)\in \N\times \Z^{m}\Big\}.\end{align*}

Thus one has\begin{itemize}\item the index set $J$: $$
J=\Big\{\alpha=(q; p_1,\cdots,p_m): q\in \N, \ 0\le p_i\le q, 1\le i\le m\Big\},
$$
\item the resonant sets $\R_{\alpha}$: $$\R_{\alpha}=\Big(\frac{p_1}{q}, \cdots, \frac{p_m}{q}\Big), \ \ {\text{for}}\ \alpha=(q; p_1,\cdots, p_m).$$

\item weight function $\beta_{\alpha}$: $$\beta_{\alpha}=q, \ \ {\text{for}}\ \alpha=(q; p_1,\cdots, p_m).$$

\item ubiquitous function $\rho$: $$\rho:\mathbb{R}^+\to \mathbb{R}^+: u\to u^{-1}.$$

\item let $M\ge 2^{3m+2}$ be a sufficiently large integer, and take $$
\ell_k=M^{k-1}, \ \ u_k=M^k, \ k\ge 1.
$$
\end{itemize}

\begin{proposition}\label{p5} Let $\a=(a_1,\cdots, a_m)$ with $a_i\ge 1$ ($1\le i\le m$) and $a_1+\cdots+a_m=m+1$. Then
$(\{\R_{\alpha}\}_{\alpha\in J}, \beta)$ is a ubiquitous system for rectangles with respect to $(\rho, \a)$.
Meanwhile, the $\kappa$-scaling property holds with $\kappa=0$.\end{proposition}
\begin{proof} This is a special case of the linear forms, so its proof is included in the next proposition (Prop. \ref{p3}).
\end{proof}

\subsubsection{Linear forms}\  \smallskip

For an integer vector $\q=(q_1,\cdots,q_n)$, we use $|\q|$ to denote $\max\{|q_i|: 1\le i\le n\}$. Linear form concerns the set \begin{align*}
W_{mn}(\psi_1,\cdots, \psi_m)=\Big\{x\in [0,1]^{m\times n}: \|q_1x_{i1}+q_2&x_{i2}+\cdots+q_nx_{in}\|<\psi_i(|\q|),\\
 & \q=(q_1,\cdots,q_n)\in \Z^n, 1\le i\le m\Big\}.
\end{align*}
Thus one has\begin{itemize}\item the index set $J$: $$
J=\Big\{\alpha=(q_1,\cdots, q_n; p_1,\cdots,p_m): (q_1,\cdots, q_n)\in \Z^n\setminus \{0\}, \ -n|\q|\le p_i\le n|\q|, 1\le i\le m\Big\},
$$
\item the resonant sets $\R_{\alpha}=(\R_{\alpha, 1},\cdots, \R_{\alpha,m})$: for $1\le i\le m$, $$\R_{\alpha, i}=\{(x_{i1},\cdots, x_{in}): q_1x_{i1}+\cdots+q_nx_{in}-p_i=0\}, \  \ \ {\text{for}}\ \alpha=(q_1,\cdots, q_n; p_1,\cdots,p_m).$$

\item weight function $\beta_{\alpha}$: $$\beta_{\alpha}=|\q|, \ \ {\text{for}}\ \alpha=(q_1,\cdots, q_n; p_1,\cdots,p_m).$$

\item ubiquitous function $\rho$: $$\rho:\mathbb{R}^+\to \mathbb{R}^+: u\to u^{-1}.$$

\item let $M\ge 2^{2m+1}\cdot n^m$ be a sufficiently large integer, and take $$
\ell_k=M^{k-1}, \ \ u_k=M^k, \ k\ge 1.
$$
\end{itemize}

\begin{proposition}\label{p3} Let $\a=(a_1,\cdots, a_m)$ with $a_i\ge 1$ ($1\le i\le m$) and $a_1+\cdots+a_m=m+n$. Then
$(\{\R_{\alpha}\}_{\alpha\in J}, \beta)$ is a ubiquitous system for rectangles with respect to $(\rho, \a)$.
Meanwhile, the $\kappa$-scaling property holds with $\kappa=1-1/n$.\end{proposition}
\begin{proof} Fix a point $x=(x_1,\cdots,x_m)\in [0,1]^{mn}$  with $x_i\in [0,1]^n$ ($1\le i\le m$) and $r>0$, consider the ball $$
B:=B(x,r)=\prod_{i=1}^mB(x_i, r):=\prod_{i=1}^mB_i.
$$
Since $M$ is large and choose $k$ sufficiently large, one can have \begin{align}\label{1}\frac{2^{2m}n^{m/2}}{M^n}\le \frac{1}{4}, \ \ \ \frac{n2^{m}3^{m+n}k\log M}{M^k}\le \frac{1}{4}r^m.\end{align}
Write $Q=u_k=M^k$. Consider the set $$
\widehat{B}=\bigcup_{\alpha=(\bold{q}, \bold{p}): \beta_{\alpha}<\ell_k}\prod_{i=1}^m\Delta(\R_{\alpha,i}, \frac{1}{|\q|Q^{a_i-1}})\cap B.
$$ The measure of $\widehat{B}$ is estimated as follows. \begin{align*}
|\widehat{B}|\le& \sum_{(q_1,\cdots,q_n): q< \ell_k, \max\{|q_i|: 1\le i\le n\}=q}\  \sum_{-nq\le p_1,\cdots, p_m\le nq}\prod_{i=1}^m\Big|\Delta\Big(\R_{\alpha,i}, \frac{1}{qQ^{a_i-1}}\Big)\cap B_i\Big|\\
=&\sum_{(q_1,\cdots,q_n): q< \ell_k, \max\{|q_i|: 1\le i\le n\}=q}\  \prod_{i=1}^m\sum_{-nq\le p_i\le nq}\Big|\Delta\Big(\R_{\alpha, i}, \frac{1}{qQ^{a_i-1}}\Big)\cap B_i\Big|.
\end{align*}
For the inner summation over $p_i$, since the hyperplanes $\R_{\alpha, i}$ are $1/\sqrt{n}q$ separated, the intersection is nonempty for at most $2\sqrt{n}rq+3$ many $p_i$. Note also that \begin{align*}
\sharp\{(q_1,\cdots,q_n)\in \Z^n: \max\{|q_i|\}_{i=1}^n=q\}&=2n(2q+1)^{n-1}\le n3^nq^{n-1}, \ \\
 \sharp\{(q_1,\cdots,q_n)\in \Z^n: |\q|<l\}&=(2l-1)^{n}\le 2^nl^n,\\
(a+b)^t\le 2^t (a^t&+b^t), \ {\text{for}}\ a, b, t>0.
\end{align*} Thus one has
\begin{align*}
|\widehat{B}|&\le \sum_{(q_1,\cdots,q_n): q< \ell_k, \max\{|q_i|: 1\le i\le n\}=q}\ \prod_{i=1}^m\left((2\sqrt{n}rq+3)\cdot \frac{r^{n-1}}{qQ^{a_i-1}}\right)\\
&\le r^{m(n-1)}\sum_{(q_1,\cdots,q_n): q< \ell_k, \max\{|q_i|: 1\le i\le n\}=q}\ \left(\frac{2^{2m}n^{m/2}r^m}{Q^n}+\frac{2^{m}3^m}{q^mQ^n}\right)\\
&\le r^{mn}\cdot\frac{2^{2m+n}n^{m/2}\ell_k^n}{Q^n}+r^{m(n-1)}\cdot\sum_{q<\ell_k}\frac{n2^{m}3^{m+n}q^{n-1}}{q^mQ^n}\\
&\le r^{mn}\cdot\frac{2^{2m+n}n^{m/2}\ell_k^n}{Q^n}+r^{m(n-1)}\cdot\sum_{q<\ell_k}\frac{n2^{m}3^{m+n}}{q^mQ}\\
&\le r^{mn}\cdot\frac{2^{2m+n}n^{m/2}}{M^n}+r^{m(n-1)}\cdot\frac{n2^{m}3^{m+n}\log Q}{Q}\le \frac{1}{2}r^{mn},
\end{align*} where the last inequality follows from (\ref{1}) by the choice of $Q$.

Recall Minkowski's theorem: for any $x\in B$, there exist integers $(q_1,\cdots,q_n)$ with $1\le |\q|\le Q$ and $p_1,\cdots, p_m\in \Z^m$ such that
$$\big|q_1x_{i1}+\cdots+ q_{n}x_{in}-p_i\big|<\frac{1}{Q^{a_i-1}}, \ 1\le i\le m.$$
 Thus by taking $\alpha=(q_1,\cdots, q_n; p_1,\cdots, p_m)$
\begin{equation}\label{fff9}
x\in \prod_{i=1}^m\Delta(\R_{\alpha, i}, \frac{1}{|\q|Q^{a_i-1}}).
\end{equation}
Recall the definition of $\widehat{B}$. So, for any $x\in B\setminus \widehat{B}$, the inclusion in (\ref{fff9}) holds for $\ell_k\le |\q|\le Q$. Thus, $$
x\in \bigcup_{\alpha: \ell_k\le \beta_{\alpha}\le Q}\prod_{i=1}^{d}\Delta(\R_{\alpha, i}, \frac{M}{Q^{a_i}}\Big)=\bigcup_{\alpha: \ell_k\le \beta_{\alpha}\le u_k}\Delta\Big(\R_{\alpha}, M\cdot\rho(u_k)^{\a}\Big).
$$ As a conclusion, $$
\left|\bigcup_{\alpha: \ell_k\le \beta_{\alpha}\le u_k}\Delta\Big(\R_{\alpha}, M\cdot\rho(u_k)^{\a}\Big)\cap B\right|\ge |B\setminus \widehat{B}|\ge \frac{1}{2}|B|.
$$

The $\kappa$-scaling property is trivial.
\end{proof}

 \subsubsection{Shrinking target problems}\label{s5.1.3}\ \smallskip

For any $d$ integers $b_1,\cdots, b_d\ge 2$, define a transformation $T:[0,1]^d\to [0,1]^d$ by $$
 T(x_1,\cdots, x_d)=\Big(b_1x_1 (\text{mod}\ 1),\cdots, b_dx_d (\text{mod}\ 1)\Big).
$$ Then consider the following shrinking target problem: for $x_0\in [0,1]^d$
$$\Big\{x\in [0,1]^d: |T^nx-x_0|<\psi(n), \ {\text{i.m.}}\ n\in \N\Big\},$$
which is a special case considered by Hill \& Velani \cite{HV99} for shrinking target problems on torus actioned by integral matrix.

We go beyond this by considering the case in Cantor set. Let $b_1,\cdots, b_d\ge 2$ be $d$ integers. Let $$\Lambda_{b_i}\subset \{0,1,\cdots, b_i-1\}, \ {\text{and}} \ \ \sharp \Lambda_{b_i}\ge 2, \ \ {\text{for}}
\ \ 1\le i\le d.$$ Then let $\C_i$ be the Cantor sets defined by the iterated function systems $$
\Big\{g_{b_i,k}(x)=\frac{x+k}{b_i}, \ x\in [0,1], \  k\in \Lambda_{b_i}\Big\}.
$$ The natural Cantor measure $\mu_i$ supported on $\C_i$ is Ahlfors regular \cite{Hutch} with exponent $\delta_i=\frac{\log \sharp \Lambda_i}{\log b_i}$.

For $d$ positive functions $\psi_i: \mathbb{R}^{+}\to \mathbb{R}^+$ ($1\le i\le d$), define
$$M_c(\psi):=\Big\{(x_1,\cdots, x_d)\in \prod_{i=1}^d\mathcal{C}_i: \|b_i^nx_i-x_{o,i}\|<\psi_i(n), \ {\text{i.m.}}\ n\in \N\Big\}, \ \ x_o\in \prod_{i=1}^d\mathcal{C}_i.
$$

We will use the symbolic representations of the points $x_i$ in $\C_i$. For each $v_i=(\epsilon_1,\cdots, \epsilon_n)\in \Lambda_{b_i}^n$ with $n\ge 1$, write $$
I_{n,b_i}(v_i):=g_{b_i, \epsilon_1}\circ g_{b_i, \epsilon_2}\circ \cdots \circ g_{b_i, \epsilon_n}[0,1], \ \ x_{i}(v_i)=\frac{\epsilon_1}{b_i}+\cdots+\frac{\epsilon_n+x_{o,i}}{b_i^n},
$$ in other words $I_{n,b_i}(v_i)$ is an $n$th order cylinder respect to $\mathcal{C}_i$ and $x_i(v_i)$ is the $n$th inverse image of $x_{o,i}$ in $I_{n,b_i}(v_i)$. Note that for any $v_i$ there is an inverse image of $x_{o,i}$ in  $I_{n,b_i}(v_i)$ and the length of $I_{n,b_i}(v_i)$ is $b_i^{-n}$.

Clearly the set $M_c(\psi)$ can be rewritten as \begin{align*}
M_c(\psi)&=\Big\{x\in \prod_{i=1}^d\mathcal{C}_i: \big|x_i-x_{i}(v_i)\big|<\frac{\psi_i(n)}{b_i^n}, \ v_i\in \Lambda_{b_i}^n, 1\le i\le d, \ {\text{i.m.}}\ n\in \N\Big\}
\end{align*}

Thus one has\begin{itemize}\item the index set $J$: $$
J=\Big\{\alpha=(v_1,\cdots, v_d)\in \prod_{i=1}^d\Lambda_{b_i}^n: n\ge 1\Big\},
$$
\item the resonant sets $\R_{\alpha}$: $$\R_{\alpha}=\Big(x_{1}(v_1),\cdots, x_{d}(v_d)\Big), \  \ \ {\text{for}}\ \alpha=(v_1,\cdots, v_d).$$

\item weight function $\beta_{\alpha}$: $$\beta_{\alpha}=n, \ \ {\text{for}}\ \alpha=(v_1,\cdots, v_d)\in \prod_{i=1}^d\Lambda_{b_i}^n.$$

\item ubiquitous function $\rho$: $$\rho:\mathbb{R}^+\to \mathbb{R}^+: u\to e^{-u}.$$

\item take $$
\ell_n=u_n=n, \ n\ge 1.
$$
\end{itemize}
\begin{proposition}\label{p4} Let $\a=(\log b_1,\cdots, \log b_d)$. The pair
$(\{\R_{\alpha}\}_{\alpha\in J}, \beta)$ is a ubiquitous system for rectangles with respect to $(\rho, \a)$.
Meanwhile, the $\kappa$-scaling property holds with $\kappa=0$.\end{proposition}
\begin{proof} This is rather simple since $$
\bigcup_{\ell_n\le \beta_{\alpha}\le u_n}\Delta(\R_{\alpha}, \rho(u_n)^{\a})=\bigcup_{v_i\in \Lambda_{b_i}^n, 1\le i\le d}\ \prod_{i=1}^dB\Big(x_{i}(v_i), b_i^{-n}\Big)=\prod_{i=1}^d\mathcal{C}_i.
$$
\end{proof}

%

\subsection{Consequence of ubiquity for rectangles}\label{s5.2}
Now we state and prove some auxiliary results for latter use.
The first one is about the $5r$-covering lemma for rectangles. Generally speaking, there are no such covering lemmas for arbitrary rectangles compared with for balls. But if the rectangles have special forms, it does.
\begin{lemma}[$5r$-covering lemma for rectangles]\label{lll1} Let $\a=(a_1,\cdots, a_d)\in (\mathbb{R}^+)^d$. Let $\mathcal{E}$ be a collection of rectangles in $X$ with the form $$
\prod_{i=1}^d B(x_i, r^{a_i}), \ \ \ \ (x_1,\cdots, x_d)\in X, \ r>0.
$$ Then there is a sub-collection $\mathcal{E}'$ of $\mathcal{E}$ such that \begin{itemize}
\item separation condition: for any two different rectangles in $\mathcal{E}'$, saying $$
R=\prod_{i=1}^d B(x_i, r^{a_i}), R'=\prod_{i=1}^d B(x_i', {r'}^{a_i}),
$$ one has $5R\cap 5R'=\emptyset$.

\item covering property:
$$
\bigcup_{R\in \mathcal{E}}R\subset \bigcup_{R\in \mathcal{E}'}5R.$$
\end{itemize}\end{lemma}

The following is an analogy of the $K_{G,B}$-lemma a well known property essentially in proving  the mass transference principle from balls to balls \cite{BV06}.
\begin{lemma}[$K_{G,B}$-lemma]\label{lll2}
Assume the local ubiquity condition for rectangles (\ref{ff5}). Let $B$ be a ball in $X$ and $G\in \N$. For infinitely many $n\in \N$ with $n\ge G$, there exists a finite sub-collection $K_{G,B}$ of the rectangles \begin{equation}\label{a1}
\Big\{\widetilde{R}=\prod_{i=1}^dB(z_i, \rho(u_n)^{a_i}): (z_1,\cdots, z_d)\in \R_{\alpha}, \ \alpha\in  J_n,\ \widetilde{R}\cap 1/2B\ne \emptyset\Big\}
\end{equation} such that \begin{itemize}
\item all the rectangles in $K_{G,B}$ are contained in $2/3B$;

\item the rectangles are $3r$-disjoint in the sense that for any different elements in $K_{G,B}$
$$
3\prod_{i=1}^dB(z_i, \rho(u_n)^{a_i})\cap 3\prod_{i=1}^dB(z_i', \rho(u_n)^{a_i})=\emptyset;
$$

\item these rectangles almost pack the ball $B$ in the sense that, for a universal constant $c'>0$ depending only on the constant in the ubiquity property (\ref{ff5}), $$
m\left(\bigcup_{\widetilde{R}\in K_{G,B}}\widetilde{R}\right)\ge c' m(B).
$$
\end{itemize}
\end{lemma}
\begin{proof}
Clearly the rectangles in (\ref{a1})
 cover the set $$
1/2B\cap \bigcup_{\alpha\in J_n} \Delta (\R_{\alpha}, \rho(u_n)^{\a}).
$$

 Apply the ubiquity property to $\frac{1}{2}B$ and $5r$-covering lemma for rectangles, so for infinitely many $n$, there is a finite subcollection, denoted by $K_{G,B}$, of the rectangles in (\ref{a1}) such that  \begin{align*}
c\cdot m(\frac{1}{2}B)&\le m\left(\frac{1}{2}B\cap \bigcup_{\alpha\in J_n} \Delta (\R_{\alpha}, \rho(u_n)^{\a})\right)
\le m\left(\bigcup_{\widetilde{R}\in K_{G,B}}5\widetilde{R}\right)\\
&\le \sum_{\widetilde{R}\in K_{G,B}}m(5\widetilde{R})\asymp \sum_{\widetilde{R}\in K_{G,B}}m(\widetilde{R}).
\end{align*}

All the rectangles $\widetilde{R}$ in $K_{G,B}$ are contained in $2/3B$ when $n$ sufficiently large, since $\widetilde{R}\cap 1/2B\ne \emptyset$ and the sidelengths of $\widetilde{R}$ tend to $0$.
\end{proof}

The next one is an application of the $\kappa$-scaling property.
\begin{lemma}[Shrinking lemma]\label{l2} Let $\alpha\in J$ and $z=(z_1,\cdots,z_d)\in X$ with $z_i\in \R_{\alpha,i}$ for each $1\le i\le d$.  Let $\widetilde{R}$ be a rectangle centered at $z$ with sidelengths $(r^{a_1},\cdots,r^{a_d})$ with $r$ small. Then there exist a finite collection $\D(\widetilde{R})$ of rectangles with sidelengths $(r^{a_1+t_1}, \cdots, r^{a_d+t_d})$ satisfying that \begin{itemize}
\item all the rectangles are contained in \begin{equation*}
\widetilde{R}\cap \Delta(\R_{\alpha}, r^{\a+\bold{t}});
\end{equation*}

\item any two different rectangles in $\D(\widetilde{R})$ are $5r$-separated;

\item the number of the elements in $\D(\widetilde{R})$ is about $$
\sharp \D(\widetilde{R})\asymp \prod_{i=1}^d\frac{1}{r^{\delta_it_i\kappa}}.
$$
\end{itemize}
\end{lemma}\begin{proof}
Cover the set $1/2\widetilde{R}\cap \Delta(\R_{\alpha}, r^{\a+\bold{t}})$ by rectangles \begin{equation}\label{f6}
\prod_{i=1}^dB(y_i, r^{a_i+t_i}), \ \ (y_1,\cdots, y_d)\in \R_{\alpha}.
\end{equation}
Then using $5r$-covering lemma again to get a collection $\D(\widetilde{R})$ of well separated rectangles with the form as the one in (\ref{f6}). By a volume argument, one has\begin{align*}
\sharp\D(\widetilde{R})\cdot \prod_{i=1}^d r^{(a_i+t_i)\delta_i}&\asymp m\Big(\frac{1}{2}\widetilde{R}\cap \prod_{i=1}^d \Delta(\R_{\alpha, i}, r^{a_i+t_i})\Big)
\asymp \prod_{i=1}^d r^{a_i\delta_i \kappa}\cdot r^{(a_i+t_i)\delta_i (1-\kappa)},
\end{align*} where the $\kappa$-scaling property is used for the second relation $\asymp$. Thus $$
\sharp \D(\widetilde{R})\asymp \prod_{i=1}^d\frac{1}{r^{t_i\delta_i \kappa}}.
$$
\end{proof}

The last one is about the number of balls when a rectangle is divided into balls. Assume $a_d+t_d\ge a_i+t_i$ for all $1\le i\le d$.
\begin{lemma}[Division lemma]\label{l3}
Let $R$ be a rectangle with sidelengths $(r^{a_1+t_1},\cdots,r^{a_d+t_d})$. Then we have a collection $\C(R)$ of balls with radius $r^{a_d+t_d}$
satisfying that \begin{itemize}
\item all the balls are contained in $R$;

\item any two different balls in $\C(R)$ are $5r$-separated;

\item the number of the elements in $\C(R)$ is about $$
\sharp \C(R)\asymp \prod_{i=1}^d\left(\frac{r^{a_i+t_i}}{r^{a_d+t_d}}\right)^{\delta_i}.
$$
\end{itemize}
\end{lemma}
\begin{proof}
Divide the rectangle $R$ into balls with radius $r^{a_d+t_d}$. Then using $5r$-covering lemma and a volume estimate.
\end{proof}

The route of the above lemmas is: for a given ball $B$, \begin{align*}
B\overset{K_{G,B}}{--\longrightarrow}&\ K_{G,B}: \ {\text{big rectangles}}\ \widetilde{R}=\prod_{i=1}^dB(z_i, r^{a_i})\\
&\overset{{\text{shrinking: intersect with}}\ \Delta(R_{\alpha}, r^{\a+\bold{t}})}{--------------\longrightarrow}\D(\widetilde{R}):\ {\text{shrinking rectangles}}\ {R}=\prod_{i=1}^dB(y_i, r^{a_i+t_i})\\ &\overset{{\text{division}}}{---\longrightarrow}\C(R): \ {\text{balls}}\ B(\star,r^{a_d+t_d}).
\end{align*}


\section{Proof: Hausdorff measure under ubiquity}

The method to determine the Hausdorff measure of $W(\t)$ is quite classical and similar ideas are also well applied in for examples \cite{BV06}, \cite{Bug}. At first, we construct a Cantor subset $\F_{\infty}$ of $W(\t)$; secondly, define a suitable mass distribution $\mu$ supported on $\F_{\infty}$; thirdly, estimate the $\mu$-measure of a general ball; and at last, we conclude the result by applying the following mass distribution principle.

\begin{proposition}[Mass Distribution Principle \cite{Fal}]\label{p1}
Let $\mu$ be a probability measure supported on a measurable set $F$. Suppose there are positive constants $c$ and $r_o$ such that
$$\mu(B(x,r))\le c r^s$$
for any ball $B(x,r)$ with radius $r\le r_o$ and center $x\in F$. Then $\mathcal{H}^s(F)\ge 1/c$ and so $\hdim F\ge s$.
\end{proposition}

%
%

We assume that $\max\{t_i: 1\le i\le d\}>0$, since otherwise under the ubiquity condition $$
m\left(\limsup_{\alpha\in J, \beta_{\alpha}\to \infty}\Delta(\R_{\alpha}, \rho(\beta_{\alpha})^{\a+\bold{t}})\right)=m\left(\limsup_{\alpha\in J, \beta_{\alpha}\to \infty}\Delta(\R_{\alpha}, \rho(\beta_{\alpha})^{\a})\right)=m(X),
$$ and for any ball $B$ in $X$, $$
\mathcal{H}^s(B)\asymp m(B), \ {\text{when}}\ s=\sum_{i=1}^d \delta_i.
$$

\subsection{Cantor subset construction}\

For notational reasons, we will use $C_n$ to denote a generic ball in the $n$th level of the Cantor set to be constructed and use $B$ to denote a ball appearing in the intermediate process of the construction.

Let $s=s(\bold{t})$ be the dimensional number given in Theorem \ref{t1}. Since $$
\max\{t_i: 1\le i\le d\}>0\Longrightarrow s<\sum_{i=1}^d\delta_k,
$$ thus for any ball $C$ in $X$, $$\mathcal{H}^s(C)=\infty.$$ So to prove Theorem \ref{t2}, it suffices to show that for any ball $C_0$ $$
\mathcal{H}^s(W(\t)\cap C_0)=\infty.
$$

From now on, we
fix a positive number $\eta>0$ and a ball $C_0$ in $X$. Recall that $$
J_n=\{\alpha\in J: \ell_n\le \beta_{\alpha}\le u_n\},
\ \
\mathcal{A}=\{a_1,\cdots, a_d, a_1+t_1,\cdots, a_t+t_d\},
$$  and we assume that $a_1$ is the smallest one in $\mathcal{A}$ and $a_d+t_d$ the largest one.\medskip


{\em The first level $\F_1$}.\smallskip

The first level $\F_1=\F_1(C_0)$ consists of a collection of sublevels $\{\F_1(C_0,\ell): \ell\ge 1\}$. We define $\F_1(C_0, 1)$ at first.

\begin{itemize}
\item {\em Step $1_1$: $\mathcal{U}(1, C_0, 1)$: a collection of balls almost packing $C_0$} and also those balls where $K_{G,B}$-lemma will be applied. Let $$\mathcal{U}(1, C_0, 1)=\{C_0\}.$$ It is trivial in this case that \begin{equation}\label{f4}
    \sum_{B\in \mathcal{U}(1, C_0, 1)}m(B)\asymp m(C_0).
    \end{equation}

\item {\em Step $2_1$. Use $K_{G,B}$-lemma}. For each ball $B\in \mathcal{U}(1, C_0, 1)$, apply the $K_{G,B}$-lemma to $B$ to obtain a collection of well separated rectangles with the form $$
\widetilde{R}=\prod_{i=1}^dB(z_i, \rho(u_{n_1})^{a_i}),$$  centered at the point in some resonant set $\R_{\alpha}$  with ${\alpha}\in J_{n_1}.$
We write $r_1=\rho(u_{n_1})$. Also we have \begin{equation}\label{f5}
    \sum_{\widetilde{R}\in K_{G,B}}m(\widetilde{R})\asymp m(B).
    \end{equation}

\item {\em Step $3_1$. Shrinking}. For each rectangle $\widetilde{R}\in K_{G,B}$, let $\alpha\in J_{n_1}$ be the index such that the center $(z_1,\cdots,z_d)$ of $\widetilde{R}$ sits
in the resonant set $\R_{\alpha}$. Now consider the intersection $$
\widetilde{R}\cap \Delta(\R_{\alpha}, r_1^{\a+\bold{t}}).
$$

By Lemma \ref{l2}, we obtain a collection $\D(\widetilde{R})$ of $5r$-separated rectangles with the form $$
R=\prod_{i=1}^dB(y_i, r_1^{a_i+t_i}), \ \ {\text{with}}\  \ (y_1,\cdots, y_d)\in \R_{\alpha},
$$ moreover its cardinality satisfies $$
\sharp \D(\widetilde{R})\asymp \prod_{i=1}^d\frac{1}{r_1^{t_i\delta_i \kappa}}
$$ and \begin{align}\label{f7}
    \sum_{R\in \D(\widetilde{R})}m(R)&\asymp m(\widetilde{R}\cap \Delta(\R_{\alpha}, r_1^{\a+\bold{t}}))=m(\widetilde{R}\cap \prod_{i=1}^d\Delta(\R_{\alpha,i}, r_1^{a_i+t_i}))\nonumber\\ &\asymp \prod_{i=1}^d r_1^{a_i\delta_i\kappa}r_1^{(a_i+t_i)\delta_i(1-\kappa)}
    =m(\widetilde{R})\cdot \prod_{i=1}^d r_1^{t_i\delta_i(1-\kappa)}.
    \end{align}
We call the rectangle $\widetilde{R}$ in $K_{G,B}$ as {\em big rectangle} and the small rectangle $R$ in $\D(\widetilde{R})$ as {\em shrunk rectangle}.

\item {\em Step $4_1$. Dividing.} For each shrunk rectangle $R\in \D(\widetilde{R})$, 
applying Lemma \ref{l3} to get a collection $\C(R)$ of $5r$-separated balls $C_1$ with radius $r_1^{a_d+t_d}$ contained in $R$, moreover the cardinality of $\C(R)$ satisfies $$
\sharp \C(R)\asymp \prod_{i=1}^d\left(\frac{r_1^{a_i+t_i}}{r_1^{a_d+t_d}}\right)^{\delta}.
$$

Then the first sublevel is defined as \begin{equation*}
\F_1(C_0,1)=\bigcup_{B\in \mathcal{U}(1, C_0,1)}\ \bigcup_{\widetilde{R}\in K_{G,B}}\ \bigcup_{R\in \D(\widetilde{R})}\ \bigcup_{C_1\in \C(R)}\ C_1.
\end{equation*}

To render the possibility of the construction of the next sublevel, we will show the balls in $\F_1(C_0,1)$ only takes a small proportion inside $C_0$. So, the last step is about a volume estimation.

\item {\em Step $5_1$. Volume estimation.}
    By the formulas (\ref{f7}), (\ref{f5}) and (\ref{f4}), one has\begin{align*}
    &m\left(\bigcup_{B\in \mathcal{U}(1, C_0,1)}\ \bigcup_{\widetilde{R}\in K_{G,B}}\ \bigcup_{R\in \D(\widetilde{R})}6R\right)
    \\
    &\le\sum_{B\in \mathcal{U}(1, C_0,1)}\ \sum_{\widetilde{R}\in K_{G,B}}\ \sum_{R\in \D(\widetilde{R})}m(6R)
    \asymp \sum_{B\in \mathcal{U}(1, C_0,1)}\ \sum_{\widetilde{R}\in K_{G,B}}\ \sum_{R\in \D(\widetilde{R})}m(R)\\
   & =\sum_{B\in \mathcal{U}(1, C_0,1)}\ \sum_{\widetilde{R}\in K_{G,B}}m(\widetilde{R})\cdot \prod_{i=1}^d r_1^{t_i\delta_i(1-\kappa)}
    \asymp \sum_{B\in \mathcal{U}(1, C_0,1)} m(B)\cdot \prod_{i=1}^d r_1^{t_i\delta_i(1-\kappa)}\\ &\asymp m(C_0) \prod_{i=1}^d r_1^{t_i\delta_i(1-\kappa)}.
    \end{align*}
Since $\max\{t_i: 1\le i\le d\}>0$ and we can ask $u_{n_1}$ so large that $r_1$ is sufficiently small, then we can have \begin{align*}
    m\left(\bigcup_{B\in \mathcal{U}(1, C_0,1)}\ \bigcup_{\widetilde{R}\in K_{G,B}}\ \bigcup_{R\in \D(\widetilde{R})}6R\right)\le \frac{1}{4}m(C_0).\end{align*}
   \end{itemize}

    This finishes the construction of the first sublevel. Now we use an induction to construct the next sublevel. Assume that the sublevels $$
    \F_1(C_0,1), \ \cdots, \ \ \F_{1}(C_0,\ell-1)
    $$ have been well constructed. Each of them has the form, saying $$
    \F_1(C_0,l)=\bigcup_{B\in \mathcal{U}(1, C_0, l)}\ \bigcup_{\widetilde{R}\in K_{G,B}}\ \bigcup_{R\in \D(\widetilde{R})}\ \bigcup_{C_1\in \C(R)}\ C_1,
    $$
    and \begin{align}\label{f9}
      m\left(\bigcup_{B\in \mathcal{U}(1, C_0,l)}\ \bigcup_{\widetilde{R}\in K_{G,B}}\ \bigcup_{R\in \D(\widetilde{R})}6R\right)\le \frac{1}{4^l}m(C_0).
    \end{align}

    \begin{itemize}\item {\em Step $1_{\ell}$. $\mathcal{U}(1, C_0, \ell)$: a collection of balls almost pack $C_0$} and also those balls where $K_{G,B}$ lemma will be applied.
    Consider the difference set $$
    \widetilde{C}_0:=\frac{2}{3}C_0\setminus \left(\bigcup_{l=1}^{\ell-1}\bigcup_{B\in \mathcal{U}(1, C_0,l)}\ \bigcup_{\widetilde{R}\in K_{G,B}}\ \bigcup_{R\in \D(\widetilde{R})}6R\right),
    $$ i.e., we delete the parts inside $C_0$ near all the previous sublevels.
 By (\ref{f9}), there are still much room left inside $C_0$. More precisely, $$
m\left(\bigcup_{l=1}^{\ell-1}\bigcup_{B\in \mathcal{U}(1, C_0,l)}\ \bigcup_{\widetilde{R}\in K_{G,B}}\ \bigcup_{R\in \D(\widetilde{R})}6R\right)\le \sum_{l=1}^{\infty}\frac{1}{4^l}m(C_0)\le \frac{1}{3}m(C_0).
 $$

 For each $w\in \widetilde{C}_0$, sprout it into a ball $B(w,r)$ with the radius $r$ sufficiently small that $3r$ is smaller than the radius of any balls in the levels constructed before. Then we get a collection of balls \begin{equation}\label{g1}
 \Big\{B(w,r): w\in \widetilde{C}_0\Big\}
  \end{equation} which covers $\widetilde{C}_0$ and is contained in $C_0$. By the definition of $\widetilde{C}_0$, these balls are far away from any shrunk rectangles and balls in the previous sublevels, more precisely\begin{equation}\label{d6}
 B(w,r)\cap 5R=\emptyset, \ \ B(w,r)\cap 5C_1=\emptyset \end{equation} for any $$R\in \bigcup_{l=1}^{\ell-1}\bigcup_{B\in \mathcal{U}(1, C_0,l)}\ \bigcup_{\widetilde{R}\in K_{G,B}}\D(\widetilde{R})\ \ {\text{and}}\ \ \  C_1\in \mathcal{C}(R).
 $$

 By $5r$-covering lemma, we have a finite sub-collection, denoted by $\mathcal{U}(1, C_0, \ell)$, of the balls in (\ref{g1}) satisfying that \begin{itemize}\item the balls in $\mathcal{U}(1, C_0, \ell)$ are $5r$-separated;

 \item they almost pack $C_0$ in the sense that $$
 \sum_{B\in \mathcal{U}(1, C_0,\ell)}m(B)\asymp m(\widetilde{C}_0)\asymp m(C_0).
 $$ \end{itemize}

 The other steps are the same as in the construction of $\F_1(C_0,1)$. We only give the outline.

 \item {\em Step $2_{\ell}$. Use $K_{G,B}$-lemma}. For each ball $B\in \mathcal{U}(1, C_0, \ell)$, use $K_{G,B}$-lemma to get a collection of $5r$-separated big rectangles $\widetilde{R}$ satisfying $$
     \sum_{\widetilde{R}\in K_{G,B}}m(\widetilde{R})\asymp m(B). 
     $$
These big rectangles $\widetilde{R}$ have the same sidelengths, saying $(\hat{r}_1^{a_1}, \cdots \hat{r}_1^{a_d})$, with $\hat{r}_1=\rho(u_n)$ for some $n$ as large as we want.

     \item {\em Step $3_{\ell}$. Shrinking}. For each big rectangle $\widetilde{R}=\prod_{i=1}^dB(z_i, \hat{r}_1^{a_i})$, cover the intersection $$
     \widetilde{R}\cap \Delta(\R_{\alpha}, \hat{r}_1^{\a+\bold{t}}), \ \ (\text{assume}\ (z_1,\cdots, z_d)\in \R_{\alpha}\ {\text{for some}}\ \alpha\in J_n),
     $$ by smaller rectangles $$
     \prod_{i=1}^dB(y_i, \hat{r}_1^{a_i+t_i}), \ \ (y_1,\cdots, y_d)\in \R_{\alpha}
     $$ to get the collection $\D(\widetilde{R})$ of shrunk rectangles (Lemma \ref{l2}).

     \item {\em Step $4_{\ell}$. Dividing.} For each $R\in \D(\widetilde{R})$, cut it into balls of radius $\hat{r}_1^{a_d+t_d}$ to get $\C(R)$ (Lemma \ref{l3}).

     This gives the $\ell$th sublevel:$$
     \F_1(C_0,\ell)=\bigcup_{B\in \mathcal{U}(1, C_0, \ell)}\ \bigcup_{\widetilde{R}\in K_{G,B}}\ \bigcup_{R\in \D(\widetilde{R})}\ \bigcup_{C_1\in \C(R)}C_1.
     $$

     \item {\em Step $5_{\ell}$. Volume estimation.}  Since $\hat{r}_1$ can be rather small, \begin{align*}
    &m\left(\bigcup_{B\in \mathcal{U}(1, C_0,\ell)}\ \bigcup_{\widetilde{R}\in K_{G,B}}\ \bigcup_{R\in \D(\widetilde{R})}6R\right)
   \asymp \sum_{B\in \mathcal{U}(1, C_0,\ell)}\ \sum_{\widetilde{R}\in K_{G,B}}\ \sum_{R\in \D(\widetilde{R})}m(R)\\
   & \asymp\sum_{B\in \mathcal{U}(1, C_0,\ell)}\ \sum_{\widetilde{R}\in K_{G,B}}m(\widetilde{R})\cdot \prod_{i=1}^d \hat{r}_1^{t_i\delta_i(1-\kappa)}
    \le \frac{1}{4^{\ell}} \sum_{B\in \mathcal{U}(1, C_0,\ell)} m(B)\\ &\le \frac{1}{4^{\ell}} m(C_0).
    \end{align*}

     This finishes the construction of the first level: let $L_{C_0}=\eta\cdot m(C_0)^{-1}$ and set $$
     \F_1=\F_1(C_0)=\bigcup_{\ell=1}^{L_{C_0}}\F_1(C_0,\ell)=\bigcup_{\ell=1}^{L_{C_0}}\ \bigcup_{B\in \mathcal{U}(1, C_0, \ell)}\ \bigcup_{\widetilde{R}\in K_{G,B}}\ \bigcup_{R\in \D(\widetilde{R})}\ \bigcup_{C_1\in \C(R)}\ C_1.
     $$
%
    \end{itemize}

    {\em The general level $\F_k$.}

    Assume that $\F_1, \cdots, \F_{k-1}$ have been well constructed, which is a collection of balls. Define $$
    \F_k=\bigcup_{C_{k-1}\in \F_{k-1}}\F_k(C_{k-1}),\ {\text{and}} \ \F_k(C_{k-1})=\bigcup_{\ell=1}^{L_{C_{k-1}}}\F_k(C_{k-1}, \ell).
    $$
   Replace the role of $C_0$ by $C_{k-1}$ and repeat the construction of $\F_1(C_0)$ to get $\F_k(C_{k-1})$.
Generally, the route is \begin{align*}
C_{k-1}\overset{{\text{cover}}}{----\longrightarrow} B\in \mathcal{U}(k, C_{k-1}, \ell)&\overset{{K_{G,B}}}{----\longrightarrow}\widetilde{R}\in K_{G,B}\\ &\overset{{\text{shrinking}}}{----\longrightarrow} R\in \D(\widetilde{R})\overset{{\text{division}}}{----\longrightarrow}C_k\in \C(R).
\end{align*}
So the sublevel of $\F_k$ is of the form: $$
\F_{k}(C_{k-1})=\bigcup_{\ell=1}^{L_{C_{k-1}}}\bigcup_{B\in \mathcal{U}(k, C_{k-1}, \ell)}\bigcup_{\widetilde{R}\in K_{G,B}}\bigcup_{R\in \mathcal{D}(\widetilde{R})}\bigcup_{C_k\in \mathcal{C}(R)}C_k,
$$The integer $L_{C_{k-1}}$ is chosen such that $$L_{C_{k-1}}m(C_{k-1})=r_{C_{k-1}}^s$$ where $s=s(\bold{t})$ is given in Theorem \ref{t1}.

Finally, the desired Cantor set is defined as $$
\F_{\infty}=\bigcap_{k=1}^{\infty}\bigcup_{C_k\in \F_k}C_k,
$$ which is clearly a subset of $W(\t)$, since the shrunk rectangles $R$ in the $k$th level is contained in $$
\Delta(\R_{\alpha}, \rho(u_{n_k})^{\a+\bold{t}})\subset \Delta(\R_{\alpha}, \rho(\beta_{\alpha})^{\a+\bold{t}}),
$$ for some $\alpha\in J_{n_k}$.

We give a summary on the properties of the Cantor set constructed above for latter use.
\begin{proposition}\label{p2}
Let $C_{k-1}$ be an element in $\F_{k-1}$ and $\mathcal{U}(k,C_{k-1}, \ell)$ be the collection of balls almost packing $C_{k-1}$ appearing in the construction of the local sublevel $\F_k(C_{k-1},\ell)$. \begin{itemize}
 \item For each $1\le \ell\le L_{C_{k-1}}$, \begin{equation}\label{g2}
    \sum_{B\in \mathcal{U}(k, C_{k-1}, \ell)}\sum_{\widetilde{R}\in K_{G,B}}m(\widetilde{R})\asymp\sum_{B\in \mathcal{U}(k, C_{k-1}, \ell)}\ m(B)\asymp m(C_{k-1}),
    \end{equation}
    with\begin{equation}\label{d7}  L_{C_{k-1}}m(C_{k-1})=r_{C_{k-1}}^s, \ L_{C_0}=\eta\cdot m(C_0)^{-1}.
    \end{equation}

       \item For each ball $B$ in $\mathcal{U}(k,C_{k-1}, \ell)$, it is disjoint with $5R$ for any shrunk rectangle $R$ in any previous sub-levels and its radius $r_{B}$ is much smaller than the radius of the balls in any previous sub-levels.

    \item Fix an element $B$ in $\mathcal{U}(k, C_k, \ell)$. \begin{itemize} \item The big rectangles $\widetilde{R}$ from $K_{G,B}$ are $5r$-separated, and \begin{equation}\label{f10}
\sum_{\widetilde{R}\in K_{G,B}}m(\widetilde{R})\asymp m(B);
\end{equation}

\item For each $\widetilde{R}\in K_{G,B}$, the shrunk rectangles $R$ in $\D(\widetilde{R})$ are of the same sidelengths $(r^{a_1+t_1},\cdots, r^{a_d+t_d})$ (for some $r$), $5r$-separated and $$
    \sharp \D(\widetilde{R})\asymp \prod_{i=1}^d\frac{1}{r^{t_i\delta_i \kappa}}.
    $$

\item For each $R\in \D(\widetilde{R})$, the dividing collection $\C(R)$ contains $$
\sharp \C(R)\asymp \prod_{i=1}^d\left(\frac{r^{a_i+t_i}}{r^{a_d+t_d}}\right)^{\delta_i}
$$ $5r$-separated balls with radius $r^{a_d+t_d}$.
\end{itemize}

        \item Let $C, C'$ be two different balls in $\F_{k}$, then they are at least $3r$-separated.
         So if a ball $B$ can intersect at least two elements in $\F_k$, say $C, C'$, then $r_B\ge r_C$ and all of them are contained in $3B$.
\end{itemize}
\end{proposition}
\begin{proof} The other items are clear.
We check the last item, which comes from a simple geometric observation. By induction, assume that there exists some $C_{k-1}\in \F_{k-1}$ such that $C, C'\in \F_k(C_{k-1})$, otherwise we consider their predecessors. Let $\ell\le \ell'$ be the integers such that $C\in \F_{k}(C_{k-1}, \ell)$ and $C'\in \F_k(C_{k-1}, \ell')$.
Use $r_C$ and $r_{C'}$ to denote the radius of $C$ and $C'$.
%
%
%

(1). When $\ell=\ell'$. Recall the process of which $C$ and $C'$ are generated: \begin{equation}\label{h2}\begin{array}{ccccccc}
                C\in \mathcal{C}(R)& \overset{1}{\leftarrow} & R\in \mathcal{D}(\widetilde{R}) &\overset{2}{\leftarrow} & \widetilde{R}\in K_{G,B} & \overset{3}{\leftarrow}& B\in \mathcal{U}(\ell, C_{k-1}, \ell) \\
                C'\in \mathcal{C}(R')& \overset{1}{\leftarrow} & R'\in\mathcal{D}(\widetilde{R}') &\overset{2}{\leftarrow} & \widetilde{R}'\in K_{G,B'} & \overset{3}{\leftarrow} & B'\in \mathcal{U}(\ell, C_{k-1}, \ell).
              \end{array}\end{equation}
              Note that at each stage, different elements having the same predecessor are $5r$-separated.
              \begin{itemize}\item If $C$ and $C'$ have different predecessors at each stage in (\ref{h2}), i.e. $B\ne B'$ which implies that $B$ and $B'$ are $5r$-separated, so is $C$ and $C'$. \item If $C$ and $C'$ have the same predecessors at some stage. Let $1\le i\le 3$ be the smallest integer such that $C$ and $C'$ have the same predecessor at this stage. Then the offsprings of this predecessor are $5r$-separated, so is $C$ and $C'$.\end{itemize}

(2). When $\ell<\ell'$. By the construction of $\mathcal{U}(k, C_{k-1}, \ell')$, we have $r_C\ge 3r_{C'}$ and $5C\cap C'=\emptyset$.
Let $z$ and $z'$ be the centers of $C$ and $C'$ respectively. Then $$
|z-z'|\ge 5r_C+r_{C'}\ge 3r_C+3r_{C'}.
$$
 \end{proof}

%
%
%
%
%
%

\subsection{Mass distribution}

The mass distribution to be defined comes from the following consideration: given a big rectangle $\widetilde{R}$ in the construction of $\F_{\infty}$, \begin{itemize} \item Since the measure is to be supported on $\F_{\infty}$, the total measure of the shrunk rectangles in $\D(\widetilde{R})$ should be equal to the measure of $\widetilde{R}$. However, the shrunk rectangles $R$ in $\D(\widetilde{R})$ are of the same size, so it is reasonable to distribute the mass of $\widetilde{R}$ equally to the shrunk rectangles $R$.

\item With the same reason, the mass of $R$ should be equally distributed on $\C(R)$;

\item By the volume estimation (see for example (\ref{f10})), the big rectangles $\widetilde{R}$ are disjoint and almost pack the whole space, so its mass will be defined as its $m$-measure with a suitable normalizer.
\end{itemize}

Let $\mu(C_0)=1$.
Let $R_k$ be a shrunk rectangle appearing in the construction of $\F_k(C_{k-1})$ for some  $C_{k-1}\in \F_{k-1}$.
 Let $\widetilde{R}_k$ be the big rectangle for which $R_k\in \D(\widetilde{R}_k)$. Write $R_k$ and $\widetilde{R}_k$ as $$
R_k=\prod_{i=1}^dB(y_i, r_k^{a_i+t_i}), \ \ \widetilde{R}_k=\prod_{i=1}^dB(z_i, r_k^{a_i}).
$$
Then define the measure on the big rectangle $\widetilde{R}_k$ as
\begin{align}\label{f13}
\mu(\widetilde{R}_k)&=\frac{m(\widetilde{R}_k)}{\sum_{\ell=1}^{L_{C_{k-1}}}\sum_{B\in \mathcal{U}(k, C_{k-1}, \ell)}\sum_{\widetilde{R}\in K_{G,B}}m(\widetilde{R})}\times \mu(C_{k-1})\nonumber\\
&\asymp  \frac{m(\widetilde{R}_k)}{L_{C_{k-1}}m(C_{k-1})}\times\mu(C_{k-1})\asymp\prod_{i=1}^dr_{k}^{a_i\delta_i}\times \frac{\mu(C_{k-1})}{L_{C_{k-1}}m(C_{k-1})},
\end{align} where the second relation follows from (\ref{g2}). Then define the measure on the shrunk rectangle $R_k$ as
\begin{align*}
\mu(R_k)=\frac{1}{\sharp \D(\widetilde{R}_k)}\cdot \mu(\widetilde{R}_k)
\asymp \prod_{i=1}^d r_{k}^{t_i\delta_i\kappa}\cdot\prod_{i=1}^dr_{k}^{a_i\delta_i}\times \frac{\mu(C_{k-1})}{L_{C_{k-1}}m(C_{k-1})}.
\end{align*}
At last for a ball $C_k$ in $\C(R_k)$, define \begin{align}\label{f11}
\mu(C_k)&=\frac{1}{\sharp \C(R_k)}\cdot \mu(R_k)\asymp \prod_{i=1}^d\left(\frac{r_k^{a_d+t_d}}{r_k^{a_i+t_i}}\right)^{\delta_i}\cdot \mu(R_k)\nonumber\\
&\asymp \prod_{i=1}^d\left(\frac{r_k^{a_d+t_d}}{r_k^{a_i+t_i}}\right)^{\delta_i}\cdot\prod_{i=1}^d r_k^{t_i\delta_i\kappa}\cdot\prod_{i=1}^dr_k^{a_i\delta_i}\times \frac{\mu(C_{k-1})}{L_{C_{k-1}}m(C_{k-1})}.
\end{align}

Then by Kolmogorov's consistency theorem, the set function $\mu$ can be uniquely extended into a probability measure supported on $\F_{\infty}$.

%

\subsubsection{Measure of balls in $\F_k$.}\

When $k=1$. Let $C_1$ be a ball in $\F_1$ with radius $r_1^{a_d+t_d}.$ Recall the choice of $L_{C_0}$ and (\ref{f11})
 one has \begin{align*}
\mu(C_1)\asymp \frac{1}{\eta}\cdot \prod_{i=1}^d\left(\frac{r_1^{a_d+t_d}}{r_1^{a_i+t_i}}\right)^{\delta_i}\cdot\prod_{i=1}^d r_{1}^{t_i\delta_i\kappa}\cdot\prod_{i=1}^dr_{1}^{a_i\delta_i}=\frac{r_1^{(a_d+t_d)s_d}}{\eta},
\end{align*} where $s_d$ is given as $$
s_d=\sum_{i=1}^d\delta_i-(1-\kappa)\frac{\sum_{i=1}^dt_i\delta_i}{a_d+t_d}.
$$
This dimensional number $s_d$ is just the one in Theorem \ref{t1} defined by choosing $A=a_d+t_d$, since $$
\K_1=\{k: a_k\ge a_d+t_d\}=\emptyset, \ \K_2=\{k: a_k+t_k\le a_d+t_d\}=\{1,\cdots, d\}.
$$ Thus for any $C_1\in \F_1$, $$
\mu(C_1)\le \frac{r_{C_1}^s}{\eta}.
$$

Assume that we have shown for all balls $C_{k-1}$ in $\F_{k-1}$, \begin{equation}\label{f12}
\mu(C_{k-1})\le \frac{r_{C_{k-1}}^s}{\eta}.
\end{equation}
Recall $L_{C_{k-1}}$ in (\ref{d7}). Now for a ball $C_k$ in $\F_k$ with radius $r_k^{a_d+t_d}$, by (\ref{f11}) and (\ref{f12}), one has \begin{equation}\label{f15}
\mu(C_k)\le \frac{1}{\eta}\prod_{i=1}^d\left(\frac{r_k^{a_d+t_d}}{r_k^{a_i+t_i}}\right)^{\delta_i}\cdot\prod_{i=1}^d r_k^{t_i\delta_i\kappa}\cdot\prod_{i=1}^dr_k^{a_i\delta_i}=\frac{r_{k}^{(a_d+t_d)s_d}}{\eta}.
\end{equation}
So  we have \begin{equation}\label{f14}
\mu(C_k)\le \frac{r_{C_k}^{s}}{\eta}.
\end{equation} 


\subsubsection{Measure of a general ball.}\

Now we consider the measure of a general ball $B(x,r)$ with $x\in \F_{\infty}$ and $r$ small enough. Let $k$ be the integer such that $B(x,r)$ intersects only one element in $\F_{k}$ and at least two elements in $\F_{k+1}$. Denote $C_k$ the unique element in $\F_k$ for which $B(x, r)$ can intersect.
Let $r_{k}^{a_d+t_d}$ be the radius of the ball $C_k$.
Without loss of generality, we can assume that $r\le r_k^{a_d+t_d}$, otherwise $$
\mu(B(x,r))\le \mu(C_k)\le \frac{r_{C_k}^s}{\eta}\le \frac{r^s}{\eta}.
$$

By the choice of $k$, all the elements in $\F_{k+1}$ which can intersect $B(x,r)$ are contained in the local level $$
\F_{k+1}(C_k)=\bigcup_{\ell=1}^{L_{C_k}}\F_{k+1}(C_k, \ell)=\bigcup_{\ell=1}^{L_{C_k}}\bigcup_{B\in \mathcal{U}(k+1, C_k, \ell)}\bigcup_{\widetilde{R}\in K_{G,B}}\bigcup_{R\in \D(\widetilde{R})}\bigcup_{C_{k+1}\in \mathcal{C}(R)}C_{k+1}.
$$

Let $\ell_0$ be the smallest integer such that there exists an element in $\F_{k+1}(C_k, \ell)$ intersecting $B(x,r)$.
By the construction of the sublevels $\F_{k+1}(C_k, \ell)$ ($\ell>\ell_0$) after $\F_{k+1}(C_k, \ell_0)$,
any ball $B\in \mathcal{U}(k+1, C_k, \ell)$ which can intersect $B(x,r)$ (if exist) will be contained in $B(x,2r)$. More precisely, let $R$ be a shrunk rectangle appearing in the construction of $\F_{k+1}(C_k, \ell_0)$ which intersects $B(x,r)$, then by (\ref{d6}) one has $$
B(x,r)\cap R\ne \emptyset, \ B(x,r)\setminus 5R\ne \emptyset.
$$ Let $v=(v_1,\cdots, v_d)\in B(x,r)\cap R$ and $v'=(v_1',\cdots, v_d')\in B(x,r)\setminus 5R$ and write $$
R=\prod_{i=1}^dB(y_i, r_{k+1}^{a_i+t_i}).
$$ Recall the metric on $X$: $|\cdot|=\max_{1\le i\le d}|\cdot|_i$. Then there exists $1\le i\le d$ such that $$
|v_i'-y_i|_i\ge 5r_{k+1}^{a_i+t_i}, \ \
|v_i-y_i|_i\le r_{k+1}^{a_i+t_i}.
$$ Moreover, by the construction of $\mathcal{U}(k+1, C_k ,\ell)$ with $\ell>\ell_0$, we know $3r_B\le r_{k+1}^{a_d+t_d}$ for any $B\in \mathcal{U}(k+1, C_k ,\ell)$. Thus $$
2r\ge |v-v'|\ge |v_i- v'_i|_i\ge 4r_{k+1}^{a_i+t_i}\ge 12r_B.
$$

Thus the contribution of the sublevels after $\F_{k+1}(C_k, \ell_0)$ to the measure of $B(x,r)$ can be estimated as: \begin{align*}
I_1:&=\mu\Big(B(x,r)\cap \bigcup_{\ell>\ell_0}\F_{k+1}(C_k, \ell)\Big)\le
\mu\Big(B(x,r)\cap \bigcup_{\ell>\ell_0}\bigcup_{B\in \mathcal{U}(k+1, C_k,\ell)}\ \bigcup_{\widetilde{R}\in K_{G,B}}\widetilde{R}\Big)\\
&\le \sum_{\ell>\ell_0}\sum_{B\in \mathcal{U}(k+1, C_k, \ell), B\cap B(x,r)\ne \emptyset}\ \ \ \sum_{\widetilde{R}\in K_{G,B}}\mu(\widetilde{R}).\end{align*}
Then by the first formula (\ref{f13}) on the measure of $\widetilde{R}$ and the measure estimation on balls in $\F_{\infty}$ (\ref{f14}), we have \begin{align*}I_1&\le \sum_{\ell>\ell_0}\sum_{B\in \mathcal{U}(k+1, C_k, \ell), B\cap B(x,r)\ne \emptyset}\ \ \ \sum_{\widetilde{R}\in K_{G,B}}{m(\widetilde{R})}\cdot\frac{r_{C_k}^s}{\eta\cdot L_{C_k}m(C_{k})}\\
&\le \sum_{\ell>\ell_0}\sum_{B\in \mathcal{U}(k+1, C_k, \ell), B\subset B(x,2r)}\ {m(B)}\cdot\frac{r_{C_k}^s}{\eta \cdot L_{C_k}m(C_{k})}.\end{align*} Since the balls in $\mathcal{U}(k+1, C_k, \ell)$ are disjoint, so \begin{align*}I_1&\le \sum_{\ell>\ell_0}m(B(x,2r))\cdot \frac{r_{C_k}^s}{\eta \cdot L_{C_k}m(C_{k})}\le \frac{m(B(x,2r))}{\eta}\cdot \frac{r_{C_k}^s}{m(C_{k})}\\
&\le \frac{c\cdot r^{\delta_1+\cdots+\delta_d}}{\eta}\cdot \frac{r_{C_k}^s}{r_{C_k}^{\delta_1+\cdots+\delta_d}}\le \frac{c\cdot r^s}{\eta},
\end{align*} since $s\le \delta_1+\cdots+\delta_d$ and $r\le r_{C_k}.$


So we only need focus on the contribution of the elements in $\F_{k+1}(C_k, \ell_0)$ to the measure of $B(x,r)$.

Case (1). The ball $B(x,r)$ can intersect at least two balls $B$ in $\mathcal{U}(k+1, C_k, \ell_0)$. By the $5r$-separation condition of these balls in $\mathcal{U}(k+1, C_k, \ell_0)$, we also have that those balls which intersect $B(x,r)$ are contained in $B(x,2r)$.
The same argument as above applies (without the summation over $\ell$), i.e. we still have
\begin{align*}
\mu\Big(B(x,r)\cap \F_{k+1}(C_k, \ell_0)\Big)\le \frac{c\cdot r^s}{\eta}.\end{align*}

Case (2). The ball $B(x,r)$ only intersects one ball in $\mathcal{U}(k+1, C_k, \ell_0)$, saying $B$. Thus \begin{align}\label{a3}
B(x,r)\cap \F_{k+1}(C_k,\ell_0)=B(x,r)\cap \bigcup_{\widetilde{R}\in K_{G,B}}\bigcup_{R\in \mathcal{D}(\widetilde{R})}\bigcup_{C_{k+1}\in \mathcal{C}(R)}C_{k+1}.
\end{align}Note that the big rectangles $\widetilde{R}$ in $K_{G,B}$, the shrunk rectangles $R$ in $\mathcal{D}(\widetilde{R})$ and the dividing balls $C_{k+1}$ in $\mathcal{C}(R)$ are of the same size respectively. So the generic ones appearing in (\ref{a3}) are denoted respectively by  \begin{equation}\label{a4}
\widetilde{R}=\prod_{i=1}^d B(z_i, r_{{k+1}}^{a_i}), \ \ {R}=\prod_{i=1}^d B(y_i, r_{{k+1}}^{a_i+t_i}),\  C_{k+1}=B(\star, r_{{k+1}}^{a_d+t_d}).
\end{equation}

(i). If $r\ge r_{{k+1}}^{a_1}$. In this case, all the big rectangles in $K_{G,B}$ intersecting $B(x,r)$ are contained in $B(x,3r)$.
 Thus, by the choice of $L_{C_k}$ (\ref{d7}) and the measure on $C_{k}$ (\ref{f14}), we have \begin{align*}
\mu\Big(B(x,r)\cap \F_{k+1}(C_k, \ell_0)\Big)&\le \sum_{\widetilde{R}\in K_{G,B}, \widetilde{R}\cap B(x,r)\ne \emptyset}\mu(\widetilde{R})\\&\le \sum_{\widetilde{R}\in K_{G,B}, \widetilde{R}\cap B(x,r)\ne \emptyset}m(\widetilde{R})\cdot \frac{\mu(C_k)}{L_{C_k}m(C_k)}\\
&\le \frac{1}{\eta}\sum_{\widetilde{R}\in K_{G,B}, \widetilde{R}\subset B(x,3r)}m(\widetilde{R})\\&\le \frac{m(B(x,3r))}{\eta}
\le \frac{c\cdot r^{s}}{\eta}.
\end{align*}

(ii). If $r< r_{{k+1}}^{a_1}$. Recall the last item in Proposition \ref{p2}. Since $B(x,r)$
can intersect at least two balls in $\F_{k+1}$ and at least one, saying $C_{k+1}$, in $\F_{k+1}(C_k, \ell_0)$,
so $$
r\ge r_{{k+1}}^{a_d+t_d}.
$$ In other words, we are in the situation $$
r_{{k+1}}^{a_d+t_d}\le r< r_{{k+1}}^{a_1}.
$$

Recall (\ref{a4}) for the generic form of the rectangles and balls in Case (2). Arrange the elements in $\mathcal{A}$ in non-decreasing order. Let $A_{l+1}$ and $A_l$ be the two different consecutive terms in $\mathcal{A}$ such that $$
r_{{k+1}}^{A_{l+1}}\le r< r_{{k+1}}^{A_l}.
$$
Now we consider how many balls in $\F_{k+1}(C_k, \ell_0)$ can intersect $B(x,r)$, indeed the balls in $$
\bigcup_{\widetilde{R}\in K_{G,B}}\bigcup_{R\in \D(\widetilde{R})}\bigcup_{C_{k+1}\in \mathcal{C}(R)}C_{k+1}.
$$

Define the sets $\K$: \begin{equation}\label{e2}
\K_1=\{i: a_i\ge A_{l+1}\}, \ \K_2=\{i: a_i+t_i\le A_l\},  \ \ \K_3=\{1,\cdots,d\}\setminus (\K_1\cup \K_2).
\end{equation}

Define an enlarged body of the ball $B(x,r)$: $$
H=\prod_{i=1}^dB(x_i, 3\epsilon_i), \ {\text{where}}\ \epsilon_i=\left\{
                                                              \begin{array}{ll}
                                                                r, & \hbox{$i\in \K_1$;} \\
                                                                r_{k+1}^{a_i}, & \hbox{otherwise.}
                                                              \end{array}
                                                            \right.
$$ Then for any big rectangle  $\widetilde{R}=\prod_{i=1}^dB(z_i, r_{k+1}^{a_i})\in K_{G,B}$ intersecting $B(x,r)$, we must have that $\widetilde{R}\subset H$. Since these big rectangles $\widetilde{R}$ are disjoint, then a volume argument gives the number of rectangles $\widetilde{R}$ in $K_{G,B}$ which can possibly intersect the ball $B(x,r)$:
\begin{equation}\label{e1}\frac{m(H)}{m(\widetilde{R})}\asymp \prod_{i\in \K_1}\left(\frac{r}{r_{{k+1}}^{a_i}}\right)^{\delta_i}.\end{equation}

Fix a generic big rectangle $\widetilde{R}$ which intersects $B(x,r)$. Let $\alpha\in J$ be the index such that the center of $\widetilde{R}$ sits in the resonant set $\R_{\alpha}$. We consider the number $T$ of balls \begin{equation}\label{d8}C\in \bigcup_{R\in \D(\widetilde{R})}\mathcal{C}(R)\end{equation} which can intersect $B(x,r)$. Since $B(x,r)$ can intersect at least two elements in $\F_{k+1}$, all these balls $C$ in (\ref{d8}) are contained in $$
\widetilde{R}\cap \Delta(\R_{\alpha}, r_{k+1}^{\a+\bold{t}})\cap B(x,2r).
$$ Still we use a volume argument: \begin{align*}
T\cdot \prod_{i=1}^d\left(r_{k+1}^{a_d+t_d}\right)^{\delta_i}&\le m\left(\widetilde{R}\cap B(x,2r)\cap \prod_{i=1}^d \Delta(\R_{\alpha, i}, r_{k+1}^{a_i+t_i})\right)\\
&=\prod_{i=1}^dm_i\Big(B(z_i, r_{k+1}^{a_i})\cap B(x_i,2r)\cap \Delta(\R_{\alpha, i}, r_{k+1}^{a_i+t_i})\Big).\end{align*}
Note that \begin{itemize}\item in the directions $i\in \K_1$, $$
r\ge r_{k+1}^{a_i}\ge r_{k+1}^{a_i+t_i},
$$so by the scaling property $$
\prod_{i\in \K_1}m_i\Big(B(z_i, r_{k+1}^{a_i})\cap B(x_i,2r)\cap \Delta(\R_{\alpha, i}, r_{k+1}^{a_i+t_i})\Big)\asymp \prod_{i\in \K_1}r_{k+1}^{a_i\delta_i\kappa}\cdot r_{k+1}^{(a_i+t_i)\delta_i(1-\kappa)};
$$

\item in the directions $i\in \K_2$, $$
r<r_{k+1}^{a_i+t_i}\le r_{k+1}^{a_i},
$$ clearly $$ \prod_{i\in \K_2}m_i\Big(B(z_i, r_{k+1}^{a_i})\cap B(x_i,2r)\cap \Delta(\R_{\alpha, i}, r_{k+1}^{a_i+t_i})\Big)\ll \prod_{i\in \K_2}r^{\delta_i};$$

\item in the directions $i\in \K_3$, $$
r_{k+1}^{a_i+t_i}\le r\le r_{k+1}^{a_i},
$$ so by the scaling property$$
\prod_{i\in \K_3}m_i\Big(B(z_i, r_{k+1}^{a_i})\cap B(x_i,2r)\cap \Delta(\R_{\alpha, i}, r_{k+1}^{a_i+t_i})\Big)\asymp \prod_{i\in \K_3}r^{\delta_i \kappa}\cdot r_{k+1}^{(a_i+t_i)\delta_i(1-\kappa)}.
$$
\end{itemize}
Thus \begin{align*}
T\cdot \prod_{i=1}^d\left(r_{k+1}^{a_d+t_d}\right)^{\delta_i}&\le \left(\prod_{i\in \K_1}r_{k+1}^{a_i\delta_i\kappa}\cdot r_{k+1}^{(a_i+t_i)\delta_i(1-\kappa)}\right)\times \left(\prod_{i\in \K_2}r^{\delta_i}\right)\times \left(\prod_{i\in \K_3}r^{\delta_i \kappa}\cdot r_{k+1}^{(a_i+t_i)\delta_i(1-\kappa)}\right).
\end{align*}
Recall (\ref{e1}). Thus the total number of balls $C$ in $\F_{k+1}(C_k, \ell_0)$ which can intersect $B(x,r)$ is less than:\begin{align*}
\prod_{i\in \K_1}\left(\frac{r}{r_{{k+1}}^{a_i}}\right)^{\delta_i}\cdot T:=M.
\end{align*}

So, by the first inequality in (\ref{f15}) and the choice of $L_{C_k}$ (\ref{d7}), one has\begin{align*}
\mu\Big(B(x,r)\cap B\Big)&\le \frac{M}{\eta}\prod_{i=1}^d\left(\frac{r_{k+1}^{a_d+t_d}}{r_{k+1}^{a_i+t_i}}\right)^{\delta_i}\cdot \prod_{i=1}^dr_{k+1}^{t_i\delta_i\kappa}\cdot \prod_{i=1}^dr_{k+1}^{a_i\delta_i}\\
&\le \frac{1}{\eta}\cdot \prod_{i\in \K_1}\left(\frac{r}{r_{{k+1}}^{a_i}}\right)^{\delta_i}\times \left(\prod_{i\in \K_1}r_{k+1}^{a_i\delta_i\kappa}\cdot r_{k+1}^{(a_i+t_i)\delta_i(1-\kappa)}\right)\times \left(\prod_{i\in \K_2}r^{\delta_i}\right)\\
&\qquad \qquad \times \left(\prod_{i\in \K_3}r^{\delta_i \kappa}\cdot r_{k+1}^{(a_i+t_i)\delta_i(1-\kappa)}\right)\cdot \prod_{i=1}^d\left(\frac{1}{r_{k+1}^{a_i+t_i}}\right)^{\delta_i}\cdot \prod_{i=1}^dr_{k+1}^{t_i\delta_i\kappa}\cdot \prod_{i=1}^dr_{k+1}^{a_i\delta_i}.
\end{align*}
Then
\begin{align}\label{g4}
\mu\Big(B(x,r)\cap B\Big)&\le \frac{1}{\eta} \left(\prod_{i\in \K_1}r^{\delta_i}\cdot r_{k+1}^{t_i\delta_i(1-\kappa)}\right)\times \left(\prod_{i\in \K_2}r^{\delta_i}\right)\times \left(\prod_{i\in \K_3}r^{\delta_i\kappa}\cdot r_{k+1}^{(a_i+t_i)\delta_i(1-\kappa)}\right)\nonumber\\
&\qquad \qquad \qquad \qquad\times \left(\prod_{i=1}^d\frac{1}{r_{k+1}^{t_i}}\right)^{\delta_i}\cdot \prod_{i=1}^dr_{k+1}^{t_i\delta_i\kappa}\nonumber\\
&=\frac{1}{\eta}\cdot \prod_{i\in \K_1}r^{\delta_i}\cdot \prod_{i\in \K_2}r^{\delta_i}\cdot \prod_{i\in \K_3}r^{\delta_i\kappa}\cdot \prod_{i\in \K_3}r_{k+1}^{a_i\delta_i(1-\kappa)}\cdot \prod_{i\in \K_2}\frac{1}{r_{k+1}^{t_i\delta_i(1-\kappa)}}.
\end{align}
So we will have that $$
\mu\Big(B(x,r)\cap B\Big)\le \frac{r^s}{\eta}
$$
if we can check \begin{align}\label{g3}
s&\le \sum_{i\in \K_1}\delta_i+\sum_{i\in \K_2}\delta_i+\kappa\sum_{i\in \K_3}\delta_i+(1-\kappa)\frac{(\sum_{i\in \K_3}a_i\delta_i-\sum_{i\in \K_2}t_i\delta_i)\log r_{k+1}}{\log r},
\end{align} for all $r$ inside the range $$
r_{{k+1}}^{A_{l+1}}\le r< r_{{k+1}}^{A_l}.
$$
Since the term in the right hand of (\ref{g3}) is monotonic with respect to $r$, the minimal value attains when $r$ takes the boundary values of its arrange. So \begin{align*}
(\ref{g4})\Longleftarrow &\ s\le \min\Big\{\sum_{i\in \K_1}\delta_i+\sum_{i\in \K_2}\delta_i+\kappa\sum_{i\in \K_3}\delta_i+(1-\kappa)\frac{\sum_{i\in \K_3}a_i\delta_i-\sum_{i\in \K_2}t_i\delta_i}{A_l},\\ &
 \qquad \qquad  \qquad \qquad\sum_{i\in \K_1}\delta_i+\sum_{i\in \K_2}\delta_i+\kappa\sum_{i\in \K_3}\delta_i+(1-\kappa)\frac{\sum_{i\in \K_3}a_i\delta_i-\sum_{i\in \K_2}t_i\delta_i}{A_{l+1}}\Big\}\\
\Longleftarrow &\ s\le s(\bold{t}).
\end{align*} Note that the minor difference between $\K$ defined in (\ref{e2}) and that in Theorem \ref{t1} makes no difference on the dimensional number $s(\bold{t})$ as explained in (\ref{d4}).

In a summary, we have shown that for all $x\in \F_{\infty}$ and $r$ small, $$
\mu(B(x,r))\le \frac{c\cdot r^s}{\eta}.
$$
Then Proposition \ref{p1} is applied to conclude the desired result.

\section{Proof: general case under ubiquity}

The proof of Theorem \ref{t3} is almost identical to the proof of Theorem \ref{t2} after a minor modification on the Cantor set construction in applying the uniform ubiquity property.
Recall the construction of the $k$th level of the Cantor set: given $C_{k-1}\in \F_{k-1}$, \begin{itemize}
\item Use $K_{G,B}$-lemma to get a subcollection of well separated big rectangles from $$
\{\widetilde{R}=\prod_{i=1}^dB(z_i, \rho(u_n)^{a_i}): z\in \R_{\alpha}, \alpha\in {J}_n\}, \ {\text{for some}}\ n\in \N.
$$ We call them as the big rectangles of order $n$.

\item Then consider the intersection: $$
\widetilde{R}\cap \Delta(\R_{\alpha}, \rho(u_n)^{\a}\Psi(u_n))
$$ to get a collection of shrunk rectangles.
\end{itemize}

For general function $\Psi$, since we are equipped with uniform ubiquity property, in its applications, we have much freedom in choosing the order $n$ of the big rectangles. So, for any $\bold{t}=(t_1,\cdots,t_d)\in \widehat{\mathcal{U}}$ given in advance, we can always ask $n$ to fall into the following set $$
\mathcal{N}:=\Big\{n\in \N: \rho(u_n)^{t_i+\epsilon}\le \psi_i(u_n)\le \rho(u_n)^{t_i-\epsilon}, \ {\text{for all}}\ 1\le i\le d\Big\}, \ {\text{for given}}\ \epsilon>0 \
$$ in each level of the Cantor set construction when apply the ubiquity property.
So, along the sequence $\mathcal{N}$, $\Psi(u_n)$ behaves like $\rho(u_n)^{\bold{t}}$.
The rest argument is just following the proof in Theorem \ref{t2} line by line arriving at $$
\hdim W(\Psi)\ge s(\bold{t}).
$$

Now we show the last assertion in Theorem \ref{t3}. At first, it is clear that the dimensional number $s(t_1,\cdots, t_d)$ is non-increasing with respect to $(t_1,\cdots, t_d).$

Secondly, under the condition that \begin{equation}\label{d2}\lim_{n\to \infty}\frac{\log \rho(u_{n+1})}{\log \rho(u_n)}=1,\end{equation} one claims that for any $\bold{t}=(t_1,\cdots,t_d)\in \mathcal{U}$, there exists $\bold{t}'=(t_1',\cdots,t_d')\in \widehat{\mathcal{U}}$, such that $$
t_i\ge t_i', \  \ {\text{for all}}\ 1\le i\le d.
$$ More precisely, assume that $$
\lim_{k\to\infty}\frac{\log \psi_i({n_k})}{\log \rho({n_k})}=t_i, \ 1\le i\le d.
$$  Let $u_{t_k}$ be the largest element in $\{u_n\}$ such that $u_{t_k}\le n_k$. Since both $\psi_i$ and $\rho$ are non-increasing, $$
\frac{\log \psi_i(u_{t_k})}{\log \rho(u_{t_k})}\le \frac{\log \psi_i(n_k)}{\log \rho(n_k)}\cdot \frac{\log \rho(u_{t_k+1})}{\log \rho(u_{t_k})}.$$ So by (\ref{d2}), the claim follows.

Combining the monotonicity of $s(\bold{t})$ and the above claim, one has $$\sup\Big\{s(t_1,\cdots, t_d): (t_1,\cdots,t_d)\in \widehat{\mathcal{U}}\Big\}\ge
\sup\Big\{s(t_1,\cdots, t_d): (t_1,\cdots,t_d)\in {\mathcal{U}}\Big\}.$$

\section{Proof: Hausdorff dimension under full measure}

In this section, we show Theorem \ref{ttt1}, i.e. the dimension of $$
W(\bold{t})=\Big\{x\in X: x\in \Delta(\R_{\alpha}, \rho(\beta_{\alpha})^{\bold{a+t}}), \ {\text{i.m}}\ \alpha\in J\Big\}
$$ under the full measure condition (\ref{j1}) that for any ball $B\subset X$, \begin{equation}\label{h4}
m\Big(B\cap \limsup_{\alpha:\beta_{\alpha}\to \infty}\Delta(\R_{\alpha}, \rho(\beta_{\alpha})^{\bold{a}})\Big)=m(B).
\end{equation}

\subsection{Preliminaries}\

The proof follows almost from the argument in the ubiquity case. So, we give a summary about the preliminary lemmas used in the proof of the ubiquity case given in Subsection \ref{s5.2}. \begin{itemize}
\item $5r$-covering lemma (Lemma \ref{lll1}): given a collection of aligned rectangles, one can find a sub-collection of rectangles which are well separated and can cover the union of the original rectangles by a finite scale.

    \item $K_{G,B}$-lemma (Lemma \ref{lll2}): by ubiquity, for any ball $B\subset X$ and any $G$ large, one can find a collection of big rectangles with the same sidelengths which are well separated and take a positive proportion inside $B$;
\item Shrinking lemma (Lemma \ref{l2}): For each big rectangle $\widetilde{R}=\Delta(z, \rho(u_n)^{\bold{a}})$ centered at some point $z\in \R_{\alpha}$, consider the intersection of $\widetilde{R}\cap \Delta(\R_{\alpha}, \rho(u_n)^{\bold{a+t}})$ to get a collection of well separated shrunk rectangles $R$ with the same sidelengths $\rho(u_n)^{\bold{a+t}}$.

    \item Division lemma (Lemma \ref{l3}): For each shrunk rectangle $R$, it is divided into a collection of well separated balls with radius as the smallest sidelength of $R$.
\end{itemize}

When the ubiquity condition is replaced by the full measure property (\ref{h4}), one still have that \begin{itemize}
  \item $5r$-covering lemma which is unchanged as Lemma \ref{lll1};
  \item $K_{G,B}$-lemma is modified to the following one. \begin{lemma}[$K_{G,B}$-lemma]\label{k2}
Assume the full measure property for rectangles (\ref{h4}). Let $B$ be a ball in $X$. For any $G\in \N$, there exists a finite sub-collection $K_{G,B}$ of the rectangles
\begin{equation*}
\Big\{\widetilde{R}=\prod_{i=1}^dB(z_i, \rho(\beta_{\alpha})^{a_i}): (z_1,\cdots, z_d)\in \R_{\alpha}, \ \beta_{\alpha}\ge G,\ \alpha\in J, \ \widetilde{R}\cap 1/2B\ne \emptyset\Big\}
\end{equation*} such that \begin{itemize}
\item all the rectangles in $K_{G,B}$ are contained in $2/3B$;

\item the rectangles are $3r$-disjoint in the sense that for any different elements in $K_{G,B}$
$$
3\widetilde{R}\cap 3\widetilde{R'}=\emptyset;
$$

\item these rectangles almost pack the ball $B$ in the sense that $$
m\left(\bigcup_{\widetilde{R}\in K_{G,B}}\widetilde{R}\right)\ge c' m(B).
$$
\end{itemize}

\end{lemma}
\begin{proof} The idea of the proof is the same as in \cite{BV06} for the ball case.\end{proof}
\item Shrinking lemma: the only difference is that we change $\rho(u_n)$ to $\rho(\beta_{\alpha})$, i.e. we consider the intersection of $\widetilde{R}=\Delta(z, \rho(\beta_{\alpha})^{\bold{a}})$ with $\Delta(\R_{\alpha}, \rho(\beta_{\alpha})^{\bold{a+t}})$ to get a collection of well separated shrunk rectangles of sidelength $\rho(\beta_{\alpha})^{\bold{a+t}}$. So Lemma \ref{l2} is unchanged.

    \item Division lemma which is unchanged as Lemma \ref{l3}.\end{itemize}

\subsection{Cantor subset construction}\ \medskip


Since only the Hausdorff dimension of $W(\bold{t})$ is concerned, for the Cantor subset to be constructed, each local level contains only one sublevel instead of many sublevels as did for the Hausdorff measure in ubiquity case.

Let $s=s(\bold{t})$ be the dimensional number given in Theorem \ref{t1}.
 Recall that $$
\mathcal{A}=\{a_1,\cdots, a_d, a_1+t_1,\cdots, a_t+t_d\},
$$  and we assume that $a_1$ is the smallest one in $\mathcal{A}$ and $a_d+t_d$ the largest one. \medskip


{\em The first level $\F_1$}. Let $B_0=X$.\smallskip

\begin{itemize}
\item {\em Step $1$. Use $K_{G,B}$-lemma}. Choose an integer $G$ sufficiently large. Apply the $K_{G,B}$-lemma to $B_0$
 to obtain a collection of rectangles.
We call these rectangles as big rectangles and we have \begin{equation*}
    \sum_{\widetilde{R}\in K_{G,B_0}}m(\widetilde{R})\asymp m(B_0).
    \end{equation*}

\item {\em Step $2$. Shrinking}. Fix a rectangle $\widetilde{R}\in K_{G,B}$.
Assume that its center $(z_1,\cdots,z_d)$ sits in a resonant set $\R_{\alpha}$, so $\widetilde{R}$ is of the form $$
\widetilde{R}=\prod_{i=1}^dB(z_i, \rho(\beta_{\alpha})^{a_i}).
$$

Consider the intersection $$
\widetilde{R}\cap \Delta(\R_{\alpha}, \rho(\beta_{\alpha})^{\a+\bold{t}}).
$$
By Lemma \ref{l2}, we obtain a collection $\D(\widetilde{R})$ of $5r$-separated rectangles with the form \begin{equation}\label{j2}
R=\prod_{i=1}^dB(y_i, \rho(\beta_{\alpha})^{a_i+t_i}), \ \ {\text{and}}\  \ (y_1,\cdots, y_d)\in \R_{\alpha}.
\end{equation} Moreover \begin{align*}
\sharp \D(\widetilde{R})\asymp \prod_{i=1}^d\frac{1}{\rho(\beta_{\alpha})^{t_i\delta_i \kappa}}, \ \
    \sum_{R\in \D(\widetilde{R})}m(R)\asymp m(\widetilde{R})\cdot \prod_{i=1}^d \rho(\beta_{\alpha})^{t_i\delta_i(1-\kappa)}.
    \end{align*}
We call the  small rectangles in $\D(\widetilde{R})$ as {\em shrunk rectangles}.

\item {\em Step $3$. Dividing.} For each shrunk rectangle $R\in \D(\widetilde{R})$ with the form as in (\ref{j2}), cut it into balls with radius $\rho(\beta_{\alpha})^{a_d+t_d}$. Then by Lemma \ref{l3}, we get a collection $\C(R)$ of $5r$-separated balls $B_1$ with radius $\rho(\beta_{\alpha})^{a_d+t_d}$ contained in $R$, moreover the cardinality of $\C(R)$ satisfies $$
\sharp \C(R)\asymp \prod_{i=1}^d\left(\frac{\rho(\beta_{\alpha})^{a_i+t_i}}{\rho(\beta_{\alpha})^{a_d+t_d}}\right)^{\delta}.
$$  \end{itemize}

Then the first subslevel is defined as \begin{equation*}
\F_1=\F_1(B_0)=\bigcup_{\widetilde{R}\in K_{G,B_0}}\ \bigcup_{R\in \D(\widetilde{R})}\ \bigcup_{B_1\in \C(R)}\ B_1.
\end{equation*}

    {\em The general level $\F_k$.}

    Assume that $\F_1, \cdots, \F_{k-1}$ have been well constructed, which is a collection of balls. Define $$
    \F_k=\bigcup_{B_{k-1}\in \F_{k-1}}\F_k(B_{k-1}),
    $$ where $\F_k(B_{k-1})$ is defined in the same manner as that for $\F_1(B_0)$ by just replacing the role of $B_0$ by $B_{k-1}$.
More precisely,
$$
\F_{k}(B_{k-1})=\bigcup_{\widetilde{R}\in K_{G,B_{k-1}}}\bigcup_{R\in \mathcal{D}(\widetilde{R})}\bigcup_{B_k\in \mathcal{C}(R)}B_k.
$$

Finally, the desired Cantor set is defined as $$
\F_{\infty}=\bigcap_{k=1}^{\infty}\bigcup_{B_k\in \F_k}B_k,
$$ which is clearly a subset of $W(\t)$.

We still have the following separation conditions for the balls in $\F_k$.
\begin{proposition}\label{pp2}
Any balls in $\F_k$ are $3r$-separated. Consequently, if a ball $B(x,r)$ can intersect two balls in $\F_k$, then all these balls are contained in $B(x,3r)$.
\end{proposition}

\subsection{Mass distribution}\ \medskip


%
%
%

Let $\mu(B_0)=1$. Assume the measure $\mu$ have been defined for the balls in $\F_{k-1}$. Let $B_k$ be a ball in $\F_k$. Assume that $$
B_k\in \mathcal{C}(R), \ \ R\in \mathcal{D}(\widetilde{R}), \ \ \widetilde{R}\in K_{G,B_{k-1}},
$$ for some $B_{k-1}\in \F_{k-1}$. We define \begin{align*}
\mu(\widetilde{R})=\frac{m(\widetilde{R})}{\sum_{\widetilde{R'}\in K_{G,B_{k-1}}}m(\widetilde{R'})}\cdot \mu(B_{k-1});\ \
 \mu(R)=\frac{1}{\sharp \mathcal{D}(\widetilde{R})}\cdot\mu(\widetilde{R});  \ \ \mu(B_k)=\frac{1}{\sharp \mathcal{C}(R)}\cdot\mu(R).
\end{align*}

Assume that the center of $B_k$ sits in $\R_{\alpha}$ for some $\alpha\in J$, then a generic form of the ball $B_k$, the shrunk rectangle $R$ and the big rectangle $\widetilde{R}$ is $$ B_k=B(x, r_k^{a_d+t_d}), \ \
R=\prod_{i=1}^dB(y_i, r_k^{a_i+t_i}), \ \ \widetilde{R}=\prod_{i=1}^dB(z_i, r_k^{a_i}),
$$ respectively with $r_k=\rho(\beta_{\alpha})$. Then we have the following more precise formula for their measure:
\begin{align}\label{gg5}
\mu(\widetilde{R})&\asymp \prod_{i=1}^d r_k^{a_i\delta_i}\cdot \frac{\mu(B_{k-1})}{m(B_{k-1})},\ \
\mu(R)
\asymp \prod_{i=1}^d r_{k}^{t_i\delta_i\kappa}\cdot\prod_{i=1}^dr_{k}^{a_i\delta_i}\times \frac{\mu(B_{k-1})}{m(B_{k-1})},\\
&\mu(B_k)
\asymp \prod_{i=1}^d\left(\frac{r_k^{a_d+t_d}}{r_k^{a_i+t_i}}\right)^{\delta_i}\cdot\prod_{i=1}^d r_k^{t_i\delta_i\kappa}\cdot\prod_{i=1}^dr_k^{a_i\delta_i}\times \frac{\mu(B_{k-1})}{m(B_{k-1})}.\label{e11}
\end{align}
Note that $r_k$ depends on the index $\alpha\in J$ of the resonant set $\R_{\alpha}$ for which the center of $B_k$ sits. So different balls $B_k\in \F_k$ may be of different radius.

By Kolmogorov's consistency theorem, the set function $\mu$ can be uniquely extended into a probability measure supported on $\F_{\infty}$.

%

\subsubsection{Measure of balls in $\F_k$.}\

When $k=1$. Let $B_1$ be a ball in $\F_1$ with radius $r_1^{a_d+t_d}.$
 one has \begin{align*}
\mu(B_1)\asymp \prod_{i=1}^d\left(\frac{r_1^{a_d+t_d}}{r_1^{a_i+t_i}}\right)^{\delta_i}\cdot\prod_{i=1}^d r_{1}^{t_i\delta_i\kappa}\cdot\prod_{i=1}^dr_{1}^{a_i\delta_i}={r_1^{(a_d+t_d)s_d}}=r_{B_1}^{s_d},
\end{align*} where $s_d$ is given as $$
s_d=\sum_{i=1}^d\delta_i-(1-\kappa)\frac{\sum_{i=1}^dt_i\delta_i}{a_d+t_d}.
$$

Now for a ball $B_k$ in $\F_k$ with radius $r_k^{a_d+t_d}$, by (\ref{e11}), one has \begin{equation*}
\mu(B_k)\asymp r_k^{(a_d+t_d)s_d}\cdot \frac{\mu(B_{k-1})}{m(B_{k-1})}\le r_k^{(a_d+t_d)(s_d-\epsilon)}=r_{B_k}^{s_d-\epsilon},
\end{equation*} where the inequality can be guaranteed if ask $G$ as large as we want in applying $K_{G,B}$-lemma to $B_{k-1}$.


\subsubsection{Measure of a general ball.}\ \medskip

In the ubiquity case, all the rectangles and balls in a sublevel are of the same sidelengths respectively. But here they may be different. The idea to deal with it is to group the big rectangles of equivalent sidelength. For each group, use the argument from the ubiquity case. This is the only tricky point here.


Now we consider the measure of a general ball $B(x,r)$ with $x\in \F_{\infty}$ and $r$ small enough. Let $k$ be the integer such that $B(x,r)$ intersects only one element in $\F_{k}$ and at least two elements in $\F_{k+1}$. Denote $B_k$ the unique element in $\F_k$ for which $B(x, r)$ can intersect.
Let $r_{k}^{a_d+t_d}$ be the radius of the ball $B_k$.
Without loss of generality, we can assume that $r\le r_k^{a_d+t_d}$, otherwise $$
\mu(B(x,r))\le \mu(B_k)\le {r_{B_k}^{s-\epsilon}}\le {r^{s-\epsilon}}.
$$

By the choice of $k$, all the elements in $\F_{k+1}$ which can intersect $B(x,r)$ are contained in the local level $$
\F_{k+1}(B_k)=\bigcup_{\widetilde{R}\in K_{G,B_k}}\bigcup_{R\in \D(\widetilde{R})}\bigcup_{B_{k+1}\in \mathcal{C}(R)}B_{k+1}.
$$

Let $$\widetilde{R}_1, \widetilde{R}_2, \cdots, \widetilde{R}_h,\cdots
$$ be the big rectangles in $K_{G,B_{k}}$ for which there exist balls $B_{k+1}\in \F_{k+1}$ contained in these big rectangles and having non-empty intersection with the ball $B(x,r)$. Note that these big rectangles are of the form \begin{equation}\label{j4}
\widetilde{R}_h=\prod_{i=1}^dB(z_i, \rho_h^{a_i}), \ {\text{for each}}\ h\ge 1.
\end{equation} Then we arrange them in a decreasing order according to the value of $\rho_h$.

(i). $r\ge \rho_1^{a_1}$. In this case, the ball $B(x,r)$ is so large that any big rectangle in (\ref{j4}) are contained in $B(x,3r)$. Thus,
\begin{align*}
\mu(B(x,r))&\le \sum_{h\ge 1}\mu(\widetilde{R}_h)\asymp\sum_{h\ge 1}m(\widetilde{R}_h)\cdot \frac{\mu(B_{k})}{m(B_{k})}
\le m(B(x,3r))\cdot \frac{\mu(B_{k})}{m(B_{k})}\nonumber\\
&\le r^{\sum_{i=1}^d\delta_i}\cdot \frac{r_{B_{k}}^{s-\epsilon}}{r_{B_{k}}^{\sum_{i=1}^d\delta_i}}\le r^{s-\epsilon}
\end{align*} where for the last inequality, we have used the fact $$
s\le \sum_{i=1}^d\delta_i, \ {\text{and}}\ r\le r_{B_{k}}.
$$

(ii). $r< \rho_1^{a_1}$. Let $h_0$ be the largest integer such that $r<\rho_{h_0}^{a_1}$. Then the contribution of the rectangles $$
\widetilde{R}_{h_0+1}, \widetilde{R}_{h_0+2}, \cdots
$$ to the mass of the ball $B(x,r)$ can be controlled as in case (i).

So, we need only consider the contribution of the big rectangles \begin{equation}\label{j5}
\widetilde{R}_{1}, \widetilde{R}_{2}, \cdots, \widetilde{R}_{h_0}.
\end{equation} Note that by the definition of $h_0$ and Proposition \ref{pp2}, we have $$
\rho_{h}^{a_d+t_d}\le r\le \rho_{h}^{a_1}, \ \ {\text{for all}}\ 1\le h\le h_0.
$$

We group the big rectangles in (\ref{j5}) according to the range of $\rho_h$. Let $$
\mathcal{G}_{\ell}=\Big\{\widetilde{R}_h: 2^{-\ell-1}\le \rho_h<2^{-\ell}, \ 1\le h\le h_0\Big\}.
$$
Then we estimate for each $\ell\ge 1$, the quantity $$
\mu\Big(B(x,r)\cap \bigcup_{\widetilde{R}\in \mathcal{G}_{\ell}}\widetilde{R}\Big).
$$

 Since the elements in $\mathcal{G}_{\ell}$ are of almost the same sidelength, from this point on until the formula (\ref{j6}), the following is just a copy from the ubiquity case.

Let $A_{l+1}$ and $A_l$ be the two different consecutive terms in $\mathcal{A}$ such that $$
2^{-\ell\cdot A_{l+1}}\le r< 2^{-\ell\cdot A_l}.
$$
Define the sets $\K$: \begin{equation*}
\K_1=\{i: a_i\ge A_{l+1}\}, \ \K_2=\{i: a_i+t_i\le A_l\},  \ \ \K_3=\{1,\cdots,d\}\setminus (\K_1\cup \K_2).
\end{equation*}

\begin{itemize}\item (a). The number of $\mathcal{G}_{\ell}$. Define an enlarged body of the ball $B(x,r)$: $$
H=\prod_{i=1}^dB(x_i, 3\epsilon_i), \ {\text{where}}\ \epsilon_i=\left\{
                                                              \begin{array}{ll}
                                                                r, & \hbox{$i\in \K_1$;} \\
                                                                2^{-\ell \cdot a_i}, & \hbox{otherwise.}
                                                              \end{array}
                                                            \right.
$$ Then all the big rectangle $\widetilde{R}\in \mathcal{G}_{\ell}$ are contained in $H$. Since these big rectangles $\widetilde{R}$ are disjoint, then a volume argument gives the number of $\mathcal{G}_{\ell}$:
\begin{equation}\label{ee1}\sharp \mathcal{G}_{\ell}\le \prod_{i\in \K_1}\left(\frac{r}{2^{-\ell \cdot a_i}}\right)^{\delta_i}.\end{equation}

\item (b). The number of balls $B_{k+1}\in \F_{k+1}$ contained in $$
B(x,2r)\cap \widetilde{R}, \ {\text{for each}}\ \widetilde{R}\in \mathcal{G}_{\ell}.
$$
Fix a generic big rectangle $\widetilde{R}\in \mathcal{G}_{\ell}$. Let $T$ be the number of the balls \begin{equation}\label{g9}B_{k+1}\in \bigcup_{R\in \D(\widetilde{R})}\mathcal{C}(R)\end{equation} which can intersect $B(x,r)$. Note that the radius of $B_{k+1}$ is about $2^{-\ell \cdot (a_d+t_d)}$.

 Let $\alpha\in J$ be the index such that the center of $\widetilde{R}$ sits in $\R_{\alpha}$. Recall the process of the construction of $\mathcal{D}(\widetilde{R})$, all these balls $B_{k+1}$ in (\ref{g9}) are contained in $$
\widetilde{R}\cap \Delta(\R_{\alpha}, 2^{-\ell\cdot (\a+\bold{t})})\cap B(x,2r).
$$ Still we use a volume argument: \begin{align*}
T\cdot \prod_{i=1}^d\left(2^{-\ell\cdot (a_d+t_d)}\right)^{\delta_i}&\le m\left(\widetilde{R}\cap B(x,2r)\cap \prod_{i=1}^d \Delta(\R_{\alpha, i}, 2^{-\ell\cdot(a_i+t_i)})\right)\\
&=\prod_{i=1}^dm_i\Big(B(z_i, 2^{-\ell\cdot a_i})\cap B(x_i,2r)\cap \Delta(\R_{\alpha, i}, 2^{-\ell\cdot (a_i+t_i)})\Big).\end{align*}
Note that \begin{itemize}\item in the directions $i\in \K_1$, $$
r\ge 2^{-\ell\cdot a_i}\ge 2^{-\ell\cdot (a_i+t_i)},
$$so by the scaling property $$
\prod_{i\in \K_1}m_i\Big(B(z_i, 2^{-\ell\cdot a_i})\cap B(x_i,2r)\cap \Delta(\R_{\alpha, i}, 2^{-\ell\cdot (a_i+t_i)})\Big)\asymp \prod_{i\in \K_1}2^{-\ell\cdot a_i\delta_i\kappa}\cdot 2^{-\ell\cdot (a_i+t_i)\delta_i(1-\kappa)};
$$

\item in the directions $i\in \K_2$, $$
r<2^{-\ell\cdot (a_i+t_i)}\le 2^{-\ell\cdot a_i},
$$ clearly $$ \prod_{i\in \K_2}m_i\Big(B(z_i, 2^{-\ell\cdot a_i})\cap B(x_i,2r)\cap \Delta(\R_{\alpha, i}, 2^{-\ell\cdot (a_i+t_i)})\Big)\ll \prod_{i\in \K_2}r^{\delta_i};$$

\item in the directions $i\in \K_3$, $$
2^{-\ell\cdot (a_i+t_i)}\le r\le 2^{-\ell\cdot a_i},
$$ so by the scaling property$$
\prod_{i\in \K_3}m_i\Big(B(z_i, 2^{-\ell\cdot a_i})\cap B(x_i,2r)\cap \Delta(\R_{\alpha, i}, 2^{-\ell\cdot (a_i+t_i)})\Big)\asymp \prod_{i\in \K_3}r^{\delta_i \kappa}\cdot 2^{-\ell\cdot (a_i+t_i)\delta_i(1-\kappa)}.
$$
\end{itemize}
Thus \begin{align*}
T\cdot \prod_{i=1}^d\left(2^{-\ell\cdot(a_d+t_d)}\right)^{\delta_i}&\le \left(\prod_{i\in \K_1}2^{-\ell\cdot a_i\delta_i\kappa}\cdot 2^{-\ell\cdot(a_i+t_i)\delta_i(1-\kappa)}\right)\times \left(\prod_{i\in \K_2}r^{\delta_i}\right)\times \left(\prod_{i\in \K_3}r^{\delta_i \kappa}\cdot 2^{-\ell\cdot(a_i+t_i)\delta_i(1-\kappa)}\right).
\end{align*}
Recall (\ref{ee1}). Thus the total number of balls $B_{k+1}$ in $\F_{k+1}(B_k)$ which can intersect $B(x,r)$ is less than:\begin{align*}
\prod_{i\in \K_1}\left(\frac{r}{2^{-\ell\cdot a_i}}\right)^{\delta_i}\cdot T:=M.
\end{align*}

\end{itemize}

So, by (\ref{gg5}) one has\begin{align*}
\mu(B(x,r)\cap \mathcal{G}_{\ell})&\le M\cdot \prod_{i=1}^d\left(\frac{2^{-\ell\cdot (a_d+t_d)}}{2^{-\ell\cdot (a_i+t_i)}}\right)^{\delta_i}\cdot \prod_{i=1}^d2^{-\ell\cdot t_i\delta_i\kappa}\cdot \prod_{i=1}^d2^{-\ell\cdot a_i\delta_i}\cdot \frac{\mu(B_k)}{m(B_k)}\\
&= \prod_{i\in \K_1}r^{\delta_i}\cdot \prod_{i\in \K_2}r^{\delta_i}\cdot \prod_{i\in \K_3}r^{\delta_i\kappa}\cdot \prod_{i\in \K_3}2^{-\ell\cdot a_i\delta_i(1-\kappa)}\cdot \prod_{i\in \K_2}\frac{1}{2^{-\ell\cdot t_i\delta_i(1-\kappa)}}\cdot \frac{\mu(B_k)}{m(B_k)}\nonumber\\
&\le {r^s}\cdot \frac{\mu(B_k)}{m(B_k)},
\end{align*}
where the last inequality is true since one can check as before that\begin{align*}
s&\le \sum_{i\in \K_1}\delta_i+\sum_{i\in \K_2}\delta_i+\kappa\sum_{i\in \K_3}\delta_i+(1-\kappa)\frac{(\sum_{i\in \K_3}a_i\delta_i-\sum_{i\in \K_2}t_i\delta_i)\log 2^{-\ell}}{\log r},
\end{align*} for all $r$ inside its range $$
2^{-\ell\cdot A_{l+1}}\le r< 2^{-\ell\cdot A_l}.
$$

Thus we obtain that \begin{align}\label{j6}
\mu\Big(B(x,r)\cap \bigcup_{\widetilde{R}\in \mathcal{G}_{\ell}}\widetilde{R}\Big)\le {r^s}\cdot \frac{\mu(B_k)}{m(B_k)}.
 \end{align} Recall that in applying the $K_{G,B}$-lemma to $B_k$, we can ask $G$ so large that for all $\widetilde{R}$, its sidelength is much small compared with the radius of $B_k$, saying $$
\rho_h^{a_1\epsilon} \le m(B_k)\Longrightarrow \frac{\mu(B_k)}{m(B_k)}\le \frac{1}{m(B_k)}\le \rho_h^{-a_1\epsilon}\le r^{-\epsilon}.
$$ So, $$
\mu\Big(B(x,r)\cap \bigcup_{\widetilde{R}\in \mathcal{G}_{\ell}}\widetilde{R}\Big)\le r^{s-\epsilon}\le r^{s-2\epsilon} \cdot 2^{-\ell a_1\epsilon}.
$$

As a conclusion, we get $$
\sum_{\ell\ge 1}\mu\Big(B(x,r)\cap \bigcup_{\widetilde{R}\in \mathcal{G}_{\ell}}\widetilde{R}\Big)\le \sum_{\ell\ge 1}r^{s-2\epsilon} \cdot 2^{-\ell a_1\epsilon}\asymp r^{s-2\epsilon}.
$$

In a summary, we have shown that for all $x\in \F_{\infty}$ and $r$ small, $$
\mu(B(x,r))\le c\cdot r^{s-2\epsilon}.
$$
Then Proposition \ref{p1} is applied to conclude the desired result.

\section{A last property on the dimensional number}

The proof of Proposition \ref{p6} is elementary but not obvious. Recall that $$
\mathcal{A}=\{a_i, a_i+t_i: 1\le i\le d\}, \ \ \widehat{\mathcal{A}}=\{a_i+t_i: 1\le i\le d\}
$$ The dimensional number in Theorem \ref{t1} corresponding to  $A\in \mathcal{A}$ is given as \begin{equation}\label{h3}
\sum_{k\in \K_1}\delta_k+\sum_{k\in \K_2}\delta_k+\kappa\sum_{k\in \K_3}\delta_k+(1-\kappa)\frac{\sum_{k\in \K_3}a_k\delta_k-\sum_{k\in \K_2}t_k\delta_k}{A},
 \end{equation}
 where $$
 \K_1=\{k: a_k\ge A\},\ \K_2=\{k: a_k+t_k\le A\}\setminus \K_1,\ \ \K_3=\{1, 2,\cdots, d\}\setminus (\K_1\cup \K_2).
 $$

 When $A=a_i$, the dimensional number is denoted by $\sigma_i$ and  when $A=a_i+t_i$, the dimensional number is denoted by $s_i$. So what we need to show is that $$
\min\{\sigma_i:  1\le i\le d\}\ge \min\{s_i: 1\le i\le d\}.
$$

Without loss of generality, we assume that \begin{equation}\label{h1}
a_1\le a_2\le \cdots \le a_d.
\end{equation} If $a_i+t_i=a_k$ for some $i$ and $k$, they define the same dimensional number,
so we merge $a_k$ into $a_i+t_i$. Thus a general part when list the elements in $\mathcal{A}$ in ascending order can be expressed as follows:
\begin{center}
\begin{tabular}{|p{60pt}l|p{121pt}c|p{71pt}c|p{98pt}c|}
  \hline \smallskip
  $\ \ {\text{Terms}} \ a+t$ &   &\smallskip \ \ \ \ \ \ \ \ ${\text{Terms}}\ a$ &   & \smallskip$\ \ \  {\text{Terms}}\ a+t$ \smallskip& &\smallskip \ \ \ \ \ \ \ \ ${\text{Terms}}\ a$ & \\
  \hline \smallskip
 $\cdots\le a_{j_1}+t_{j_1}$  &   &\smallskip $< a_{i_{1}+1}\le a_{i_{1}+2}\le\cdots \le a_{i_{2}}$ &   & \smallskip$< a_{j_2}+t_{j_2}\le \cdots$ & &\smallskip $< a_{i_{2}+1}\le a_{i_{2}+2}\le\cdots $ & \smallskip \\
  \hline
\end{tabular}\end{center}
\smallskip

Let $$
K_2=\{k: a_k+t_k\le a_{j_1}+t_{j_1}\}.
$$ By (\ref{h1}), we know that $$
K_2\subset \{1,\cdots, i_1\}, \ {\text{and}}\ j_2\le i_2, \ j_2\notin K_2.
$$

Now we specify the dimensional numbers corresponding to the elements in the above table:
\begin{itemize}\item when $A=a_{j_1}+t_{j_1}$, $$
\K_1=\{i_1+1,\cdots,d\}, \ \K_2=K_2,\  \ \K_3=\{1,\cdots, i_1\}\setminus K_2.
$$Thus by (\ref{h3})\begin{align*}
s_{j_1}=\sum_{k\ge i_1+1}\delta_k+(1-\kappa)\sum_{k\in K_2}\delta_k+\kappa\sum_{k\le i_1}\delta_k+(1-\kappa)\frac{\sum_{k\le i_1}a_k\delta_k-\sum_{k\in K_2}\delta_k(a_k+t_k)}{a_{j_1}+t_{j_1}}.
\end{align*}

\item when $A=a_{i_1+1}$, $$
\K_1=\{i_1+1,\cdots,d\}, \ \K_2=K_2,\  \ \K_3=\{1,\cdots, i_1\}\setminus K_2.
$$Thus \begin{align*}
\sigma_{i_1+1}&=\sum_{k\ge i_1+1}\delta_k+(1-\kappa)\sum_{k\in K_2}\delta_k+\kappa\sum_{k\le i_1}\delta_k+(1-\kappa)\frac{\sum_{k\le i_1}a_k\delta_k-\sum_{k\in K_2}\delta_k(a_k+t_k)}{a_{i_1+1}}\\
&=\sum_{k\ge i_1+2}\delta_k+(1-\kappa)\sum_{k\in K_2}\delta_k+\kappa\sum_{k\le i_1+1}\delta_k+(1-\kappa)\frac{\sum_{k\le i_1+1}a_k\delta_k-\sum_{k\in K_2}\delta_k(a_k+t_k)}{a_{i_1+1}}.
\end{align*}

\item when $A=a_{i_1+2}$, $$
\K_1=\{i_1+2,\cdots,d\}, \ \K_2=K_2,\  \ \K_3=\{1,\cdots, i_1+1\}\setminus K_2.
$$Thus \begin{align*}
\sigma_{i_1+2}&=\sum_{k\ge i_1+2}\delta_k+(1-\kappa)\sum_{k\in K_2}\delta_k+\kappa\sum_{k\le i_1+1}\delta_k+(1-\kappa)\frac{\sum_{k\le i_1+1}a_k\delta_k-\sum_{k\in K_2}\delta_k(a_k+t_k)}{a_{i_1+2}}\\
&=\sum_{k\ge i_1+3}\delta_k+(1-\kappa)\sum_{k\in K_2}\delta_k+\kappa\sum_{k\le i_1+2}\delta_k+(1-\kappa)\frac{\sum_{k\le i_1+2}a_k\delta_k-\sum_{k\in K_2}\delta_k(a_k+t_k)}{a_{i_1+2}}.
\end{align*}

\item $\vdots$

\item when $A=a_{i_2}$, $$
\K_1=\{i_2,\cdots, d\}, \ \K_2=K_2, \ \K_3=\{1,\cdots, i_2-1\}\setminus K_2.
$$ Thus
\begin{align*}
\sigma_{i_2}&=\sum_{k\ge i_2}\delta_k+(1-\kappa)\sum_{k\in K_2}\delta_k+\kappa\sum_{k\le i_2-1}\delta_k+(1-\kappa)\frac{\sum_{k\le i_2-1}a_k\delta_k-\sum_{k\in K_2}\delta_k(a_k+t_k)}{a_{i_2}}\\
&=\sum_{k\ge i_2+1}\delta_k+(1-\kappa)\sum_{k\in K_2}\delta_k+\kappa\sum_{k\le i_2}\delta_k+(1-\kappa)\frac{\sum_{k\le i_2}a_k\delta_k-\sum_{k\in K_2}\delta_k(a_k+t_k)}{a_{i_2}}.
\end{align*}

\item when $A=a_{j_2}+t_{j_2}$, $$
\K_1=\{i_2+1,\cdots,d\}, \ \K_2=K_2\cup {j_2},\  \ \K_3=\{1,\cdots, i_2\}\setminus \{K_2 \cup {j_2}\}.
$$Thus \begin{align*}
s_{j_2}=\sum_{k\ge i_2+1}\delta_k+(1-\kappa)\sum_{k\in K_2}\delta_k+\kappa\sum_{k\le i_2}\delta_k+(1-\kappa)\frac{\sum_{k\le i_2}a_k\delta_k-\sum_{k\in K_2}\delta_k(a_k+t_k)}{a_{j_2}+t_{j_2}}.
\end{align*}
\end{itemize}

Now we compare $$
s_{j_1}, \ \sigma_{i_1+1}, \sigma_{i_1+2},\cdots, \sigma_{i_2}, \ s_{j_2}.
$$ Look at the numerators of the fractions in these dimensional numbers.
\begin{enumerate}\item When $$
\sum_{k\le i_1}a_k\delta_k-\sum_{k\in K_2}\delta_k(a_k+t_k)\ge 0.
$$ Then all the numerators there are non-negative.
Look at the second expression of $\sigma_{i_1+1}$ and the first expression of $\sigma_{i_1+2}$, it is clear that $$
\sigma_{i_1+1}\ge \sigma_{i_1+2}.
$$ With the same observation, we see that $\sigma_{i}$ is decreasing with respect to $i\in (i_1, i_2]$ and $
\sigma_{i_2}\ge s_{j_2}.
$ In other words, we have $$
\min\{\sigma_i: i_1+1\le i\le i_2\}=\sigma_{i_2}\ge s_{j_2}.
$$

\item When $$
\sum_{k\le i_2}a_k\delta_k-\sum_{k\in K_2}\delta_k(a_k+t_k)< 0.
$$
Then all the numerators there are negative. One can observe that $\sigma_{i}$ is increasing and $\sigma_{i_1+1}\ge s_{j_1}$. In other words, we have $$
\min\{\sigma_i: i_1+1\le i\le i_2\}=\sigma_{i_1+1}\ge s_{j_1}.
$$

\item When $$
\sum_{k\le i_1}a_k\delta_k-\sum_{k\in K_2}\delta_k(a_k+t_k)< 0, \ \ \ \sum_{k\le i_2}a_k\delta_k-\sum_{k\in K_2}\delta_k(a_k+t_k)\ge 0.
$$ Then some numerators are negative until some point on and from which on, the numerators are positive. Thus $\sigma_i$ increases at the beginning of $i\in (i_1+1, i_2]$ to some point and then decreases from that point on until $i=i_2$. Moreover, it is easy to see $$
s_{j_1}\le \sigma_{i_1+1}, \ \ \sigma_{i_2}\ge s_{j_2}.
$$ Thus $$
\min\{\sigma_i: i_1+1\le i\le i_2\}=\min\{\sigma_{i_1}, \sigma_{i_2}\}\ge \min\{s_{j_1}, s_{j_2}\}.
$$
\end{enumerate}
This completes the proof.

\section{Application 1: Simultaneous Diophantine approximation}

Consider the simultaneous Diophantine approximation: let $\psi_i:\N\to \mathbb{R}^+$ be non-increasing for each $1\le i\le m$ and write $\Psi=(\psi_1,\cdots,\psi_d)$. Define
\begin{align*}
W(\Psi)&=\Big\{x\in [0,1]^m: \|qx_i\|<\psi_i(q), \ 1\le i\le m, \ \ {\text{i.m.}}\ q\in \N\Big\}
\end{align*} with $$
\psi_1(q)\cdots \psi_m(q)\le q^{-1}, \ \ {\text{for all}}\ q\gg 1
$$ otherwise by Minkowski's theorem for convex body, $W(\Psi)$ is of full Lebesgue measure.

\subsection{Special case}
We begin with the special case that $\psi_i(q)=q^{-\tau_i}$, $1\le i\le m$ for all $q\ge 1$, i.e. \begin{align*}
W({\bold{\tau}})&=\Big\{x\in [0,1]^m: \|qx_i\|<q^{-\tau_i}, \ 1\le i\le m, \ \ {\text{i.m.}}\ q\in \N\Big\}\\
&=\Big\{x\in [0,1]^m: x\in \prod_{i=1}^mB\left(\frac{p_i}{q}, \Big(\frac{1}{q}\Big)^{1+\tau_i}\right), \ {\text{i.m.}}\ (q; p_1,\cdots,p_m)\in \N^{m+1}\Big\}
\end{align*} where ${\bold{\tau}}=(\tau_1,\cdots,\tau_m)$ with $$
\tau_1+\tau_2+\cdots+\tau_m>1, \ \ \tau_1\ge \tau_2\ge \cdots \ge \tau_m.
$$

The upper bound of $\hdim W(\tau)$ can be obtained by a direct argument on its natural cover, so we only focus on its lower bound.

By Proposition \ref{p5}, we have a ubiquitous system with $$
J=\Big\{\alpha=(q; p_1,\cdots, p_m): q\in \N,\  0\le p_i<q, 1\le i\le m\Big\}, \ \ \ R_{\alpha}=\Big(\frac{p_1}{q},\cdots,\frac{p_m}{q}\Big),
$$
and  $$
\beta_{\alpha}=q, \ \ \rho(q)=q^{-1},
$$  for any $a_1,\cdots, a_m$ with \begin{align}\label{d1}
 a_1+\cdots+a_m=m+1, \ \ {\text{and}}\ a_i\ge 1, \ 1\le i\le m.
\end{align}
Rewrite $W(\tau)$ as the form in Theorem \ref{t1}: $$
W(\tau)=\Big\{x\in [0,1]^m: x\in \Delta(R_{\alpha}, \rho(\beta_{\alpha})^{\a+\bold{t}}), \ {\text{i.m.}}\ \alpha\in J\Big\}
$$ for any $\a$ and $\bold{t}$ positive  satisfying (\ref{d1}) and $$
a_i+t_i=1+\tau_i, \ 1\le i\le m.
$$

(1). $\tau_m\ge 1/m$. Choose $$
a_i=1+\frac{1}{m}, \ \ t_i=(1+\tau_i)-a_i=\tau_i-\frac{1}{m}, \ \ 1\le i\le d.
$$
Then the order of the elements in the alphabet $\mathcal{A}$ is $$
a_1+t_1\ge \cdots \ge a_m+t_m\ge a_1=\cdots=a_m.
$$
By a direct substitution to the formula in Theorem \ref{t1}, one has \begin{align*}\hdim W(\tau)&=\min\left\{\frac{m+1+\sum_{k=i}^m(\tau_i-\tau_k)}{1+\tau_i}, 1\le i\le m\right\}. \end{align*}

(2). $\tau_m<1/m$. Let $K$ be the largest integer such that $$
\tau_K> \frac{1-(\tau_{K+1}+\cdots+\tau_m)}{K}.
$$
Then choose $$
a_i=\tau_i+1, \ \ {\text{for}}\ \ i\ge K+1, \ \ \ \ {\text{and}}\ \ a_i=\frac{1-(\tau_{K+1}+\cdots+\tau_m)}{K}+1,\ \ {\text{for}} \ 1\le i\le K.
$$
and let $$
t_i=1+\tau_i-a_i, \ \ {\text{for all}} \ 1\le i\le m.
$$
So
the order of the alphabet $\mathcal{A}$ is \begin{align*}
a_1+t_1(=\tau_1+1)\ge \cdots &\ge a_K+t_K(=\tau_K+1)\\
&> a_1=\cdots =a_K> a_{K+1}=a_{K+1}+t_{K+1}\ge \cdots \ge a_m=a_m+t_m.
\end{align*}

By a direct substitution to the formula in Theorem \ref{t1}, one has\begin{theorem}
\begin{align*}
\hdim W(\tau)&\ge \min\left\{\frac{m+1+(m-i)\tau_i-\sum_{k=i+1}^m\tau_k}{1+\tau_i}: 1\le i\le K\right\}\\&=\min\left\{\frac{m+1+(m-i)\tau_i-\sum_{k=i+1}^m\tau_k}{1+\tau_i}: 1\le i\le m\right\},
\end{align*}
where the second equality follows that the terms for $i>K$ are always not the smallest one by direct comparison or by Proposition \ref{p6}.
\end{theorem}
\subsection{General case}

With the choices of $\rho$ and $u_n$ as above, it is clear that
$$\lim_{n\to\infty}\frac{\log \rho(u_{n+1})}{\log \rho(u_n)}=1,$$ so Theorem \ref{t3} applies to give the lower bound of $
\hdim W(\Psi)$.
So in this case, $\mathcal{U}$ is the set of the accumulation points of the sequences $$
 \left\{\Big(\frac{-\log \psi_1(n)}{\log n},\cdots, \frac{-\log \psi_d(n)}{\log n}\Big): n\in \N\right\}.
$$

For the upper bound, we require that $\mathcal{U}$ is bounded.
we give the following partition on $\N$.
The boundedness of $\mathcal{U}$ implies that there exist positive numbers $L_1,\cdots, L_m$ finite such that
$$\mathcal{U}\subset [0,~L_{1})\times\cdots\times [0,~L_{m}).$$
Fix $\varepsilon>0$. Cover the set $[0,~L_{1})\times\cdots\times [0,~L_{m})$ by cubes of sidelength no more than $\varepsilon$:
$$[0,~L_{1})\times\cdots\times [0,~L_{m})=\bigcup\limits_{k_{1},\cdots,k_m\in \mathbb{N}}[a_{k_{1}}^{(1)},a_{k_{1}+1}^{(1)})\times\cdots \times [a_{k_{m}}^{(m)},a_{k_{m}+1}^{(m)}).$$
The union over $(k_{1},\cdots,k_m)$~is finite since $\{L_i\}$ are finite. We denote  all the possible choices of $(k_{1},\cdots,k_m)$~by $\mathcal{B}.$ So,~$\#\mathcal{B}< \infty.$

Choose $q_0$ sufficiently large such that
$$-\frac{\log\psi_{i}(q)}{\log q}< L_{i},~~~~~ \ {\text{for all}}\ q\ge q_0, \ 1\leq i\leq m.$$
Then for any $q\geq q_0,$~$\exists(k_{1},\cdots,k_d)\in \mathcal{B},$ such that
$$\Big(\frac{-\log\psi_{1}(q)}{\log q},\cdots,\frac{-\log\psi_{m}(q)}{\log q}\Big)\in [a_{k_{1}}^{(1)},a_{k_{1}+1}^{(1)})\times\cdots \times [a_{k_{m}}^{(m)},a_{k_{m}+1}^{(m)}).$$

For each $(k_{1},\cdots,k_m)\in \mathcal{B},$ define
\begin{align*}\mathcal{B}(k_{1},\cdots,k_m)=\Big\{q\geq q_0: \Big(\frac{-\log\psi_{1}(q)}{\log q},&\cdots,\frac{-\log\psi_{m}(q)}{\log q}\Big)\\
 &\in [a_{k_{1}}^{(1)},a_{k_{1}+1}^{(1)})\times\cdots \times [a_{k_{m}}^{(m)},a_{k_{m}+1}^{(m)})\Big\}.
\end{align*}
Then we get a partition of $\mathbb{N}\setminus\{1,2,\cdots,q_{0}-1\}:$
$$\mathbb{N}\setminus\{1,2,\cdots,q_{0}-1\}=\bigcup\limits_{(k_{1},\cdots,k_{m})\in \mathcal{B}}\mathcal{B}(k_{1},\cdots,k_{m}).$$

Define $$\mathcal{B}_{1}=\Big\{(k_{1},\cdots,k_{m})\in \mathcal{B}:~\#\mathcal{B}(k_{1},\cdots,k_{m})=+\infty\Big\}.$$
Since $\mathcal{B}$~is finite,
$$\max\Big\{q\in \mathcal{B}(k_{1},\cdots,k_{m}):~(k_{1},\cdots,k_{m})\not\in \mathcal{B}_1\Big\}:=q_1<+\infty.$$

Finally, we have for any $\tilde{q}\ge \max\{q_0, q_1\}$, there is a cover of $W(\psi)$, saying $$
W(\Psi)\subset \bigcup_{(k_1,\cdots,k_m)\in \mathcal{B}_1}\bigcup_{q\ge \tilde{q}, q\in \mathcal{B}(k_1,\cdots, k_m)}\Big\{x\in [0,1)^m: \|qx_i\|<\psi_i(q),  \ 1\le i\le d\Big\},
$$ i.e. the natural cover of $W(\Psi)$ is divided into a finite class. Along each class, for example $q\in \mathcal{B}(k_1,\cdots, k_m)$, the function $\psi_i(q)$ behaves like $q^{-a_{k_i}^{(i)}}$ for each $1\le i\le m$. Then a direct argument for the Hausdorff measure of $W(\Psi)$ gives the upper bound
of its dimension. We omit the details.

Thus in a conclusion, we have \begin{theorem} Let $\psi_i:\mathbb{R}^+\to \mathbb{R}^+$ be non-increasing for all $1\le i\le m$. If $\mathcal{U}$ is bounded,
$$\hdim W(\Psi)=\sup\Big\{s(\tau): \tau\in \mathcal{U}\Big\}.$$
\end{theorem}
\section{Application 2: Linear form}

Consider the application to the dimension of linear forms:
\begin{align*}
W_{mn}(\Psi)=\Big\{x\in [0,1]^{m\times n}: \|q_1x_{i1}+\cdots +q_nx_{in}\|<\psi_i(|\q|)|\q|, \ 1\le i\le m, \ {\text{i.m.}}\ \q\in \N^n\Big\}.
\end{align*}
Here the product space is $$
\prod_{i=1}^m\mathbb{R}^n.
$$ For each $i$, the resonant set $\R_i$ is an $(n-1)$-dimensional hyperplane in $\mathbb{R}^n$. So, $$
\delta_i=n, \ \kappa=1-n^{-1}, \ 1\le i\le m.
$$

By Proposition \ref{p3}, we have a ubiquity system with: $$
\R_{\bold{p},\q}=\prod_{i=1}^m\R_{p_i,\q}=\prod_{i=1}^m\Big\{x_i\in \mathbb{R}^n: q_1x_{i1}+\cdots+q_nx_{in}-p_i=0\Big\}.
$$
$$
\beta(\alpha)=\beta(\bold{p},\q)=|\q|, \ \ \rho(u)=u^{-1},\ \ell_k=u_k=M^{k},\
$$ and for all $a_1\ge 1,\cdots, a_m\ge 1$ such that
$$
\sum_{i=1}^ma_i=m+n.
$$

We still begin with the special case. \begin{align*}
W_{mn}(\lambda)&=\Big\{x\in [0,1]^{m\times n}: \|q_1x_{i1}+\cdots +q_nx_{in}\|<|\q|^{-\lambda_i}|\q|, \ 1\le i\le m, \ {\text{i.m.}}\ \q\in \N^n\Big\}\\
&=\bigcap_{Q=1}^{\infty}\bigcup_{\q\in \N^n: |\q|=Q}\bigcup_{0\le p_1,\cdots, p_n\le Q}\prod_{i=1}^m\Delta(\R_{p_i,\q}, |\q|^{-\lambda_i}),
\end{align*} with
$$\lambda_1\ge \lambda_2\ge \cdots\ge \lambda_m\ge 1,  \ \ \ \sum_{i=1}^m\lambda_i\ge m+n.$$

\subsection{Upper bound}\

The upper bound follows easily by considering its natural cover. For the set $$\prod_{i=1}^m\Delta(\R_{p_i,\q}, |\q|^{-\lambda_i}),$$
we cover it by balls of radius $$
|\q|^{-\lambda_i}, \  \ {\text{for all}} \ 1\le i\le m.
$$
Then $m$ dimensional numbers will arise and we take the smallest one which gives an upper bound of $\hdim W(\lambda)$. Easy computation shows that
\begin{align*}
\hdim W_{mn}(\lambda)\le \min\left\{m(n-1)+\frac{m+n+\sum_{k>i}(\lambda_i-\lambda_k)}{\lambda_i}, 1\le i\le m\right\}.
\end{align*}

\subsection{Lower bound}\


Now we give a choice of $a_i$ for $1\le i\le m$ in the following way.

(1). $\lambda_m\ge 1+n/m$.  Let $a_i=1+n/m$, for all $1\le i\le m$ and $a_i+t_i=\lambda_i$.
The alphabet $\mathcal{A}$ can be described as $$
a_1+t_1(=\lambda_1)\ge \cdots \ge a_m+t_m(=\lambda_n)\ge a_1=\cdots=a_m(=1+n/m).
$$

(2). $\lambda_m<1+n/m$. Let $K$ be the largest integer such that $$
\lambda_K> \frac{m+n-(\lambda_{K+1}+\cdots+\lambda_m)}{K}.
$$
Let $$
a_i=\lambda_i, \ \ {\text{for}}\ \ i\ge K+1, \ \ \ \ {\text{and}}\ \ a_i=\frac{m+n-(\lambda_{K+1}+\cdots+\lambda_m)}{K}+1,\ \ {\text{for}} \ 1\le i\le K.
$$
Then $$
1\le a_i\le \lambda_i, \ {\text{for all}}\ \ 1\le i\le m, \ \ {\text{and}}, \ \ a_1=\cdots=a_K>a_{K+1}=\lambda_{K+1}\ge \cdots\ge a_m=\lambda_m.
$$
Let $t_i$ be given as $$
t_i+a_i=\lambda_i, \ 1\le i\le m.
$$
The order of the alphabet $\mathcal{A}$ can be expressed as \begin{align*}
a_1+t_1(=\lambda_1)\ge \cdots \ge a_K+t_K&(=\lambda_K)> a_1=\cdots =a_K\\
&> a_{K+1}=a_{K+1}+t_{K+1}\ge \cdots \ge a_m=a_m+t_m.
\end{align*}

Then a direct substitution to Theorem \ref{t1} gives that \begin{theorem}$$
\hdim W_{mn}(\lambda)=\min_{1\le i\le m}\left\{m(n-1)+\frac{m+n+\sum_{k=i}^m(\lambda_i-\lambda_k)}{\lambda_i}\right\}.
$$\end{theorem}

With the choices of $\rho$ and $u_k$ as above, it is clear that
$$\lim_{k\to\infty}\frac{\log \rho(u_{k+1})}{\log \rho(u_k)}=1,$$ so Theorem \ref{t3} applies.
With the same argument as in the last section, we have \begin{theorem}Let $\psi_i:\mathbb{R}^+\to \mathbb{R}^+$ be non-increasing for all $1\le i\le m$. If $\mathcal{U}$ is bounded, then$$
\hdim W_{mn}(\Psi)= \sup\Big\{s(\lambda): \lambda\in \mathcal{U}\Big\}.
$$\end{theorem}

\section{Application 3: Shrinking target problems}

Let $b_1,\cdots, b_d\ge 2$ be $d$ integers. Recall the Cantor sets $\C_i$ and other notations are defined in Section \ref{s5.1.3}. For $x_o\in \prod_{i=1}^d\mathcal{C}_i$, define
$$S(\Psi):=\Big\{(x_1,\cdots, x_d)\in \prod_{i=1}^d\mathcal{C}_i: \|b_i^nx_i-x_{o,i}\|<\psi_i(n), \ {\text{i.m.}}\ n\in \N\Big\}.
$$
At first, consider the special case:  let $t_1,\cdots, t_d$ be $d$ positive numbers and write
$$S(\bold{t}):=\Big\{(x_1,\cdots, x_d)\in \prod_{i=1}^d\mathcal{C}_i: \|b_i^nx_i-x_{o,i}\|<e^{-nt_i}, \ {\text{i.m.}}\ n\in \N\Big\}.
$$

With the notation given in Section \ref{s5}  and recall the set $\widehat{W}(t)$ in Section \ref{s4}, the set $S(\Psi)$ can be re-expressed as \begin{align*}
S(\bold{t})&=\{x\in \prod_{i=1}^d\mathcal{C}_i: x\in \Delta(R_{\alpha}, \rho(\beta)^{\a+\bold{t}}), \ {\text{i.m.}}\ \alpha\in J\}\\
&=\bigcap_{N=1}^{\infty}\bigcup_{n=N}^{\infty} \bigcup_{\alpha\in J_n}\Delta(R_{\alpha}, \rho(u_n)^{\bold{a}+\bold{t}})=\widehat{W}(\bold{t}).
\end{align*}
Clearly, the big rectangles $$
\Big\{\Delta(R_{\alpha}, \rho(u_n)^{\bold{a}}): \alpha\in J_n\Big\}
$$are disjoint. So, Corollary \ref{c1} and Theorem \ref{t1} together give the exact dimension of $S(\bold{t})$:\begin{theorem} Write
$$s(\bold{t}):=\min_{A_i\in \mathcal{A}}\Big\{\sum_{k\in \K_1}\delta_k+\sum_{k\in \K_2}\delta_k+\frac{\sum_{k\in \K_3}\delta_k\log b_k-\sum_{k\in \K_2}\delta_kt_k}{A_i}\Big\}
$$ where $$
\mathcal{A}=\{\log b_i, t_i+\log b_i, 1\le i\le d\}
$$ and for each $A_i\in \mathcal{A}$, $\K_1, K_2,\K_3$ gives a partition of $\{1,\cdots,d\}$ defined as \begin{align*}&\K_1=\Big\{k: \log b_k\ge A_i\Big\},\ \ \K_2=\Big\{k: {t_k}+\log b_k\le A_i\Big\}\setminus \K_1,\ \
\K_3=\{1,\cdots,d\}\setminus (\K_1\cup \K_2).\end{align*} Then one has $$\hdim S(\bold{t})=s(\bold{t}), \ {\text{and}}\ \ \mathcal{H}^{s(\bold{t})}(S(\bold{t}))=\infty.$$\end{theorem}

Generally, by denoting $\mathcal{U}$ the set of the accumulation points of the sequences $$
 \left\{\Big(\frac{-\log \psi_1(n)}{n},\cdots, \frac{-\log \psi_d(n)}{n}\Big): n\in \N\right\},
$$\begin{theorem}Let $\psi_i:\mathbb{R}^+\to \mathbb{R}^+$ be non-increasing for all $1\le i\le d$. If $\mathcal{U}$ is bounded, then$$
\hdim S(\Psi)= \sup\Big\{s(\bold{t}): \bold{t}\in \mathcal{U}\Big\}.
$$\end{theorem}

%

\section{Application 4: Multiplicative Diophantine approximation}\label{s11}

We give an example where the known method fails to give a complete solution, so one has to go back to explore its relation to limsup sets defined by rectangles.

Before that, we give some words about Diophantine approximation on Cantor set which is initiated from a problem asked by K. Mahler \cite{Mah}: how well the points on triadic Cantor set can be approximated by rational numbers. This problem is still far away from being solved completely. See Levesley, Salp \& Velani \cite{LeSV}, Bugeaud \cite{Bu2} and Seuret \& Wang \cite{SeW} for partial results and see Bugeaud \& Durand \cite{BuD} for a random version.
The following can be viewed as a multiplicative version of the setting after \cite{LeSV}.

Let $a,b\ge 3$ be two integers. Use the notation from Subsection \ref{s5.1.3}.
The exponents of the Cantor measure supported on $\C_a$ and $\C_b$
are denoted by $\delta_1$ and $\delta_2$. Assume that $\delta_1\ge \delta_2$.

For any $(x_o,y_o)\in \mathcal{C}_a\times \mathcal{C}_b$, define$$M_c(\psi):=\Big\{(x,y)\in \mathcal{C}_a\times \mathcal{C}_b: \|a^nx-x_o\|\cdot \|b^ny-y_o\|<\psi(n), \ {\text{i.m.}}\ n\in \N\Big\}.
$$
We consider the dimension of $M_c(t)$ at first, i.e. the set when $\psi(n)=e^{-nt}$ for some $t\ge 0$.

One can also present a similar covering lemma  for the setting of Cantor space as did by Bovey \& Dodson.\begin{lemma}\label{ll1}
Assume $\delta_1\ge \delta_2$. The set $$
\Big\{(x,y)\in \C_a\times C_b: xy<\rho\Big\}
$$ has a covering of $2$-dimensional hypercubes $\mathscr{C}=\{C_k\}_{k\geq 1}$ whose $s$-volume satisfies
\begin{equation*}
\sum\limits_{C_i\in \mathscr{C_i}}|C_i|^s\ll\rho^{s-\delta_1}, \ {\text{for any}}\ \delta_1\le s\le \delta_1+\delta_2.\end{equation*}
\end{lemma}
\begin{proof}The proof is almost identical to that given by Bovey \& Dodson, so we omit it.
\end{proof}
Now we turn back to the dimension of $M_c(t)$. If the covering lemma and the slicing lemma work well as in the classic multiplicative Diophantine approximation \cite{BV15,BD}, one should have \begin{equation}\label{fff3}
\hdim M_c(t)=\max\Big\{\delta_1+\frac{\delta_2\log b}{t+\log b}, \ \delta_2+\frac{\delta_1\log a}{t+\log a}\Big\}.\end{equation}
In fact, we have the following result.\begin{theorem}\label{t4}
Assume $\delta_1\ge \delta_2$. \begin{itemize}
\item When $a\le b$, the formula (\ref{fff3}) holds.
\item When $a>b$,
\begin{itemize}
\item when $\log a \le t+\log b$,  the formula (\ref{fff3}) holds.

\item when $\log a>t+\log b$, \begin{itemize}
\item when $\delta_2(t+\log a)\ge \delta_1\log a$, the formula (\ref{fff3}) holds.
\item when $\delta_2(t+\log a)< \delta_1\log a$, the formula (\ref{fff3}) does not hold.
\end{itemize}
\end{itemize}
\end{itemize}
\end{theorem}

Before the proof, we give some notation.
For any $t_1,t_2\ge 0$, define
$$
M_c(t_1,t_2)=\Big\{(x,y)\in \C_a\times \C_b: \|a^nx-x_o\|<e^{-t_1n}, \ \|b^ny-y_o\|<e^{-t_2n}, \ {\text{i.m.}}\ \ n\in \N\Big\}.
$$
Proposition \ref{p4} shows the ubiquity property with the resonant sets being $$
\Big\{B(\frac{u_1}{a}+\cdots+\frac{u_n+x_o}{a^n}, e^{-n\log a})\times B(\frac{v_1}{b}+\cdots+\frac{v_n+y_o}{b^n}, e^{-n\log b}): u\in \Lambda_a^n, v\in \Lambda_b^n, n\ge 1\Big\},
$$

It is clear that the resonant sets satisfy the separation condition of (\ref{ff1}). So, Corollary \ref{c1} and Theorem \ref{t1} together give the exact dimension of $M_c(t_1,t_2)$. Note the alphabet $\mathcal{A}$ is $$
\mathcal{A}=\{\log a, \log b, \log a +t_1, \log b+t_2\}.
$$
\begin{corollary}\label{c2}\begin{itemize}
\item When $t_1+\log a\ge \log a \ge t_2+\log b\ge \log b$, one has $$
\hdim M_c(t_1,t_2)=\min\{\delta_1+\delta_2-\frac{t_2\delta_2}{t_2+\log b}, \ \delta_1+\delta_2-\frac{\delta_1t_1+\delta_2t_2}{t_1+\log a}\}.
$$

\item When $t_1+\log a \ge t_2+\log b\ge \log a \ge \log b$, one has
$$
\hdim M_c(t_1,t_2)=\min\{\frac{\delta_1\log a +\delta_2\log b}{t_2+\log b}, \ \delta_1+\delta_2-\frac{\delta_1t_1+\delta_2t_2}{t_1+\log a}\}.
$$

\item When $t_2+\log b\ge t_1+\log a \ge \log a \ge \log b$,
$$\hdim M_c(t_1,t_2)=\min\{\frac{\delta_1\log a +\delta_2\log b}{t_1+\log a}, \ \delta_1+\delta_2-\frac{\delta_1t_1+\delta_2t_2}{t_2+\log b}\}.
$$
So, the dimension of $M_c(t_1,t_2)$ is continuous with respect to $t_1, t_2$.
\end{itemize}
\end{corollary}

By a similar decomposition as (\ref{fff4}) and the continuity of the dimension of $M_c(t_1,t_2)$, it follows that \begin{align*}
\hdim M_c(t)&=\sup_{(t_1,t_2): t_1\ge 0, t_2\ge 0, t_1+t_2=t}\hdim M_c(t_1,t_2)=\sup_{0\le t_2\le t}\hdim M_c(t-t_2, t_2)\\
&:=\sup_{0\le t_2\le t}f(t_2).
\end{align*}

{\em Proof of Theorem \ref{t4}}. (1). When $a\le b$, the covering lemma (Lemma \ref{ll1}) and the slicing lemma (Lemma \ref{l7}) work well to get the exact dimension of $M_c(t)$.

However when $a>b$, the covering lemma can only give the following upper bound $$
\hdim M_c(t)\le \delta_1+\frac{\delta_2\log a}{t+\log a}.
$$ On the other hand, using the slicing lemma, one will get $$
\hdim M_c(t)\ge \max\{\delta_1+\frac{\delta_2\log b}{t+\log b}, \delta_2+\frac{\delta_1\log a}{t+\log a}\}.
$$

The upper bound and the lower bound do not coincide, so at least one is not the exact dimension of $M_c(t)$. Thus, we have to use Theorem \ref{t1} to give the exact dimension of $M_c(t)$ in other cases. At the remaining cases,
it can happen that neither of the above dimensions given by covering lemma and slicing lemma is the exact dimension of $M_c(t)$.

Look at the dimensional numbers in Corollary \ref{c2}. Assume $t_1=t-t_2$ and view all the terms as functions of $t_2$. It is easy to see that for the first terms in the three dimensional formulas, the first two are decreasing and the third one is increasing.

For the second term in the last case, $$
\delta_1+\delta_2-\frac{t_1\delta_1+t_2\delta_2}{t_2+\log b}=2\delta_1-\frac{\delta_1t+(\delta_1-\delta_2)\log b}{t_2+\log b}
$$ which is increasing with respect to $t_2$.
For the second term in the first two cases,
$$
\delta_1+\delta_2-\frac{t_1\delta_1+t_2\delta_2}{t_1+\log a}=2\delta_2-\frac{\delta_2(t+\log a)-\delta_1\log a}{t+\log a-t_2},
$$ so its monotonicity  depends on
 the relation between $\delta_2(t+\log a)$ and $\delta_1\log a$.
Thus if we further distinguish the cases $$
\delta_2(t+\log a)\ge \delta_1\log a, \ \ \ {\text{and}} \ \ \ \delta_2(t+\log a)< \delta_1\log a,
$$ the monotonicity of all the dimensional numbers involved are clear.

We only go into the details of the last case and state it as a proposition. For other cases, we always get formula (\ref{fff3}).

\begin{proposition}
Assume $\delta_1\ge \delta_2$, $a>b$ and \begin{equation}\label{g5}
\log a> t+\log b, \ {\text{and}}\ \
\delta_2(t+\log a)< \delta_1\log a.
\end{equation} Then there exists a unique $\widehat{t}_2\in (0,t)$ satisfying (with $t_1=t-t_2$), \begin{equation}\label{g6}
\frac{\delta_1\log a +\delta_2\log b}{t_1+\log a}=\delta_1+\delta_2-\frac{\delta_1t_1+\delta_2t_2}{t_2+\log b}.
\end{equation}Moreover one has $$
\hdim M_c(t)=f(\widehat{t_2})
$$ and $$
\delta_1+\frac{\delta_2\log a}{t+\log a}>\hdim M_c(t)>\max\Big\{\delta_1+\frac{\delta_2\log b}{t+\log b}, \ \delta_2+\frac{\delta_1\log a}{t+\log a}\Big\}.
$$
\end{proposition}
\begin{proof}By the first inequality in (\ref{g5}), only case (i) in Corollary \ref{c2} happens; and by the second inequality in (\ref{g5}),
the second dimensional number in case (i) is increasing while the first one is always decreasing with respect to $t_2$.

As functions of $t_2$, the two dimensional numbers really intersect at some $t_2\in [0,t]$ by comparing their values
 at $t_2=0$ and $t_2=t$.
So the supremum of $f(t_2)$ occurs at the unique point $\widehat{t}_2$ of $t_2$ where the two dimensional number coincide, i.e. the equation (\ref{g6}).
Thus $$
\hdim M_c(t)=f(\widehat{t_2}).
$$

Directly calculation shows that the equality (\ref{g6}) cannot happen at $t_2=0$ and $t_2=t$.  Thus $$
\hdim M_c(t)>\max\{f(0), f(t)\}=\max\Big\{\delta_1+\frac{\delta_2\log b}{t+\log b}, \ \delta_2+\frac{\delta_1\log a}{t+\log a}\Big\}.
$$

At the same time, if we neglect the first dimensional number in case (i), one has $$
\hdim M_c(t)\le \sup_{0\le t_2\le t}\left(\delta_1+\delta_2-\frac{(t-t_2)\delta_1+t_2\delta_2}{t-t_2+\log a}\right)=\delta_1+\delta_2-\frac{t\delta_2}{\log a},
$$ since the function involved is increasing with respect to $t_2$ by the second inequality (\ref{g5}).
It is clear that \begin{align*}
\delta_1+\delta_2-\frac{t\delta_2}{\log a}<\delta_1+\delta_2-\frac{t\delta_2}{t+\log a}=\delta_1+\frac{\delta_2\log a}{t+\log a}.
\end{align*}
\end{proof}
Generally, one has\begin{theorem}
Let $\psi:\mathbb{R}^+\to \mathbb{R}^+$ be non-increasing. Then $$
\hdim M_c(\psi)=\hdim M_c(t), \ \ {\text{where}}\ t=\liminf_{n\to\infty}\frac{-\log \psi(n)}{n}.
$$
\end{theorem}

\end{document}